\newif\ifpersonal
\personaltrue % comment to remove personal notes
\documentclass[10pt,a4paper]{amsart} %reqno places equation numbers on the right
\usepackage{amsmath,amsthm,amssymb,mathrsfs,mathtools,tensor,eucal} % math related
\usepackage[all,cmtip]{xy}
\usepackage[LGR,T1]{fontenc} % LGR allows for Greek text input
\usepackage[utf8]{inputenc} 

\usepackage{xr}
\usepackage{microtype,inconsolata} % latex technical issues
\usepackage{enumerate,comment,braket,xspace,tikz,tikz-cd,csquotes,tabularray} % utilities
\usetikzlibrary{shapes.geometric}
\usetikzlibrary{arrows.meta,decorations.markings,calc}
\usepackage[centering,vscale=0.7,hscale=0.7]{geometry}
\usepackage{thmtools}
\usepackage[usenames,dvipsnames]{xcolor}

\usepackage{tikzarrows}
\usepackage[backref=page]{hyperref}
\hypersetup{colorlinks=true,citecolor=NavyBlue,linkcolor=Brown,urlcolor=Orange}

\usepackage[capitalize]{cleveref}

\hypersetup{hypertexnames=false}

\tikzcdset{
	cells={font=\everymath\expandafter{\the\everymath\displaystyle}},
}

\makeatletter
\newcommand{\ostar}{\mathbin{\mathpalette\make@circled\star}}
\newcommand{\make@circled}[2]{%
	\ooalign{$\m@th#1\smallbigcirc{#1}$\cr\hidewidth$\m@th#1#2$\hidewidth\cr}%
}
\newcommand{\smallbigcirc}[1]{%
	\vcenter{\hbox{\scalebox{0.77778}{$\m@th#1\bigcirc$}}}%
}
\makeatother

\newcommand{\caldararu}{{C\u{a}ld\u{a}raru}}
\newcommand{\inertia}{\mathsf I}
\newcommand{\Def}{\mathsf{def}}
\newcommand{\loopstack}{\mathscr{L}}

\newcommand{\Unfoldedloopstackfilt}{\mathsf{L}_{\mathsf{filt}}}
\newcommand{\dSt}{\mathsf{dSt}}
\newcommand{\hatZ}{\widehat{\mathbb Z}}
\newcommand{\hatcircle}{\widehat{S}^1}
\newcommand{\hatcircletame}{\widehat{S}^1_{\mathsf{tame}}}
\newcommand{\BCr}{\mathsf{BC}_r}
\newcommand{\BGa}{\mathsf{B}\bbG_{a}}
\newcommand{\Free}{\mathsf{Free}_{\mathsf{Lie}}}

\newcommand{\inertiaDM}{\inertia^{\mathsf{DM}}}
\newcommand{\inertiaDMtame}{\inertia^{\mathsf{DM}}_{\mathsf{tame}}}
\newcommand{\inertiarth}{\inertia^{(r)}}

\newcommand{\HC}{\mathsf{HC}}
\newcommand{\HH}{\mathsf{HH}}
\newcommand{\HP}{\mathsf{HP}}
\newcommand{\HN}{\mathsf{HN}}
\newcommand{\DR}{\mathsf{DR}}
\newcommand{\sqz}{\mathsf{sqz}}

\newcommand{\circlefiltth}[1]{S^{1,(#1)}_{\mathsf{filt}}} % Folded: After quotient by C_n and Gm. 
\newcommand{\tildecirclefiltth}[1]{\widetilde{S}^{1,(#1)}_{\mathsf{filt}}} % Simpson circle with C_n action 
\newcommand{\circlefilt}{S^{1}_{\mathsf{filt}}} % Simpson filtered circle

\newcommand{\tildecirclefilt}{\widetilde{S}^{1}_{\mathsf{filt}}} % Simpson circle - Before quotient by Gm
\newcommand{\FoldedFiltCircle}{\mathscr{S}^{1}_{\mathsf{filt}}} % Folded filtered circle
\newcommand{\FoldedFiltLoop}{\mathscr{L}_{\mathsf{filt}}} 
\newcommand{\tensoraction}[2]{\mathop{\otimes}\limits_{#1/\!/#2}}
\newcommand{\dStDefoverunder}[2]{\dSt^{\Def}_{#1\overunder#2}}

\newcommand{\PartitionLie}{\categ{Lie}^{\pi}}
\newcommand{\FMP}{\categ{FMP}}

\newcommand{\overunder}{/\!/}

\newcommand{\pts}{\mathsf{pts}}

\newcommand{\IndCoh}{\mathsf{IndCoh}}

\newcommand{\aff}{\mathsf{aff}}
\newcommand{\hataff}{\widehat{\mathsf{aff}}}

\newcommand{\tame}{\mathsf{tame}}

\newcommand{\piDM}{\pi^{\mathsf{DM}}}

\newcommand{\QCoh}{\mathsf{QCoh}}

\numberwithin{equation}{subsection}
\theoremstyle{plain}
\newtheorem{thm}{Theorem}[section]
\newtheorem{lem}[thm]{Lemma}
\newtheorem{prop}[thm]{Proposition}

\newtheorem{cor}[thm]{Corollary}

\theoremstyle{definition}

\newtheorem{defin}[thm]{Definition}
\newtheorem{notation}[thm]{Notation}
\newtheorem{eg}[thm]{Example}
\newtheorem{rem}[thm]{Remark}
\newtheorem{variant}[thm]{Variant}

\newtheorem{recollection}[thm]{Recollection}
\newtheorem{warning}[thm]{Warning}
\newtheorem{construction}[thm]{Construction}

\newcommand{\todo}[1]{\textcolor{red}{(Todo: #1)}}

\ifpersonal
\newcommand{\personal}[1]{\textcolor[rgb]{0,0,1}{(Personal: #1)}}
\newcommand{\discussion}[1]{\textcolor{violet}{(Discussion: #1)}}
\else
\newcommand{\personal}[1]{\ignorespaces}
\newcommand{\discussion}[1]{\ignorespaces}
\fi

\newcommand{\C}{\mathbb C}

\newcommand{\Q}{\mathbb Q}

\newcommand{\Z}{\mathbb Z}

\newcommand{\cA}{\mathcal A}

\newcommand{\cC}{\mathcal C}

\newcommand{\cE}{\mathcal E}
\newcommand{\cF}{\mathcal F}

\newcommand{\cG}{\mathcal G}

\newcommand{\cM}{\mathcal M}

\newcommand{\cO}{\mathcal O}

\newcommand{\cX}{\mathcal X}
\newcommand{\cY}{\mathcal Y}

\DeclareFontFamily{U}{BOONDOX-calo}{\skewchar\font=45 }
\DeclareFontShape{U}{BOONDOX-calo}{m}{n}{<-> s*[1.05] BOONDOX-r-calo}{}
\DeclareFontShape{U}{BOONDOX-calo}{b}{n}{<-> s*[1.05] BOONDOX-b-calo}{}
\DeclareMathAlphabet{\mathcalboondox}{U}{BOONDOX-calo}{m}{n}
\newcommand{\bbA}{\mathbb A}
\newcommand{\bbD}{\mathbb D}

\newcommand{\bbG}{\mathbb G}

\newcommand{\bbL}{\mathbb L}

\makeatletter
\let\save@mathaccent\mathaccent
\newcommand*\if@single[3]{%
	\setbox0\hbox{${\mathaccent"0362{#1}}^H$}%
	\setbox2\hbox{${\mathaccent"0362{\kern0pt#1}}^H$}%
	\ifdim\ht0=\ht2 #3\else #2\fi
}
\newcommand*\rel@kern[1]{\kern#1\dimexpr\macc@kerna}
\newcommand*\widebar[1]{\@ifnextchar^{{\wide@bar{#1}{0}}}{\wide@bar{#1}{1}}}
\newcommand*\wide@bar[2]{\if@single{#1}{\wide@bar@{#1}{#2}{1}}{\wide@bar@{#1}{#2}{2}}}
\newcommand*\wide@bar@[3]{%
	\begingroup
	\def\mathaccent##1##2{%
		\let\mathaccent\save@mathaccent
		\if#32 \let\macc@nucleus\first@char \fi
		\setbox\z@\hbox{$\macc@style{\macc@nucleus}_{}$}%
		\setbox\tw@\hbox{$\macc@style{\macc@nucleus}{}_{}$}%
		\dimen@\wd\tw@
		\advance\dimen@-\wd\z@
		\divide\dimen@ 3
		\@tempdima\wd\tw@
		\advance\@tempdima-\scriptspace
		\divide\@tempdima 10
		\advance\dimen@-\@tempdima
		\ifdim\dimen@>\z@ \dimen@0pt\fi
		\rel@kern{0.6}\kern-\dimen@
		\if#31
		\overline{\rel@kern{-0.6}\kern\dimen@\macc@nucleus\rel@kern{0.4}\kern\dimen@}%
		\advance\dimen@0.4\dimexpr\macc@kerna
		\let\final@kern#2%
		\ifdim\dimen@<\z@ \let\final@kern1\fi
		\if\final@kern1 \kern-\dimen@\fi
		\else
		\overline{\rel@kern{-0.6}\kern\dimen@#1}%
		\fi
	}%
	\macc@depth\@ne
	\let\math@bgroup\@empty \let\math@egroup\macc@set@skewchar
	\mathsurround\z@ \frozen@everymath{\mathgroup\macc@group\relax}%
	\macc@set@skewchar\relax
	\let\mathaccentV\macc@nested@a
	\if#31
	\macc@nested@a\relax111{#1}%
	\else
	\def\gobble@till@marker##1\endmarker{}%
	\futurelet\first@char\gobble@till@marker#1\endmarker
	\ifcat\noexpand\first@char A\else
	\def\first@char{}%
	\fi
	\macc@nested@a\relax111{\first@char}%
	\fi
	\endgroup
}
\makeatother

\newcommand{\Cohb}{\mathsf{Coh}^\mathsf{b}}

\newcommand{\St}{\mathsf{St}}

\newcommand{\Aff}{\cA\mathsf{ff}}

\newcommand{\trunc}{\mathrm{t}_0}

\newcommand{\CAlg}{\mathrm{CAlg}}
\newcommand{\dCAlg}{\mathsf{dCAlg}}

\newcommand{\dAff}{\mathsf{dAff}}

\newcommand{\bfMap}{\mathbf{Map}}

\newcommand{\Perf}{\mathrm{Perf}}

\newcommand{\dR}{\mathrm{dR}}
\newcommand{\ddR}{d_{\mathrm{dR}}}
\DeclareMathOperator{\Fil}{Fil}

\newcommand{\PrLR}{\categ{Pr}^{\kern0.05em\operatorname{L}, \operatorname{R}}}

\newcommand{\Gm}{\mathbb G_{m}}
\newcommand{\hatGm}{\widehat{\bbG}_{m}}

\newcommand{\hatGa}{\widehat{\bbG}_{a}}

\newcommand{\Lie}{\mathsf{Lie}}

\newcommand{\et}{_\mathrm{\acute{e}t}}
\newcommand{\ev}{\mathrm{ev}}

\newcommand{\id}{\mathrm{id}}

\newcommand{\op}{^\mathrm{op}}

\newcommand{\DM}{Deligne--Mumford\xspace}

\usetikzlibrary{decorations.markings} %arrows for open immersions and closed immersions
\tikzset{
  closed/.style = {decoration = {markings, mark = at position 0.5 with { \node[transform shape, xscale = .8, yscale=.4] {/}; } }, postaction = {decorate} },
  open/.style = {decoration = {markings, mark = at position 0.5 with { \node[transform shape, scale = .7] {$\circ$}; } }, postaction = {decorate} }
}

\DeclareMathOperator{\cofib}{cofib}

\DeclareMathOperator{\Fun}{Fun}

\DeclareMathOperator{\Hom}{Hom}

\DeclareMathOperator{\Map}{Map}

\DeclareMathOperator{\Spec}{Spec}

\DeclareMathOperator{\Sym}{Sym}

\DeclareMathOperator*{\colim}{colim}

\newcommand{\categ}[1]{\textsf{\textup{#1}}}

\newcommand{\Spc}{\categ{Spc}}
\newcommand{\Sp}{\categ{Sp}}

\newcommand{\PrR}{\categ{Pr}^{\kern0.05em\operatorname{R}}}
\newcommand{\PrL}{\categ{Pr}^{\kern0.05em\operatorname{L}}}
\newcommand{\PrLomega}{\categ{Pr}^{\kern0.05em\operatorname{L},\omega}}
\newcommand{\PrLkappa}{\categ{Pr}^{\kern0.05em\operatorname{L},\kappa}}
\newcommand{\PrLRomega}{\categ{Pr}^{\kern0.05em\operatorname{L}, \operatorname{R},\omega}}
\newcommand{\PrLotimes}{\categ{Pr}^{\kern0.05em\operatorname{L},\otimes}}
\newcommand{\PrLat}{\categ{Pr}^{\kern0.05em\operatorname{L}, \operatorname{at}}}
\newcommand{\Cat}{\categ{Cat}}

\newcommand{\Mod}{\mathsf{Mod}}

\DeclareMathOperator*{\flim}{``lim''}

\newcommand{\Ind}{\mathrm{Ind}}
\newcommand{\Pro}{\mathrm{Pro}}

\begin{document}

\title{Hochschild--Kostant--Rosenberg isomorphism for derived Deligne--Mumford stacks}

\author{Lie Fu}
\address{Lie Fu, Institut de Recherche Mathématique Avancée, 7 Rue René Descartes, 67000 Strasbourg, France}
\email{lie.fu@math.unistra.fr}

\author{Mauro Porta}
\address{Mauro Porta, Institut de Recherche Mathématique Avancée and Institut Universitaire de France (IUF), 7 Rue René Descartes, 67000 Strasbourg, France}
\email{porta@math.unistra.fr}

\author{Sarah Scherotzke}
\address{Sarah Scherotzke, Department of Mathematics, Universite du Luxembourg, 6 Av. de la Fonte, 4364 Esch-sur-Alzette, Luxembourg}
\email{sarah.scherotzke@uni.lu}

\author{Nicol\`o Sibilla}
\address{Nicol\`o Sibilla, SISSA, Via Bonomea 265, 34136 Trieste TS, Italy}
\email{nsibilla@sissa.it}

\subjclass[2020]{}

\begin{abstract}
    We prove a Hochschild--Kostant--Rosenberg (HKR) theorem for arbitrary derived Deligne--Mumford (DM) stacks, extending the results of Arinkin--\caldararu--Hablicsek beyond the smooth, global quotient case, although with different methods and confirming a conjecture of Okawa and Sano.
    Our solution requires the introduction of the \emph{orbifold inertia} stack of a derived DM stack, a subtle derived enhancement of the classical inertia stack.
    We equally provide a formulation that holds in positive and mixed characteristic, and prove that in characteristic zero it generalizes all previously known results.
    In characteristic zero, our result yields in particular a multiplicative HKR isomorphism between Hochschild homology and the cohomology of differential forms on the orbifold inertia, intertwining the circle action and the de Rham differential. As a consequence, we also obtain HKR isomorphisms for Hochschild cohomology, cyclic homology, negative cyclic homology and periodic cyclic homology. We equally provide a rich supply of computations of Hochschild homology and Hochschild cohomology of interesting derived DM stacks, for instance, weighted projective lines, root stacks, quotients by (not necessarily finite) algebraic groups and mapping stacks.
\end{abstract}

\maketitle

\tableofcontents

\section{Introduction}

The Hochschild--Kostant--Rosenberg (HKR) theorem is a fundamental theorem that, in its most basic form, allows to interpret differential forms on a smooth variety in terms of the self-intersection of its diagonal, expressed in terms of \emph{Hochschild homology groups}.
This theorem has a long history, and it has been revisited from many different perspectives.
We provide a short review of this rich history in \S\ref{subsec:history}, but we can make precise the focus of this paper from the very beginning.
In the setting of algebraic geometry, two extreme cases have been studied extensively: that of schemes and algebraic spaces, and that of Artin stacks. In the former, our present understanding and formulation of the HKR theorem is extremely satisfactory; in the latter, a weak form has been obtained by Ben-Zvi Nadler, and at the moment  to the best of our knowledge there is no evidence for a better statement to be possible.

In between schemes and general Artin stacks  sit \DM stacks, for which the situation is more particular. 
In 2014, two papers were posted a couple of days from each other \cite{Okawa-Sano-NCRigidity,Arinkin_Caldararu_Hablicsek}.
In \cite{Okawa-Sano-NCRigidity} a conjectural formulation for the HKR for Hochschild \emph{co}homology of smooth \DM stacks appeared (see Question 2.23 and (2.27) in \emph{loc.\ cit.}). On the other hand, \cite{Arinkin_Caldararu_Hablicsek} proved the very same conjecture in the smooth global quotient case, which is nevertheless not sufficient for the main application of \cite{Okawa-Sano-NCRigidity}.
The aim of this paper is to extend the result of \cite{Arinkin_Caldararu_Hablicsek} in four different directions:
\begin{enumerate}\itemsep=0.2cm
    \item we allow arbitrarily singular and even derived \DM stacks;

    \item we go beyond the global quotient case;

    \item we establish the compatibility between the natural circle action and the de Rham differential operator;

    \item we go beyond the characteristic zero case.
\end{enumerate}
In particular, the original conjecture of \cite{Okawa-Sano-NCRigidity} is settled in its full generality (see Corollary~\ref{cor:HochschildCohomology}).

\subsection{Historical context}\label{subsec:history}

The Hochschild homology can be defined for any associative algebra $A$, as the derived tensor product of $A$ with itself as an $(A, A)$-bimodule:
\[ \HH_*(A) \simeq  A \otimes_{A \otimes A^{\mathrm{op}}}^\mathrm{L} A\ . \]
In their seminal 1962 paper \cite{HKR}, Hochschild, Kostant and Rosenberg proved that, 
the Hochschild homology of a \emph{smooth commutative} algebra $A$ over a field of characteristic $0$ is equivalent to its graded  algebra of differential forms:
\begin{equation*}
    \HH_*(A) \simeq \bigoplus_i\Omega^i_{A}\ .
\end{equation*}
The HKR  theorem shows in particular that the algebra of differential forms, a subtle geometric invariant, can be computed in terms of a general construction in homological algebra which makes sense even outside the commutative world.
From this perspective, the HKR theorem is one of the cornerstones of non-commutative geometry, that attracted many further investigations \cite{connes1985non,tsygan1983homology,Kaledin_Degeneration,Mathew_Kaledin_Degeneration}.

\medskip

Remaining in the commutative setting, the HKR theorem can be extended to more global situations.
For instance,  in \cite{GellerWeibel-EtaleDescentHH} the definition of the Hochschild homology is extended to all schemes and in \cite[Corollary 2.6]{Swan-HH} Swan proved a global analogue of the HKR theorem, always in characteristic zero:
\begin{equation}
\label{globalhkr}
\HH_{- i }(X) \simeq \bigoplus_{q-p= i }  H^q(X, \Omega^p_X) \ .
\end{equation}
The smoothness assumption can be dropped, at the price of working with the derived exterior powers of the cotangent complex, as first shown in the work of André \cite{Andre_Homologie} and Quillen \cite{Quillen_Homology_of_commutative_rings}.
The advent of derived algebraic geometry \cite{HAG-II} allowed to recast these results in a more geometric language.
If $X$ is a (possibly derived) scheme, we can associate to it two derived schemes
\[ \loopstack X\coloneqq\bfMap(S^1, X) \simeq X \times_{X \times X} X \qquad \text{and} \qquad \mathsf T[-1]X \coloneqq \Spec_X\left(\Sym_{\mathcal{O}_X}\mathbb{L}_X[1]\right) \ , \]
which we refer to as the \emph{free loop space} and the \emph{shifted tangent bundle of $X$}.
Since derived fiber products are computed by derived tensor products, the LHS of \eqref{globalhkr} is recovered by taking derived functions on $\loopstack X$, while the RHS of \eqref{globalhkr} corresponds to derived functions on $\mathsf T[-1]X$.
The HKR theorem can therefore be strengthened to an equivalence of derived schemes over $X$:
\begin{equation}\label{exp}
    \begin{tikzcd}
        \mathsf T[-1]X \arrow{rr}{\exp}[swap]{\sim} \arrow{dr} & & \loopstack X \arrow{dl} \\
        & X
    \end{tikzcd} \ ,
\end{equation}
a formulation originally due to Ben--Zvi and Nadler \cite{BenZvi_Nadler_Loop_spaces_and_connections}, and later improved upon by Toën and Vezzosi \cite{Toen_Vezzosi_S1_algebras} by studying the compatibility between the de Rham differential and the Connes' operator in this geometric context.

\medskip

Further beyond (derived) schemes are (derived) \DM and Artin stacks. Although the map \eqref{exp} can be defined for Artin stacks, it was already observed in \cite{BenZvi_Nadler_Loop_spaces_and_connections} that, even though it restricts to an equivalence between formal completions (at the zero section and at the constant loops, respectively), it cannot be a global equivalence beyond the algebraic space case.
The failure of \eqref{exp} to be an equivalence is already transparent in the \DM case, where the right hand side decomposes into connected components, known as \emph{twisted sectors}, whereas the left hand side only knows information on the principal sector.
In \cite{Okawa-Sano-NCRigidity}, a general conjecture about HKR-type isomorphism for the Hochschild \emph{co}homology for general smooth \DM stacks appeared. Almost simultaneously, in \cite{Arinkin_Caldararu_Hablicsek} Arinkin--\caldararu--Hablicsek proved a geometric version of this conjecture, but restricted to the smooth \emph{global quotient} case. When $X = [Y/G]$, with $Y$ a smooth scheme equipped with an action of a finite group $G$, the result of \cite{Arinkin_Caldararu_Hablicsek} can be stated as an equivalence
\begin{equation}\label{looptwistedsectors}
    \exp \colon \mathsf T[-1]\inertia X \stackrel{\sim}{\longrightarrow} \loopstack X \ ,
\end{equation}
of derived \DM stacks over $X$.
Passing to derived functions, one recovers the following generalization of the HKR theorem for schemes:
\begin{equation}\label{globalhkr3}
    \HH_{-i}([Y/G]) \simeq  \bigoplus_{g\in G/G} \bigoplus_{q-p=i}H^q(Y^g, \Omega_{Y^g}^p)^{Z(g)}\ ,
\end{equation}
Despite being stated as in \eqref{looptwistedsectors}, the construction of the map $\exp$ in \cite{Arinkin_Caldararu_Hablicsek} depends on the presentation $X \simeq [Y/G]$ itself, thus making globalization difficult.

\subsection{Orbifold inertia}

As already announced, the main goal of this paper is to generalize the equivalence \eqref{looptwistedsectors} beyond the smooth global quotient case.
One key challenge that has to be dealt with is to correctly match the derived structures in \eqref{looptwistedsectors} when $X$ is singular, and possibly derived itself. The reason is that, when $X$ is smooth, $\inertia X$ and $\loopstack X$ can both be described as
\[ X \times_{X \times X} X \ , \]
with the difference that for $\inertia X$ the fiber product is understood \emph{classically}, whereas for $\loopstack X$ the fiber product is understood in the \emph{derived sense}.
When $X$ is itself derived, there is no immediate definition of $\inertia X$.
We obviate this problem by introducing the following definition (where $\mathsf C_r$ denotes the constant cyclic group of order $r$):

\begin{defin}[Orbifold inertia]
\label{intromain1}
    Let $X$ be a derived stack defined over a commutative ring $k$ of characteristic zero.
    The \emph{orbifold inertia stack of $X$} is defined as
    \[ \inertiaDM X \coloneqq \colim_{r} \bfMap(\BCr, X) \ , \]
    where $\bfMap$ stands for the mapping stack of derived stacks, the colimit is taken in the category of derived stacks and over positive integers with respect to the divisibility relation.
\end{defin}

This definition is clearly inspired by the work of Abramovich, Graber and Vistoli \cite{Abramovich_GW_theory_for_DM_stacks}, and it arose during past discussions of the second-named author with Adeel Khan, Charanya Ravi and Tony Yue Yu.
We summarize the key features of the above definition as follows:

\begin{thm}\label{thm-intro:basics_inertia}
    Let $X$ be a derived  stack defined over a commutative ring $k$ of characteristic zero.
    Then:
    \begin{enumerate}\itemsep=0.2cm
        \item (Theorem~\ref{thm:basics_inertia}) If $X$ is an Artin stack then so is $\inertiaDM X$. Moreover, if $X$ is \DM, then the classical truncation of $\inertiaDM X$ coincides with the classical inertia of $\trunc(X)$:
        \[ \trunc(\inertiaDM X) \simeq \inertia (\trunc(X)) \simeq \trunc(\loopstack X) \ . \]

        \item (Corollary~\ref{cor:InertiaDM-Smooth}) If $X$ is smooth (and hence underived), then so is $\inertiaDM X$.

        \item (Proposition~\ref{prop:InertiaGlobalQuotient}) If $X \simeq [Y/G]$ is a derived global quotient \DM stack, with $G$ a finite group, then
        \begin{equation} \label{inertiaglobalderived}
            \inertiaDM X \simeq \bigsqcup_{g \in G/G} [Y^g/Z(g)]\simeq \left[\left(\bigsqcup_{g\in G}Y^g \right)/ G\right]\ , 
        \end{equation}
        where $Y^g$ denotes the \emph{derived} fixed locus (\cref{genuinefixed}) of the automorphism $g$ on $Y$, and $Z(g)$ is the centralizer of $g$.
        This matches the usual formula for $\inertia [Y/G]$ when $Y$ is smooth.

        \item (Example~\ref{example:FixedLociBecomeDerived}) There exist \DM stacks $X$ that are underived but \emph{singular}, such that $\inertiaDM X$ has a nontrivial derived structure. In particular, even for classical \DM stacks, $\inertiaDM X$ does not coincide with the classical inertia in general.
    \end{enumerate}
\end{thm}

We will establish the results in Theorem \ref{thm-intro:basics_inertia} in a stronger form beyond characteristic zero:

\begin{variant}[Tame orbifold inertia in mixed characteristics]
    Let $k$ be a commutative ring with at least one invertible prime.
    Write $\mathsf{char}(k)$ for the set of non-invertible primes of $k$.
    We say that a \DM stack defined over $k$ is \emph{strongly tame} if the cardinality of each stabilizer group of the classical truncation of $X$ is coprime with every $p \in \mathsf{char}(k)$.
    If $X$ is strongly tame, we define its \emph{tame orbifold inertia stack of $X$} as
    \[ \inertiaDMtame X \coloneqq \colim_{(r,\mathsf{char}(k))=1} \bfMap(\BCr,X) \ , \]
    where the colimit is taken over positive integers coprime with all the elements of $\mathsf{char}(k)$.
    Then, replacing $\inertiaDM X$ by $\inertiaDMtame X$, all the statements of Theorem~\ref{thm-intro:basics_inertia} hold for strongly tame stacks.
\end{variant}
\begin{rem}\label{rem:StrongTameness}
    Note that tameness and strong tameness coincide when $k$ is equi-characteristic (e.g. a field). On the other hand, any \DM stack $X$ that is tame and proper over a DVR is automatically strongly tame. Indeed, if the DVR is of mixed characteristic $(0,p)$, any $p$-torsion element in a stabilizer group in the generic fiber would specialize by properness to $p$-torsion element in the special fiber, thus violating the tameness condition.
\end{rem}

\subsection{Geometric HKR theorem}

In order to formulate our main result, we need to recall the concept of \emph{affinization} (see \S\ref{affinization} for a thorough review).
Given a derived stack $S$, one can form its derived global sections $\Gamma(S)$ and produce the derived stack of affinization of $S$:
\[ \Aff(S) \coloneqq \Spec^{\mathsf{nc}}(\Gamma(S)) \ , \]
where $\Spec^{\mathsf{nc}}$ is obtained by a restricted Yoneda embedding (testing the \emph{nonconnective} ring $\Gamma(S)$ against connective ones).
This procedure has been pioneered by Toën \cite{Toen_Champs_affines} and received a lot of attention recently \cite{Raksit,Toen_Revisited_schematization,mondal2026unipotent}.
Following \cite{BenZvi_Nadler_Loop_spaces_and_connections}, we are especially interested in $\Aff(S^1)$, together with the structural morphism $\aff \colon S^1 \longrightarrow \Aff(S^1)$.
Paired with the canonical morphism $S^1 \to \BCr$ induced by $\Z \to \mathsf C_r$, we obtain maps
\[ \aff_r \colon S^1 \longrightarrow \Aff(S^1) \times \BCr \ . \]

\begin{thm}[Theorem \ref{thm:HKR-Geometric} and Theorem \ref{thm:HKR_DM}]
\label{intromain2}
Let $k$ be a commutative ring. For every \emph{strongly tame} derived \DM stack locally almost of finite presentation over $k$, the maps $\aff_r$ induce a canonical equivalence of derived DM stacks over $X$:
   \begin{equation} \label{eqn-intro:HKR-all-char}
    \begin{tikzcd}
        \bfMap(\Aff(S^1), \inertiaDMtame(X)) \arrow{rr}{\hataff^\ast}[swap]{\sim} \arrow{dr} & & \loopstack X \arrow{dl} \\
        & X
    \end{tikzcd}
    \end{equation} 
    If $k$ is a commutative $\mathbb{Q}$-algebra, by the formality of $\mathsf C_{\mathsf{sing}}^\ast(S^1;k)$, the diagram \eqref{eqn-intro:HKR-all-char} gives rise to a canonical equivalence of derived DM stacks over $X$:
    \begin{equation}\label{eqn-intro:HKR-char-0}
        \begin{tikzcd}
            \mathsf T[-1] \inertiaDM X \arrow{rr}{\hataff^\ast}[swap]{\sim} \arrow[dr] & & \loopstack X \arrow[dl] \\
            & X & \ .
        \end{tikzcd}
    \end{equation}
    These equivalences are functorial in $X$.
\end{thm}

Theorem \ref{intromain2} recovers the main result of Arinkin--\caldararu--Hablicsek \cite[Theorem 1.15]{Arinkin_Caldararu_Hablicsek} when $X$ is a classical, smooth, global quotient DM stack in characteristic zero, since in that setting, our map $\hataff^\ast$ coincides with the comparison map $\mathsf{exp}$ constructed in \emph{loc.\ cit}.
Paired with Theorem~\ref{thm-intro:basics_inertia}-(4), this is saying that already in the underived but singular case, the derived structure of $\loopstack X$ is distributed in a nontrivial way between the shifted tangent and $\inertiaDM X$.

\subsection{HKR equivalences  in characteristics zero.}

Theorem~\ref{intromain2} yields its strongest consequences in characteristic zero.
In \S\S\ref{HKRDMchar0}--\ref{sec:GroupStructure} we focus on this case.

\begin{thm}[HKR for Hochschild homology, Theorem~\ref{thm:HKR-HH} and Corollary~\ref{corS^1=dR}]
\label{intromain3}
    Let $X$ be a derived Deligne--Mumford stack locally almost of finite presentation over a commutative $\mathbb{Q}$-algebra $k$. Then Theorem~\ref{intromain2} paired with the canonical equivalences
    \[ \Gamma(\loopstack X, \cO) \simeq \HH_{*}(X) \qquad \text{and} \qquad \Gamma(\mathsf T[-1]\inertiaDM X, \cO) \simeq \Gamma(\inertiaDM X, \Sym_{\cO_{\inertiaDM X}}(\bbL_{\inertiaDM X}[1])) \]
    induce an equivalence of \emph{algebra objects} in the $\infty$-category $\Mod_k$:
    \begin{equation}\label{introhhkkrr}
	   \HH_{*}(X) \simeq \Gamma(\inertiaDM X, \Sym_{\cO_{\inertiaDM X}}(\bbL_{\inertiaDM X}[1]))\ ,
    \end{equation}
    where the Hochschild homology $\HH_{*}(X)$ is equipped with its natural algebra structure, and $\Gamma(\inertiaDM X, \Sym_{\cO_{\inertiaDM X}}(\bbL_{\inertiaDM X}[1]))$ is equipped with the natural algebra structure induced from the algebra structure on the symmetric algebra.
    
    Moreover, the equivalence (\ref{introhhkkrr}) intertwines the natural $S^1$-algebra structure on the LHS and the natural mixed structure on the RHS induced by the de Rham differential. 
\end{thm}

\begin{rem}\label{rempos}
    Theorem~\ref{intromain3} fails in positive characteristic.
    Indeed, already for schemes the above statement fails; we refer to \cite{antieau2021counterexamples} for an explicit counterexample.
\end{rem}
    
From \Cref{intromain2}, we also deduce an HKR theorem for Hochschild cohomology for smooth DM stacks in characteristic zero:

\begin{cor}[HKR for Hochschild cohomology, \Cref{Hochcoh}] \label{intromain4}
    Let $X$ be a smooth DM stack over a commutative $\mathbb{Q}$-algebra $k$.
    There is a canonical equivalence of objects in the $\infty$-category $\Mod_k$: 
    \[\mathsf{HH}^\ast(X) \simeq \Gamma(\mathsf T[-1] \inertiaDM X,q^!(\cO_X))\ ,\]
    where $q\colon \mathsf T[-1] \inertiaDM X\to X$ is the natural morphism.
    
    Moreover, there is an equivalence
    \begin{equation}
        \mathsf{HH}^\ast(X)\simeq \Gamma (\inertiaDM X, \Sym(\mathbb{L}_{\inertiaDM X}[1])\otimes i^*\omega_X^\vee[-\dim(X)])\ .
    \end{equation}
    where $\omega_X$ is the dualizing sheaf of $X$, and $i\colon \inertiaDM X\to X$ is the canonical morphism.
\end{cor}

The compatibility of the circle action and the de Rham differential in \Cref{intromain3} yields the following HKR theorems for cyclic homology and its variants:
 
\begin{cor}[HKR for cyclic, negative cyclic and periodic cyclic homology, \Cref{cor:HCHNHP}]
\label{intro-HCHNHP}
Let $X$ be a derived DM stack  locally almost of finite presentation  over a commutative $\mathbb{Q}$-algebra $k$.
Let $\DR$ be the derived de Rham complex of its orbifold inertia $\inertiaDM X$.
For any integer $i$, let $\DR^{\geq i}\coloneqq \Fil_{\mathsf{Hdg}}^i\DR$ and $\DR^{\leq i}\coloneqq \cofib(\Fil_{\mathsf{Hdg}}^i\DR \to \DR)$, where $\Fil^\bullet_{\mathsf{Hdg}}$ is the Hodge filtration.
There are canonical equivalences of objects in the $\infty$-category $\Mod_k$: 
\begin{align*} 
    \HC(X) & \simeq \bigoplus_{i\geq 0}  \Gamma \left(\inertiaDM X, \DR^{\leq i} [2i]\right)\ ;\\
    \HN(X) &\simeq \prod_{i\in \mathbb{Z}}\Gamma \left(\inertiaDM X,  \DR^{\geq i} [2i]\right)\simeq  \prod_{i\leq 0} H^*_{\dR}(\inertiaDM X)[2i] \times \prod_{i=1}^\infty \left(\inertiaDM X, \DR^{\geq i} [2i]\right)\ ;\\
    \HP(X) &\simeq \prod_{i\in \mathbb{Z}} \Gamma \left(\inertiaDM X, \DR[2i]\right) \simeq \prod_{i\in \mathbb{Z}} H^*_{\dR}(\inertiaDM X)[2i]\ .
\end{align*}
\end{cor}

In the terminology for differential forms in derived algebraic geometry \cite{PTVV, CPTVV, ParkYou-IntroductionShiftedSymplectic}, the spaces of differential forms appearing in $\HC$, $\HN$ and $\HP$ are the spaces of (shifted) \textit{exact} forms, \textit{closed} forms and \textit{de Rham} forms respectively.
\subsection{Applications}

In \S\ref{sec:examples} we focus on explicit computation of $\inertiaDM X$ and of Hochschild (co)homology based on the HKR equivalences just discussed (and for this reason we restrict here in characteristic zero).
We selected examples that were previously inaccessible to other existing methods currently available in the literature to showcase the effectiveness of our results.

To begin with, in Proposition~\ref{prop:InertiaGlobalQuotient} and Corollary~\ref{cor:HH-GlobalQuotientDerived} we revisit the case of global quotient stacks, removing any smoothness assumption (and allowing eventual derived structure).
Later, in \S\ref{sec:examplesbeyond}, we turn to \DM stacks that \emph{cannot} be presented as global quotients by finite group actions, such as Thurston's football and teardrop (see Example~\ref{eg:Football_Teardrop}). More generally, we compute the Hochschild (co)homology of arbitrary weighted projective lines (\cref{eg:WeightedProjLine}) and root stacks $\sqrt[n]{(X, D)}$: see \Cref{cor:HHRootStack} and \Cref{prop:HHcohRootStack}.

Although most DM stacks are not global quotients by finite group actions, many of them are global quotients by linear algebraic groups (\cite[Theorem 2.18]{Edidin-Hassett-Kresch-Vistoli}), most notably, for moduli stacks constructed via Geometric Invariant Theory. In Section \ref{subsec:QuotientByAlgGroups}, we extend the formula \eqref{globalhkr3} to derived DM stacks of the form $[Y/G]$ with $G$ a linear algebraic group acting on a derived algebraic space $Y$; see \Cref{cor:HHofQuotientByAlgGp}.

Finally, Hochschild cohomology of a stack controls its (possibly non-commutative, stacky, and derived) deformations. Therefore our main result on Hochschild cohomology \Cref{intromain4} opens the way of studying deformation theory of some interesting \DM stacks that are not global quotients. To exemplify the great potential in this direction, let us just mention one instance, namely, the moduli stacks of Deligne--Mumford stable curves. As already mentioned, Corollary~\ref{intromain4} implies in particular a formula for Hochschild cohomology conjectured by Okawa and Sano in \cite[Question 2.23 and (2.27)]{Okawa-Sano-NCRigidity}. Thanks to our work, their original result \cite[Corollary 3.21]{Okawa-Sano-NCRigidity} becomes unconditional: for any pairs of integers $g\geq 1$ and $n\geq 0$ such that $g+n\geq 5$, the moduli stack of Deligne--Mumford stable curves of genus $g$ with $n$ marked points $\overline{\mathscr{M}}_{g,n}$ satisfies
\begin{equation*}
    \mathsf{HH}^2(\overline{\mathscr{M}}_{g,n})=0.
\end{equation*}
In particular, $\overline{\mathscr{M}}_{g,n}$ does not admit non-commutative deformations,  generalizing a theorem of Hacking \cite{Hacking-MgnRigid} that $\overline{\mathscr{M}}_{g,n}$ is rigid, i.e., it does not admit (commutative) deformations.

\subsection{Related works}
The orbifold inertia stack $\inertiaDM X$ introduced Definition~\ref{intromain1} makes sense for any derived stacks.
In the recent work \cite{Kinjo_Multiplicative_dimensional_reduction}, Kinjo exploits this construction for \emph{Artin} stacks (under the name of \textit{torsion loop stack}) in order to prove a multiplicative version of the dimensional reduction theorem for BPS cohomology. 

In the global quotient case  a description of the derived loop space in terms of derived fixed loci was obtained in Theorem 2.9 of \cite{ben2023between}. This gives an alternative route to some of the consequences of our main results  (notably    Proposition \ref{prop:QuotientByAlgebraicGroup}),  with the important  qualification that  the authors in \cite{ben2023between} work over a field of characteristics $0$ and, as we mentioned,  their result only applies to global quotients, whereas our main results do not require these assumptions.

Our results in the mixed-characteristic case owe a lot to the work of Moulinos--Robalo--To\"en \cite{MRT} and the subsequent work of Toën \cite{Toen_Revisited_schematization}. The structure theorem they prove for $\Aff(S^1)$ is crucial in the positive and mixed characteristic formulation of Theorem~\ref{intromain1}.
In the up-coming work \cite{FPSS-FoldedFilteredCircle}, we will further extend the filtration of \cite{MRT} for the Hochschild homology of derived schemes in positive and mixed characteristic to the \DM setting.

\subsection*{AI disclosure}

This is a revision of a paper completed in 2025. No use of AI is involved.

\subsection*{Acknowledgments}
We are grateful to Pieter Belmans, Adeel Khan, M\'arton Hablicsek, Adeel Khan, Elsa Maneval, Tasos Moulinos, Joost Nuiten, Charanya Ravi, Marco Robalo, Tony Yue Yu, and Bertrand Toën for useful discussions on the topics of this paper.

L.F.\ is supported by the University of Strasbourg Institute for Advanced Study (USIAS), by the Institut Universitaire de France (IUF), by the Agence Nationale de la Recherche (ANR) under project DAG-Arts ANR-24-CE40-4098, and by the International Emerging Actions (IEA) project of CNRS.
M.P.\ is supported by the grant DAG-Arts ANR-24-CE40-4098 and by Institut Universitaire de France (IUF). N.S. was partly supported by PRIN grant 2022BTA242 ``Geometry of algebraic structures: moduli, invariants, deformations''. S.S. was partly supported by the OPEN grant O23/18107005 for the project DeMoReSh.

\subsection*{Conventions} Throughout the paper, $k$ is a commutative base ring.  We denote by $\mathsf{char}(k)$ the set of prime numbers that are not invertible in $k$; which is the empty set when $k$ is a $\Q$-algebra. An integer $r$ is said to be coprime to $\mathsf{char}(k)$, if it is coprime to each prime number in the set $\mathsf{char}(k)$.
All the derived \DM stacks considered in this paper are assumed to be 1-stacks. Even when we do not state it explicitly, all derived {schemes},  {algebraic spaces} and {\DM stacks} are assumed to be {locally almost of finite presentation} over $k$.

\section{Preliminaries on derived stacks}
\label{Appendix:Map}

In this section, we collect some foundational results that we could not locate in the exact form that we needed in the literature.
There is an evident overlap of intention with \cite[\S3.2]{BenZvi_Nadler_Drinfeld_centers} and \cite[Appendix A]{Halpern_Leistner_Preygel_Categorical_properness}, but the results discussed below are slightly more general.
We systematically use the theory of tensor products in $\PrL_k$, which was not yet fully available at the time \cite{BenZvi_Nadler_Drinfeld_centers,Halpern_Leistner_Preygel_Categorical_properness} were firstly written.
We refer the reader to \cite[\S4.8.1]{HA} and to \cite[\S I.4]{Porta_HDR} for background on this topic.

\subsection{Categorical finiteness properties}

We introduce the following finiteness conditions on (the $\infty$-categories of quasi-coherent complexes of) derived stacks.

\begin{defin}
	Let $F \in \dSt_k$ and write $p \colon F \to \Spec(k)$ for the structural morphism.
	\begin{enumerate}\itemsep=0.2cm
		\item The derived stack $F$ is called \emph{$\otimes$-universal} if for any derived stack $X \in \dSt_k$, the canonical morphism
		\[ \boxtimes \colon \QCoh(F) \otimes_k \QCoh(X) \longrightarrow \QCoh(F \times X) \]
		is an equivalence.
		
		\item The derived stack $F$ is called \emph{categorically quasi-compact} if the functor
		\[ p_\ast \colon \QCoh(F) \longrightarrow \Mod_k \]
		commutes with filtered colimits.
		
		\item The derived stack $F$ is called \emph{categorically perfect} if the functor
		\[ p^\ast \colon \Mod_k \longrightarrow \QCoh(F) \]
		admits a left adjoint $p_+$.
	\end{enumerate}	
\end{defin}

\begin{eg}\label{eg:tensor_universal}
	Let $F \in \dSt_k$.
	\begin{enumerate}\itemsep=0.2cm
		\item If $\QCoh(F)$ is dualizable in $\PrL_k$, then $F$ is $\otimes$-universal.
		Indeed, in this case $\QCoh(F) \otimes_k (-)$ commutes with limits, and therefore both the source and the target of $\boxtimes$ satisfy descent in $X$ and so it is enough to test that it is an equivalence when $X$ is affine.
		In this case, the claim follows from \cite[Proposition 4.13]{BenZvi_Nadler_Drinfeld_centers} (applied with $Y = \Spec(k)$).
		
		\item Let $K \in \Spc$ be an $\infty$-groupoid, seen as a constant derived stack.
		Then
		\[ \QCoh(K) \simeq \Fun(K, \Mod_k) \]
		is compactly generated, and for every $A \in \mathsf{dCAlg}_k$ we have
		\[ \QCoh(K) \otimes_k \Mod_A \simeq \Fun(K, \Mod_A) \ , \]
		as it follows from \cite[Proposition 4.8.1.17]{HA}.
		Thus, $K$ is $\otimes$-universal.
		If $K$ is in addition a compact object in $\Spc$, then the constant derived stack $K$ is categorically quasi-compact.
	\end{enumerate}
\end{eg}

\begin{lem}\label{lem:tensor_universal_categorically_perfect_change_of_base}
    Let $k \to k'$ be a morphism of commutative algebras.
    If $X \in \dSt_k$ is $\otimes$-universal and categorically perfect, then the same holds for $X_{k'} \coloneqq X \times_{\Spec(k)} \Spec(k') \in \dSt_{k'}$.
\end{lem}

\begin{proof}
    Since $\Spec(k')$ is $\otimes$-universal by Example~\ref{eg:tensor_universal}, associativity of the tensor product implies that $X_{k'}$ is again $\otimes$-universal.
    Now, functoriality of the tensor product shows that $p_+ \otimes \id_{\QCoh(k')}$ is left adjoint to $p^\ast \otimes \id_{\QCoh(k)}$, and the latter is canonically identified with the pullback along $X_{k'} \to \Spec(k')$.
\end{proof}

\begin{lem} \label{lem:categorical_properness}
	Assume that $F$ is categorically quasi-compact, that $\QCoh(F) \simeq \Ind(\Perf(F))$ and that $p_\ast \colon \QCoh(F) \to \Mod_k$ preserves perfect complexes.
	Then $F$ is categorically perfect, and moreover the left adjoint $p_+$ is characterized by the fact that for any $M \in \Perf(F)$, we have
	\[ p_+(M) \simeq (p_\ast(M^\vee))^\vee .\]
\end{lem}

\begin{proof}
	Under these assumptions, projection formula holds, that is, the natural comparison morphism
	\[ p_\ast(M) \otimes N \longrightarrow p_\ast(M \otimes p^\ast(N)) \]
	is an equivalence.
	The same computation of \cite[Proposition 7.11]{Porta_Yu_Hom} shows that defining $p_+$ by the given formula on perfect complexes and taking the left Kan extension along $\Perf(F) \to \Ind(\Perf(F))$ produces indeed a left adjoint for $p^\ast$.
	See also \cite[Proposition 6.4.5.3]{SAG}.
\end{proof}

\begin{lem}\label{lem:adjointability_I}
	Let $F \in \dSt_k$ be $\otimes$-universal and categorically quasi-compact.
	Then for every morphism $f \colon X \to Y$ in $\dSt_k$, the square
	\[ \begin{tikzcd}[column sep=35pt]
		\QCoh(Y) \arrow{r}{f^\ast} \arrow{d}{p_Y^\ast} & \QCoh(X) \arrow{d}{p_X^\ast} \\
		\QCoh(F \times Y) \arrow{r}{(\id_F \times f)^\ast} & \QCoh(F \times X)
	\end{tikzcd} \]
	is vertically right adjointable, that is, the natural Beck-Chevalley transformation
	\[ f^\ast \circ p_{Y,\ast} \longrightarrow p_{X,\ast} \circ (\id_F \times f)^\ast \]
	is an equivalence.
	If in addition $F$ is categorically perfect, then the above square is also vertically left adjointable.
\end{lem}

\begin{proof}
	Since $F$ is $\otimes$-universal, we find canonical identifications
	\[ (\id_F \times f)^\ast \simeq \id_{\QCoh(F)} \otimes f^\ast \ , \qquad p_X^\ast \simeq p^\ast \otimes \id_{\QCoh(X)} \ , \qquad q_X^\ast \simeq q^\ast \otimes \id_{\QCoh(X)} \ . \]
	Since $F$ is categorically quasi-compact, the adjunction $p^\ast \dashv p_\ast$ holds in $\PrL_k$.
	The functoriality of the tensor product in $\PrL_k$ guarantees therefore that the adjunctions
	\[ p^\ast \otimes \id_{\QCoh(X)} \dashv p_\ast \otimes \id_{\QCoh(X)} \qquad \text{and} \qquad p^\ast \otimes \id_{\QCoh(Y)} \dashv p_\ast \otimes \id_{\QCoh(Y)} \]
	hold.
	The uniqueness of the adjoint therefore provides two more identifications
	\[ p_{X,\ast} \simeq p_\ast \otimes \id_{\QCoh(X)} \ , \qquad q_{X,\ast} \simeq q_\ast \otimes \id_{\QCoh(X)} \ . \]
	At this point, the conclusion is obvious.
	The same argument shows that if there exists a left adjoint $p_+$ for $p^\ast$, then the adjunction $p_+ \otimes \id_{\QCoh(X)} \dashv p_X^\ast$ holds as well, and therefore the base change holds as well.
\end{proof}

\begin{notation}
	Let
	\[ \begin{tikzcd}[column sep=small]
		F \arrow{rr}{f} \arrow{dr}[swap]{p} & & G \arrow{dl}{q} \\
		{} & \Spec(k)
	\end{tikzcd} \]
	be a diagram in $\dSt_k$.
	The associated $\ast$-Beck-Chevalley transformation is the natural transformation
	\[ q_\ast \longrightarrow q_\ast f_\ast f^\ast \simeq p_\ast f^\ast \ . \]
	If furthermore $F$ and $G$ are categorically perfect, the associated $+$-Beck-Chevalley transformation is the natural transformation
	\[ p_+ f^\ast \longrightarrow p_+ f^\ast q^\ast q_+ \simeq p_+ p^\ast q_+ \longrightarrow q_+ \ . \]
\end{notation}

\begin{lem}\label{lem:adjointability_II}
	Let $f \colon F \to G$ be a morphism in $\dSt_k$.
	For every $X \in \dSt_k$, let
	\[ p_X \colon F \times X \longrightarrow X \ , \qquad q_X \colon G \times X \longrightarrow X \]
	be the canonical projections.
	Assume that:
	\begin{enumerate}\itemsep=0.2cm
		\item both $F$ and $G$ are $\otimes$-universal and categorically quasi-compact.
		
		\item the $\ast$-Beck-Chevalley transformation associated to $f$
		\[ q_\ast \longrightarrow p_\ast f^\ast \]
		is an equivalence.
	\end{enumerate}
	Then for every $X \in \dSt_k$, the triangle
	\[ \begin{tikzcd}[column sep=small]
		\QCoh(G \times X) \arrow{rr}{(f \times \id_X)^\ast} & & \QCoh(F \times X) \\
		& \QCoh(X) \arrow{ul}{q_X^\ast} \arrow{ur}[swap]{p_X^\ast}
	\end{tikzcd} \]
	is vertically right adjointable.
	If $F$ and $G$ are moreover categorically perfect, and $+$-Beck-Chevalley transformation associated to $f$ is an equivalence, then the above triangle is vertically left adjointable as well.
\end{lem}

\begin{proof}
	Assumption (1) provides canonical identifications
	\[ (f \times \id_X)^\ast \simeq f^\ast \otimes \id_{\QCoh(X)} \ , \qquad p_X^\ast \simeq p^\ast \otimes \id_{\QCoh(X)} \ , \qquad q_X^\ast \simeq q^\ast \otimes \id_{\QCoh(X)} \ . \]
	As in the proof of \cref{lem:adjointability_I}, we furthermore obtain identifications
	\[ p_{X,\ast} \simeq p_\ast \otimes \id_{\QCoh(X)} \ , \qquad q_{X,\ast} \simeq q_\ast \otimes \id_{\QCoh(X)} \ . \]
	Thus, the conclusion follows from (3) and the functoriality of the tensor product in $\PrL_k$.
	The proof of left adjointability follows along the same lines.
\end{proof}

\subsection{Cotangent complex of mapping stacks}

The following result is standard, and well-documented in the literature; see \cite[Proposition B.3.5]{Rozenblyum_Connections_infinitesimal_Hecke} or \cite[Proposition 1.4]{Naef_Safronov}.
Since the proof is elementary, we include it for the convenience of the reader:

\begin{prop}\label{prop:cotangent_complex_mapping_stack}
	Let $F, X \in \dSt_k$ be derived stacks.
	Assume that $F$ is $\otimes$-universal and categorically perfect.
	Assume that $X$ admits a global cotangent complex and is infinitesimally cohesive.
	Then $\bfMap(F,X)$ admits a global cotangent complex, given by the formula
	\[ \mathbb L_{\bfMap(F,X)} \simeq \pi_+ \ev^\ast(\mathbb L_X) \ , \]
	where $\ev \colon X \times \bfMap(F,X) \to X$ is the evaluation map and $\pi \colon X \times \bfMap(F,X) \to \bfMap(F,X)$ is the canonical projection.
\end{prop}

\begin{proof}
	Let $S \in \dAff_k$ be a test derived affine scheme and fix a morphism $S \to \bfMap(F,X)$ classifying a morphism $f \colon S \times F \to X$.
	Let $M \in \QCoh(S)_{\geqslant 0}$ and write $S[M]$ the split square-zero extension determined by $M$.
	By definition of mapping stack, the space of liftings 
	\[ \begin{tikzcd}
		S \arrow{d} \arrow{r} & \bfMap(F,X) \\
		S[M] \arrow[dashed]{ur}
	\end{tikzcd} \]
	is equivalent to the space of liftings 
	\[ \begin{tikzcd}
		S \times F \arrow{r}{f} \arrow{d} & X \\
		S[M] \times F \arrow[dashed]{ur} & \phantom{X} \ .
	\end{tikzcd} \]
	However, $S[M] \times F \simeq (S \times F)[p_S^\ast(M)]$, where $p_S \colon S \times F \to S$ is the canonical projection.
	Therefore, the space of such liftings is equivalent to
	\begin{equation}\label{eq:cotangent_complex_mapping_stack}
		\Map_{\QCoh(S \times F)}( f^\ast \mathbb L_X, p_S^\ast M ) \ .
	\end{equation}
	Since $p^\ast$ admits a left adjoint and $F$ is $\otimes$-universal, the same argument given in \cref{lem:adjointability_I} implies that $p_S^\ast$ also admits a left adjoint $p_{S,+}$.
	This shows that $\bfMap(F,X)$ admits a cotangent complex at the given point, which is furthermore given by the formula $p_{S,+}f^\ast(\mathbb L_X)$.
	The base change for the plus pushforward proven in \cref{lem:adjointability_I} readily implies that $\bfMap(F,X)$ has a global cotangent complex, given by the claimed formula.
\end{proof}

\begin{variant}\label{variant:cotangent_complex_mapping_stack_formally_etale_target}
	Let $X \in \dSt_k$ be a derived stack and assume that it is infinitesimally cohesive and formally étale over $k$.
	Then for any derived stack $F$, the mapping stack $\bfMap(F,X)$ admits a global cotangent complex, given by $0$.
	More generally, if $p \colon X \to Y$ is an infinitesimally cohesive and formally étale morphism in $\dSt_k$, then for any derived stack $F$, the induced morphism $\bfMap(F,X) \to \bfMap(F,Y)$ is again formally étale.
	The same argument given in \cref{prop:cotangent_complex_mapping_stack} supplies an identification of the space of derivations with \eqref{eq:cotangent_complex_mapping_stack}.
	The formal étaleness assumption now supplies $f^\ast \mathbb L_p \simeq 0$, so the conclusion is automatic.
\end{variant}

\begin{construction}\label{construction:comparison_cotangent_maps}
	Let $f \colon F \to G$ be a morphism in $\dSt_k$.
	Assume that $F$ is $\otimes$-universal and that both $F$ and $G$ are categorically perfect.
	Fix as well $X \in \dSt_k$ admitting a cotangent complex and consider the induced morphism
	\[ \phi_f \colon \bfMap(G,X) \longrightarrow \bfMap(F,X) \ , \]
	which is part of the following commutative diagram:
	\[ \begin{tikzcd}
		& \bfMap(G,X) \arrow{dd}[near start]{\phi_f} \arrow[equal]{rr} & & \bfMap(G,X) \\
		F \times \bfMap(G,X) \arrow[crossing over]{rr}[near end]{g} \arrow{dd}{\id_F \times \phi_f} \arrow{ur}{\pi_{F,G}} & & G \times \bfMap(G,X) \arrow{dd}{\ev_G} \arrow{ur}[swap]{\pi_G} \\
		& \bfMap(F,X) \\
		F \times \bfMap(F,X) \arrow{rr}{\ev_F} \arrow{ur}{\pi_F} & & X
	\end{tikzcd} \]
	where we set $g \coloneqq f \times \id_{\bfMap(G,X)}$.
	Since $F$ is $\otimes$-universal and categorically perfect, \cref{lem:adjointability_I} guarantees that the Beck-Chevalley transformation
	\[ \phi_f^\ast \circ \pi_{F,+} \longrightarrow \pi_{F,G,+} \circ (\id_F \times \phi_f)^\ast \]
	is an equivalence.
	Besides, the top square induces a second Beck-Chevalley transformation
	\begin{equation}\label{eq:Beck_Chevalley_maps}
		\mathsf{BC} \colon \pi_{F,G,+} \circ g^\ast \longrightarrow \pi_{G,+} \ ,
	\end{equation}
	which, paired with the above equivalence gives rise to a natural transformation
	\[ \phi_f^\ast \circ \pi_{F,+} \circ \ev_F^\ast \longrightarrow \pi_{G, +} \circ \ev_G^\ast \ . \]
	Evaluating on $\mathbb L_X$ gives rise to the natural comparison map
	\begin{equation}\label{eq:comparison_cotangent_maps}
		\phi_f^\ast \mathbb L_{\bfMap(F,X)} \longrightarrow \mathbb L_{\bfMap(G,X)} \ .
	\end{equation}
\end{construction}

\begin{prop}\label{prop:universal_criterion_formally_etale}
	Let $F, G \in \dSt_k$ be $\otimes$-universal and categorically perfect derived stacks.
	Write $p \colon F \to \Spec(k)$ and $q \colon G \to \Spec(k)$ for the structural morphisms.
	Let $f \colon F \to G$ be a morphism and assume that the $+$-Beck-Chevalley transformation
	\[ p_+ f^\ast \longrightarrow p_+ f^\ast q^\ast q_+ \simeq p_+ p^\ast q_+ \longrightarrow q_+ \]
	is an equivalence.
	Then for every derived stack $X \in \dSt_k$ admitting a cotangent complex, the induced transformation
	\[ \phi_f \colon \bfMap(G,X) \longrightarrow \bfMap(F,X) \]
	is formally étale, that is, it admits a cotangent complex which is moreover zero.
\end{prop}

\begin{proof}
	Note that $\mathbb L_{\phi_f}$ is canonically identified with the cofiber of the map \eqref{eq:comparison_cotangent_maps}.
	It is therefore enough to prove that \eqref{eq:comparison_cotangent_maps} is an equivalence.
	Tracing the definition given in \cref{construction:comparison_cotangent_maps}, and applying \cref{prop:cotangent_complex_mapping_stack}, we see that if the Beck-Chevalley transformation \eqref{eq:Beck_Chevalley_maps} is an equivalence, then the same goes for \eqref{eq:comparison_cotangent_maps}.
	Since $F$ and $G$ are $\otimes$-universal and categorically perfect, the conclusion follows directly from \cref{lem:adjointability_II}.
\end{proof}

\subsection{Deformation theory}\label{subsec:deformation_theory}

\begin{recollection}[Derived stacks with deformation theory]\label{recollection:deformation_theory}
    A derived stack $F \in \dSt_k$ is said to have \emph{deformation theory} if it is infinitesimally cohesive and nilcomplete.
    We refer to \cite[\S17.3]{SAG} for a recap on these notions.
    Let $\dSt_k^{\Def}$ be the full subcategory of $\dSt_k$ spanned by derived stacks with deformation theory.
    By \cite[Ch.\ 1, Corollary 7.2.6]{Gaitsgory_Rozenblyum_II}, this is equivalent to the more usual \cite[Ch.\ 1, Definition 7.1.2]{Gaitsgory_Rozenblyum_II}, which in addition requires $F$ to have a pro-cotangent complex.
    See also \cite[Definition B.6]{Brantner_Magidson_Nuiten}.
    By Lurie's representability theorem, all derived Artin stacks admit deformation theory, see also \cite[Ch.\ 1, Proposition 7.4.2.]{Gaitsgory_Rozenblyum_II}.
    By extension, we say that a morphism $f \colon X \to Y$ in $\dSt_k$ \emph{has deformation theory} if for every $S \in \dAff_k$ and every morphism $S \to Y$ one has $S \times_Y X \in \dSt^{\Def}_k$.
\end{recollection}

\begin{lem}\label{lem:stacks_with_deformation_theory}
    The inclusion functor
    \[ \dSt_k^{\Def} \hookrightarrow \dSt_k \]
    has a left adjoint, and in particular $\dSt_k^{\Def}$ is presentable.
\end{lem}

\begin{proof}
    Recall that by convention, all our stacks are locally almost of finite presentation.
    Write $\dSt_k^{\mathsf{conv}}$ for the full subcategory of $\dSt_k$ spanned by convergent derived stacks.
    Then the inclusion in question can be factored as
    \[ \dSt_k^{\Def} \hookrightarrow \dSt_k^{\mathsf{conv}} \hookrightarrow \dSt_k \ . \]
    Unraveling the definitions, we see that both $\dSt_k^{\Def}$ and $\dSt_k^{\mathsf{conv}}$ are closed under limits in $\dSt_k$.
    Moreover, the inclusion of $\dSt_k^{\mathsf{conv}}$ has a left adjoint by \cite[Ch.\ 2, \S1.4.8]{Gaitsgory_Rozenblyum_I}.
    In particular $\dSt_k^{\mathsf{conv}}$ is presentable.
    On the other hand, $\dSt_k^{\Def}$ is closed under filtered colimits in $\dSt_k^{\mathsf{conv}}$ by \cite[Ch.\ 1, Corollary 7.2.8]{Gaitsgory_Rozenblyum_II}.
    Thus, the main theorem of \cite{Ragimov_Schlank} implies that this inclusion admits a left adjoint as well, and the conclusion follows.
\end{proof}

\begin{notation}
    We let
    \[ (-)^{\Def} \colon \dSt_k \longrightarrow \dSt_k^{\Def} \]
    be the left adjoint to the inclusion whose existence is guaranteed by Lemma~\ref{lem:stacks_with_deformation_theory}.
\end{notation}

\begin{eg}\label{eg:S1_has_deformation_theory}
    Let $S^1 \in \dSt_k$ be the constant stack associated to $S^1 \in \Spc$.
    It is the étale \emph{hyper}sheafification of the constant derived prestack $P_{S^1}$ associated to $S^1$.
    However, since the étale sheafification of $P_{S^1}$ is $1$-truncated by \cite[Proposition 5.5.6.16]{HTT}, it is equally hypercomplete by \cite[Lemma 6.5.2.9]{HTT}.
    Thus, we can identify $S^1$ with the étale sheafification of $P_{S^1}$.
    Observe now that the constant prestack $P_{S^1}$ satisfies the following condition:
    \begin{itemize}
        \item[($\ast$)] if $T' \to T$ is a morphism of derived affine schemes which is an isomorphism on truncations, the induced map $P_{S^1}(T) \to P_{S^1}(T')$ is an equivalence.
    \end{itemize}
    In turn this implies that $P_{S^1}$ is both infinitesimally cohesive and nilcomplete.
    The formula for étale sheafification (see \cite[Remark 6.2.2.12 \& proof of Proposition 6.2.2.7]{HTT}), together with the topological invariance of the (derived) étale site, implies that condition ($\ast$) is stable under sheafification.
    In particular, $S^1$ still satisfies ($\ast$) and in particular it has deformation theory.
    As a consequence, we have a canonical equivalence $S^1 \simeq (S^1)^{\Def}$.
\end{eg}

\begin{lem}\label{lem:universal_property_of_deformation_completion}
    Let $Y \in \dSt_k^{\Def}$.
    \begin{enumerate}[(i)]\itemsep=0.2cm
        \item For any $S \in \dSt_k$, the mapping stack $\bfMap(S,Y)$ has deformation theory.

        \item Let $X \in \dSt_k$. Then the canonical morphism $X \to X^{\Def}$ induces a morphism
        \[ \bfMap(X^{\Def}, Y) \xlongrightarrow{\simeq} \bfMap(X, Y) \]
        which is moreover an equivalence.
    \end{enumerate}
\end{lem}

\begin{proof}
    Since being infinitesimally cohesive and nilcomplete are conditions closed under limits, for (i) we can assume without loss of generality that $S \in \dAff_k$.
    Then the statement follows from \cite[Proposition 19.1.3.1]{SAG}.
    As for (ii), fix a derived affine $S \in \dAff_k$. Then, unraveling the definitions:
    \begin{align*}
        \bfMap(X^{\Def}, Y)(S) & \coloneqq \Map_{\dSt_k}(X^{\Def} \times S, Y) \\
        & \simeq \Map_{\dSt_k}( X^{\Def}, \bfMap(S,Y)) \\
        & \simeq \Map_{\dSt_k}( X, \bfMap(S,Y)) & \text{by statement (i)} \\
        & = \bfMap(X,Y)(S) \ ,
    \end{align*}
    whence the conclusion.
\end{proof}

\section{Affinization and unipotent loops}
\label{affinization}
As explained in the introduction, the stacks $\loopstack X$ and $\mathsf T[-1]X$ are the key players in the study of the HKR theorem.
Both can be expressed as mapping stacks, respectively:
\[ \loopstack X \coloneqq \bfMap(S^1, X) \qquad \text{and} \qquad \mathsf T[-1]X \simeq \bfMap(\mathbb D_{-1}, X) \ . \]
Here $S^1$ is the constant derived stack associated to the circle $S^1$, seen as a homotopy type, whereas $\mathbb D_{-1}$ is introduced below as the suspension of $\mathbb D_0 \coloneqq \Spec(k \oplus k\varepsilon)$, with $\deg(\varepsilon) = 0$ and $\varepsilon^2=0$.
In this section, we collect some general material on mapping stacks of this form, that will be used throughout the rest of the paper.

\subsection{Shifted tangents}\label{subsec:shifted_tangents}

Consider the ordinary split square-zero extension $k[\varepsilon] \coloneqq k \oplus k \varepsilon$, with $\varepsilon^2=0$ and $\deg(\varepsilon) = 0$.
We set
\[ \bbD_0 \coloneqq \Spec(k[\varepsilon]) \ , \]
and we define $\bbD_{-1}$ as the following pushout in the category of derived stacks with deformation theory $\dSt_k^{\Def}$ (see Section \ref{subsec:deformation_theory}):
\begin{equation}\label{eq:D_-1}
    \begin{tikzcd}
	   \bbD_0 \arrow{r} \arrow{d} & \Spec(k) \arrow{d} \\
	   \Spec(k) \arrow{r} & \bbD_{-1} \ .
    \end{tikzcd}
\end{equation}

\begin{recollection}
    Let $Y$ be a derived Artin stack.
    For every integer $n \in \Z$, we set
    \[ \mathsf T[n]Y \coloneqq \Spec_Y(\Sym_{\cO_Y}(\bbL_Y[-n])) \ . \]
\end{recollection}

\begin{lem}\label{lem:shifted_tangents}
    Let $Y$ be a derived Artin stack.
    Then
    \[ \bfMap(\bbD_0, Y) \simeq \mathsf TY \qquad \text{and} \qquad \bfMap(\bbD_{-1}, Y) \simeq \mathsf T[-1]Y \ .\]
\end{lem}

\begin{proof}
    The first identification is the content of \cite[Proposition 1.4.1.6]{HAG-II}.
    As for the second one, since $Y$ has deformation theory, $\bfMap(-,Y)$ brings \eqref{eq:D_-1} to a pullback square, thanks to Lemma~\ref{lem:universal_property_of_deformation_completion}.
    Thus
    \[ \bfMap(\bbD_{-1}, Y) \simeq Y\times_{\mathsf TY} Y\simeq \mathsf T[-1]Y \ . \qedhere \]
\end{proof}

\subsection{Affinization}

Following \cite{Toen_Champs_affines,MRT,Brantner-Mathew-PartitionLie,Raksit}, we introduce the following:

\begin{notation}
    We denote by $\dCAlg_k^{\mathsf{nc}}$ the $\infty$-category of derived \emph{nonconnective} commutative rings, in the sense of \cite[\S4]{Raksit}.
    See also \cite[Appendix~A]{Brantner-Mathew-PartitionLie} for a short review.
\end{notation}

\begin{recollection}[Affinization]\label{recollection:affinization}
    There is an adjunction
    \[ \Gamma_k \colon \dSt_k\op \leftrightarrows \dCAlg_k^{\mathsf{nc}} \colon \Spec^{\mathsf{nc}}_k \ , \]
    where for any $A \in \dCAlg_k^{\mathsf{nc}}$, we define $\mathsf{Spec}^{\mathsf{nc}}_k(A) \in \dSt_k$ by setting for every $\Spec(R) \in \dAff_k$:
    \[ \Map_{\dSt_k}(\Spec(R), \mathsf{Spec}^{\mathsf{nc}}_k(A)) \coloneqq \Map_{\dCAlg_k^{\mathsf{nc}}}(A, R) \ . \]
    We define the following endo-functor of $\dSt_k$, called \textit{affinization}:
    \[ \Aff_k \coloneqq \Spec^{\mathsf{nc}}_k \circ \Gamma_k \ . \]
    Given $X \in \dSt_k$, we denote by
    \[ \aff_X \colon X \longrightarrow \Aff_k(X) \]
     the unit of the adjunction $\Gamma_k \dashv \Spec^{\mathsf{nc}}_k $.
    We refer to the objects in the essential image of $\Spec^{\mathsf{nc}}_k$ as \emph{coaffine stacks}.
    When $k$ is clear out of the context, we drop the subscripts and simply write $\Gamma$, $\Spec^{\mathsf{nc}}$ and $\Aff$.
\end{recollection}

\begin{warning}
    Let $k \to k'$ be a morphism of commutative rings and let $X \in \dSt_k$.
    Write $X_{k'} \coloneqq X \times_{\Spec(k)} \Spec(k')$.
    Then there is a canonical map
    \begin{equation}\label{eq:affinization_change_of_coefficients}
        \Gamma_k(X) \otimes_k k' \longrightarrow \Gamma_{k'}(X_{k'}) \ ,
    \end{equation}
    which induces a canonical morphism
    \[ \Aff_{k'}(X_{k'}) \longrightarrow \Aff_k(X)_{k'} \ . \]
    However, \eqref{eq:affinization_change_of_coefficients} is not necessarily an equivalence.
    Notice that \cite[Notation 4.2.26 \& Proposition 4.2.27]{Raksit} imply that \eqref{eq:affinization_change_of_coefficients} is an equivalence in $\dCAlg_{k'}^{\mathsf{nc}}$ if and only if it is an equivalence in $\Mod_{k'}$.
\end{warning}

\begin{eg}
\label{eg:Betti}
    Let $K \in \Spc$.
    We denote by $K_k$ (or simply $K$ when $k$ is clear from the context) the associated constant $k$-linear derived stack, often known as the Betti stack (in the sense of Simpson).
    It follows from the monodromy correspondence that
    \[ \Gamma(K_k) \simeq \mathsf C^\ast_{\mathsf{sing}}(K;k) \in \dCAlg^{\mathsf{cn}}_k \ . \]
    See for instance \cite[Corollaries 3.7 \& 3.31]{HPT_Homotopy_invariance}.
    When $K = S^1$ and $k = \Z$, \cite[Corollary 3.4]{Toen_Revisited_schematization} provides a further identification
    \[ \Aff_\Z(S^1_\Z) \simeq \mathsf B_\Z\mathsf{Fix} \ , \]
    (beware that in \emph{loc.\ cit.} the notation $\mathbb W^F$ is used in place of $\mathsf{Fix}$).
    Here $\mathsf{Fix}$ is an affine flat group scheme over $\Spec(\Z)$, satisfying $\mathsf{Fix}_\Q \simeq \bbG_a$, and it can be furthermore identified with the Cartier dual of $\widehat{\bbG}_m$.
    Since $S^1$ is a finite CW-complex and since $- \otimes_\Z k \colon \Mod_\Z \to \Mod_k$ commutes with finite limits, it follows that for $X = S^1_\Z$ the comparison map \eqref{eq:affinization_change_of_coefficients} is an equivalence.
    In particular, we obtain
    \[ \Aff_k(S^1_k) \simeq \Aff_\Z(S^1_\Z) \times \Spec(k) \ .  \]
    Similarly, $\mathsf B_k( \mathsf{Fix}_k ) \simeq \mathsf B_\Z\mathsf{Fix} \times \Spec(k)$.
    It follows that
    \[ \Aff_k(S^1_k) \simeq \mathsf B_k(\mathsf{Fix}_k) \]
    for every $k$.
\end{eg}

\begin{eg}
\label{affd-1}
    We compute the affinization of $\mathbb D_{-1}$. 
    Note that the sequence $\{\mathsf B^n\bbG_a\}_{n \geqslant 0}$ is a spectrum object in $\dSt^{\Def}_k$.
    For any stack $X$, the module underlying $\Gamma(X)$ is canonically computed as
    \[ \big\{\Map_{\dSt_k}(X, \mathsf B^n\bbG_a) \big\}_{n \geqslant 0} \in \Sp( (\Mod_k)_{\geqslant 0} ) \simeq \Mod_k \ . \]
    Since every $\mathsf B^n \bbG_a$ is a (higher) Artin stack, it has deformation theory.
    It follows that
    \begin{equation}\label{eq:affinization_D-1}
        \Gamma(\mathbb D_{-1}) \simeq k \times_{k[\varepsilon]} k \simeq k \oplus k[-1] \ ,
    \end{equation}
    where the fiber product is computed in the category of nonconnective $k$-algebras, and $k \oplus k[-1]$ is the nonconnective split square-zero extension of $k$ associated with $k[-1]$.
    In particular, we have $\Aff(\mathbb D_{-1}) \simeq \Spec^{\mathsf{nc}} ( k \oplus k[-1] )$.
    If moreover $k = \Z_{(p)}$-algebra, then \cite{MRT} supplies an identification
    \[ \Aff_{\Z_{(p)}}(\bbD_{-1}) \simeq \mathsf B_{\Z_{(p)}}\mathsf{Ker} \ , \]
    where $\mathsf{Ker}$ is an affine flat group scheme over $\Spec(\Z_{(p)})$ with the property that $\mathsf{Ker}_\Q \simeq \bbG_a$.
    By \cite[Proposition 6.3.3]{MRT}, $\mathsf{Ker}$ is furthermore identified with the Cartier dual of $\widehat{\bbG}_a$.
    The formula \eqref{eq:affinization_D-1} implies the map \eqref{eq:affinization_change_of_coefficients} is an equivalence for $X = \bbD_{-1}$.
    It follows as in Example~\ref{eg:Betti} that for any $\Z_{(p)}$-algebra $k$, one has
    \[ \Aff_k(\bbD_{-1,k}) \simeq \mathsf B_k(\mathsf{Ker}_k) \ . \]
\end{eg}

\begin{eg}\label{lemma:AffS^1=BGa}
    The stack $\BGa$ is coaffine. Indeed, by \cite[Lemma 2.2.5]{Toen_Champs_affines}
    \[ \BGa \simeq \Spec^{\mathsf{nc}}( k[\eta] ) \ , \qquad \deg(\eta) = -1 \ . \]
    Here $k[\eta]$ is the free nonconnective commutative ring associated to the module $k[-1]$.
    In particular, it comes equipped with morphisms
    \begin{equation}\label{eq:comparison_with_BGa_global_sections}
        k[\eta] \longrightarrow k \oplus k[-1] \qquad \text{and} \qquad k[\eta] \longrightarrow \mathsf C^\ast_{\mathsf{sing}}(S^1;k) \ .
    \end{equation}
    In turn, these correspond to morphisms
    \[ \begin{tikzcd}
        \mathbb D_{-1} \arrow{dr}{u} \arrow{d}[swap]{\aff_{\bbD_{-1}}} & & S^1 \arrow{dl}[swap]{v} \arrow{d}{\aff_{S^1}} \\
        \Aff(\mathbb D_{-1}) \arrow{r}{\bar u} & \BGa & \Aff(S^1) \arrow{l}[swap]{\bar v}
    \end{tikzcd} \]
    When $k$ is a $\Q$-algebra, both maps \eqref{eq:comparison_with_BGa_global_sections} are equivalences.
    In other words, $\BGa$ is in this case the common affinization of both $\mathbb D_{-1}$ and of $S^1$:
    \begin{equation}\label{eqn:Formality-weak}
      \Aff(S^1) \simeq  \BGa \simeq \Aff(\mathbb{D}_{-1})\ .
    \end{equation}
\end{eg}

\begin{lem}\label{lemma:AffS^1=BG}
    If $k$ is field, then both $\mathsf{Ker}_k$ (from Example~\ref{affd-1}) and $\mathsf{Fix}_k$ (from Example~\ref{eg:Betti}) are connected groups.
\end{lem}
\begin{proof}
    If $k$ has characteristic zero, then $\mathsf{Ker}_k = \mathsf{Fix}_k = \bbG_a$.
    As already recalled, $\mathsf{Fix}_k$ (resp.\ $\mathsf{Ker}_k$) is the Cartier dual of $\widehat{\bbG}_m$ (resp.\ of $\widehat{\bbG}_a)$.
    By the inverse Cartier duality \cite[Theorem 37.3.12]{Hazewinkel}, 
    \[\hatGm \simeq \Hom_{\operatorname{Groups}_k}(\mathsf{Fix}_k, \Gm)\ , \quad \quad \hatGa \simeq \Hom_{\operatorname{Groups}_k}(\mathsf{Ker}_k, \Gm)\ . \]
    By comparing the sets of geometric points over an algebraically closed field $\overline{K}$ over $k$ on both sides, we see that there is a unique group homomorphism from $\mathsf{Fix}_{\overline{K}}$ (resp.\ $\mathsf{Ker}_{\overline{K}}$) to $({\Gm})_{\overline{K}}$.

    Let $G$ be either $\mathsf{Fix}_k$ or $\mathsf{Ker}_k$.
    Let $\pi_0(G_{\overline{K}})$ be the discrete abelian group of connected components of $G$. If $\pi_0(G_{\overline{K}})$ were non-trivial, take a non-trivial cyclic subgroup $C$ of $\pi_0(G_{\overline{K}})$ and choose a non-trivial homomorphism $f\colon C\to \Gm(\overline{K})$. Since $\Gm(\overline{K})$ is a divisible group hence an injective object, $f$ extends to a non-trivial homomorphism $\pi_0(G_{\overline{K}})\to \Gm(\overline{K})$. This contradicts the uniqueness of such homomorphism.
\end{proof}

\begin{lem}\label{lemma:CatPerfct-AffS^1}
    $\Aff(S^1)$ is $\otimes$-universal and categorically perfect. Moreover, $\mathcal{O}_{\Aff(S^1)}$ is a compact generator of $\QCoh(\Aff(S^1))$.
\end{lem}
\begin{proof}
    All these statements are stable under change of base ring, thanks to Lemma~\ref{lem:tensor_universal_categorically_perfect_change_of_base}.
    It is therefore enough to deal with the case $k = \Z$.
    In this case, by \cite{ScherotzkeSibillaTomasini-FourierMukai}, there is a Fourier--Mukai equivalence of stable categories
    \begin{equation}
    \label{eqn:FMTransform}
        \mathsf{FT}\colon \QCoh(\Aff(S^1)) \simeq \QCoh(\hatGm),
    \end{equation}
    which intertwines the symmetric monoidal structures given by the tensor product and the convolution product.
    In particular, the equivalence respects the tensor units, i.e.~it sends $\cO_{\Aff(S^1)}$ to the skyscraper sheaf $\cO_{1}$.
    Since $\cO_{1}$ is a compact generator of $\QCoh(\hatGm)$, $\cO_{\Aff(S^1)}$ is a compact generator of $\QCoh(\Aff(S^1))$. In particular, $\QCoh(\Aff(S^1))$ is compactly generated and hence $\otimes$-universal.

    Now let $p\colon \Aff(S^1)\to \Spec(k)$ be the structural morphism.
    We show next that $\Aff(S^1)$ is categorically perfect, i.e.\ $p^*$ has left adjoint.
    Composing with the Fourier--Mukai transform \eqref{eqn:FMTransform}, $\mathsf{FT}\circ p^*$ is $i_*$ where $i\colon \{1\}\hookrightarrow \hatGm$. As $i_*$ admits a left adjoint functor $i^*$, we obtain that $p^*$ admits a left adjoint functor $p_{+}=\mathsf{FT}^{-1}\circ i^*$.
\end{proof}

\begin{cor}
\label{cor:BGa-universal}
    When $\mathsf{char}(k) = 0$, $\BGa$ is $\otimes$-universal, categorically perfect and $\mathcal O_{\BGa}$ is a compact generator of $\QCoh(\BGa)$.
\end{cor}

\begin{proof}
    This follows by combining Example~\ref{lemma:AffS^1=BGa} and Lemma~\ref{lemma:CatPerfct-AffS^1}.
\end{proof}

\subsection{A criterion for constancy of $\mathsf{B} G$-points and étale codescent}

The following proposition plays an important technical role in this paper. Intuitively, it states that when a group $G$ is orthogonal to all the stabilizers of a stack $X$, then any map $\mathsf{B}G \to X$ is constant.

\begin{prop}\label{prop:mapping_from_B_connected}
	Let $G$ be an underived algebraic group scheme over $k$.
    Let $X$ be a derived \DM stack over $k$ and assume that for every $k$-algebra $K$ that is a field and every point $x \in X(K)$, one has
    \begin{equation}\label{eq:group_orthogonality_condition}
        \Hom_{\mathsf{Grp}_K}(G_K, \pi_0 \mathsf{Stab}_X(x)) = \ast \ .
    \end{equation}
    Then the canonical morphism
    \[ \bfMap(\mathsf BG, X)\longrightarrow   \bfMap(\Spec(k),X)  \simeq X \]
    induced by the atlas map $\Spec(k) \to \mathsf BG$ is an equivalence on truncations.
\end{prop}

\begin{proof}
	Observe that, since $S$ and $S \times \mathsf BG$ are underived, we can replace $X$ by $\trunc(X)$, or, equivalently, assume that $X$ is underived from the very beginning.
    In more concrete terms, we have to check that for any underived affine scheme $S \in \Aff_k$, the canonical morphism
	\begin{equation}\label{eq:lemma:BG_DM}
		\Map(S \times \mathsf BG, X) \longrightarrow \Map(S,X)
	\end{equation}
	induced by precomposition with $S \to S \times \mathsf BG$, is an equivalence.
	
	\medskip
	
	Since \eqref{eq:lemma:BG_DM} is clearly surjective on $\pi_0$, it suffices to argue that it has contractible fibers.
	Fix therefore $x \colon S \to X$ a morphism.
	The fiber of \eqref{eq:lemma:BG_DM} at $x$ is described as the following mapping space:
	\begin{equation}\label{eq:prop:unipotent_loops_DM_II}
		\Map_{\St_{S/}}( S \times \mathsf BG, X ) \simeq \Map_{\St_{S/\!\!/S}}( S \times \mathsf BG, S \times X ) \ .
	\end{equation}
    Following \cite[Example 5.2.6.13]{HA}, we find an equivalence
	\[ S \times \mathsf BG \simeq \mathsf B_S(S \times G) \simeq \mathsf{Bar}^{(1)}(S \times G) \ , \]
	where the bar construction is performed in the $\infty$-topos $\St_{/S}$ of \emph{underived} stacks over $S$.
	Similarly, $S$ can be seen as the bar construction of the trivial group over $S$.
	Combining \cite[Notation 5.2.6.11 \& Remark 5.2.6.12]{HA}, we see that
	\[ \mathsf{Bar}^{(1)} \colon \mathsf{Mon}_{\mathbb E_1}(\St_{S/\!\!/S}) \longrightarrow \St_{S/\!\!/S} \]
	admits $\mathsf{Cobar}^{(1)}(-) \simeq \Omega^1(-)$ as a right adjoint.
	Here $\Omega^1(-)$ denotes the based loop space functor in $\St_{S /\!\!/ S}$.
	We can therefore rewrite the mapping space \eqref{eq:prop:unipotent_loops_DM_II} as
    \[ \Map_{\St_{S/\!\!/S}}( S \times \mathsf BG, S \times X ) \simeq \Map_{\mathsf{Mon}_{\mathbb E_1}(\St_{/S})}( S \times G, \Omega^1(S \times X) ) \]
    Since $X$ is a $1$-stack and $G$ is underived, it follows that this mapping space is discrete.
	Thus, in order to conclude the argument it is therefore enough to argue that any morphism of $S$-groups
	\[ S \times G \longrightarrow \Omega^1(S \times X) \]
    is constant, that is to show that it factors through the unit section of $\Omega^1(S \times X)$.

    \medskip
	
	Observe now that
	\[ \Omega^1(S \times X) \coloneqq S \times_{S \times X} S \simeq S \times_{S \times X} ( \mathsf IX \times S ) \ , \]
	and that the unit section of this $S$-group is the pullback of the diagonal embedding
	\[ X \times S \longrightarrow \mathsf IX \times S \ . \]
	Since $X$ is a \DM stack, the diagonal $X \to X \times X$ is unramified.
	This implies that the map $X \to \mathsf IX$ is an open immersion, see \cite[\href{https://stacks.math.columbia.edu/tag/02GE}{Tag 02GE}]{stacks-project}.
	It follows that the unit section $S \to \Omega^1(S \times X)$ is an open immersion.
	In particular, a morphism $S \times G \to \Omega^1(S \times X)$ factors through the unit section if and only if it factors topologically.
	In order to check this latter statement, it is enough to assume that $S=\Spec K$ for a $k$-algebra $K$ that is a field.
	In this case, the statement follows directly from the assumption \eqref{eq:group_orthogonality_condition}.
\end{proof}

\begin{rem}
\label{rmk:ConditionToApply}
    The orthogonality condition \eqref{eq:group_orthogonality_condition} of \cref{prop:mapping_from_B_connected} is in particular satisfied when $G$ as a group scheme over $k$ is fiberwise connected, or when $G$ is a constant $p$-group scheme and $X$ is furthermore tame.
\end{rem}

As an application, we deduce the following étale codescent property:

\begin{cor} \label{cor:etale_codescent_B_connected}
	Let $G$ be underived, fiberwise-connected algebraic group scheme over $k$.
	Let $X$ be a derived \DM stack over $k$ and let $X\et$ be the small étale site of $X$.
	Then the functor
	\[ \bfMap(\mathsf BG, -) \colon X\et \longrightarrow \dSt_k \]
	satisfies étale codescent.
\end{cor}

\begin{proof}
	Since the functor $\bfMap(\mathsf BG, -)$ commutes with limits, it commutes also with the formation of \v{C}ech nerves.
	In particular, it suffices to show that if $f \colon U \to V$ is an \'etale epimorphism, the same goes for
	\[ \phi_f \colon \bfMap(\mathsf BG, U) \longrightarrow \bfMap(\mathsf BG, V) \ . \]
	First, we observe that \cref{variant:cotangent_complex_mapping_stack_formally_etale_target} implies that $\phi_f$ is formally étale.
	It immediately follows that $\phi_f$ is an effective epimorphism if and only if its truncation $\trunc(\phi_f)$ is.
	Besides, since $G$ is underived, for any derived stack $Y$ we have a canonical identification
	\[ \trunc \bfMap(\mathsf BG, Y) \simeq \trunc \bfMap(\mathsf BG,\trunc(Y)) \ . \]
	We can therefore assume from the very beginning that $X$ (and hence $U$ and $V$) is underived.
	
	\medskip
	
	Fix an affine test scheme $S \in \Aff_k$, and a morphism $S \to \bfMap(\mathsf BG, V)$, corresponding to a morphism $S \times \mathsf BG \to V$.
	We have to show that there exists an étale cover $S' \to S$ such that the restriction
	\[ S' \times \mathsf BG \longrightarrow S \times \mathsf BG \longrightarrow V \]
	factors through $f \colon U \to V$.
    Since $G$ is connected, it follows from \cref{prop:mapping_from_B_connected} and \cref{rmk:ConditionToApply} that
	\[ \Map(S \times \mathsf BG, V) \longrightarrow \Map(S, V) \]
	is an equivalence, and similarly for $U$ in place of $V$.
	Thus, the conclusion follows from the fact that $f \colon U \to V$ was an effective epimorphism to begin with.
\end{proof}

\begin{rem}
    The assumption that the source of the mapping stack is of the form $\mathsf BG$ with $G$ connected cannot be easily weakened.
    For instance, $\bfMap(S^1,-)$ only satisfies codescent with respect to \emph{representable} étale covers.
\end{rem}

\subsection{Unipotent loops}

Consider the canonical map
\[ \aff_{\bbD_{-1}} \colon \bbD_{-1} \longrightarrow \Aff(\bbD_{-1}) \ . \]
For every $Y \in \dSt_k$, we therefore obtain a natural comparison morphism
\[ \aff_{\bbD_{-1}}^\ast \colon \bfMap(\Aff(\bbD_{-1}), Y) \longrightarrow \bfMap(\bbD_{-1}, Y) \ . \]
We have:

\begin{prop}\label{prop:unipotent_loops_DM}
    Assume that $k$ is a $\Z_{(p)}$-algebra.
	Let $Y$ be a derived \DM stack locally almost of finite presentation over $k$. Then the morphism $\aff_{\bbD_{-1}}^\ast$ is an equivalence.
\end{prop}

\begin{proof}
    When $Y$ is derived affine, the statement is true by the universal property of the affinization. Therefore, it is enough to show that both the source and the target of $\aff_{\bbD_{-1}}^\ast$ satisfy \'etale codescent on $Y$.
    On one hand, the target of $\aff_{\bbD_{-1}}^\ast$ trivially satisfies \'etale codescent on $Y$.
    On the other hand, $\Aff(\bbD_{-1}) \simeq \mathsf B(\mathsf{Ker}_k)$, where $\mathsf{Ker}$ is the Cartier dual of $\hatGa$ over $k$.
    Lemma~\ref{lemma:AffS^1=BG} guarantees that $\mathsf{Ker}_k$ is a fiberwise connected algebraic group scheme over $k$.
    Therefore the source of $\aff_{\bbD_{-1}}^\ast$ also satisfies étale codescent by Corollary~\ref{cor:etale_codescent_B_connected}.
\end{proof}

\begin{rem}
    It is possible to drop the assumption that $k$ is a $\Z_{(p)}$-algebra.
    Defining $\mathsf{Ker}$ to be the Cartier dual of $\hatGa$ over $\Z$, and mimicking the proof of \cite[Corollary 3.4]{Toen_Revisited_schematization} one sees that the identification $\Aff(\bbD_{-1}) \simeq \mathsf B\mathsf{Ker}$ holds over $\Z$.
    Since we do not really need this fact, we leave the details to the reader.
\end{rem}

Recall now from Example~\ref{lemma:AffS^1=BGa} that there is a canonical map
\[ u \colon \bbD_{-1} \longrightarrow \BGa \]
which factors through $\Aff(\bbD_{-1})$.

\begin{cor}\label{cor:unipotent_loops_DM}
    Let $Y$ be a derived \DM stack locally almost of finite presentation over $k$.
    If $\mathsf{char}(k) = 0$, the morphism
    \[ u^\ast \colon \bfMap(\BGa, Y) \longrightarrow \bfMap(\bbD_{-1}, Y) \]
    is an equivalence.
\end{cor}

\begin{proof}
    This follows from Proposition~\ref{prop:unipotent_loops_DM} and the identification $\Aff(\bbD_{-1}) \simeq \BGa$ of Example~\ref{lemma:AffS^1=BGa}.
\end{proof}

\begin{rem}
    In \cite{BenZvi_Nadler_Loop_spaces_and_connections} and \cite{chen2020equivariant}, the mapping stack $\bfMap(\BGa, Y)$ is called the \textit{unipotent loop space} of $Y$.
\end{rem}

\section{Orbifold inertia for derived Deligne--Mumford stacks}

Let $X$ be a derived \DM stack (see for example \cite{Porta_Comparison} for the basic definition). 
We emphasize that in this work we only consider $1$-stacks. Also, as we stated in the Convention section, all DM stacks are assumed to be locally almost of finite presentation over the base ring $k$. 
We typically think of $X$ as a derived stack, but we will occasionally use its representation as a structured $\infty$-topos.

\medskip

For an underived Deligne-Mumford stack $X$, it is customary to define the inertia stack of $X$ as
\[ \inertia X \coloneqq X \times_{X \times X} X \ , \]
where the fiber product is taken in the category $\St_k$ of underived stacks.
It can equivalently be realized as the (underived) mapping stack from the constant stack associated to $S^1 \in \Spc$:
\begin{equation}
    \inertia X\simeq \bfMap_{\St_k}(S^1, X) \ .
\end{equation}
The very same constructions performed in $\dSt_k$ yields the free loop space $\loopstack X$:
\begin{equation}
    \loopstack X\simeq \bfMap(S^1, X)\simeq X\times^{d}_{X\times X} X \ ,
\end{equation}
where the fiber product is taken in the category of derived stacks.
To generalize the HKR theorem, one needs to introduce a meaningful construction of the inertia stack.

\subsection{(Derived) orbifold inertia}

Our definition is motivated by the work of Abramovich--Graber--Vistoli \cite{Abramovich_Orbifold_quantum_product},
and arose from discussions of the second-named author with Adeel Khan, Charanya Ravi and Tony Yue Yu.
For every integer $r > 0$, let $\mathsf C_r$ be the cyclic group of order $r$, seen as a constant group scheme over $\Spec(k)$.
When $r$ divides $ r'$,  there is a canonical group homomorphism $\mathsf C_{r'} \twoheadrightarrow \mathsf C_r$, and we set
\[ \hatZ \coloneqq \flim_{r} \mathsf C_r \in \mathsf{Mon}_{\mathbb E_1}^{\mathsf{gp}}\big( \Pro(\dSt_k) \big) \ . \]
Passing to the classifying stack, we also set
\[ \hatcircle \coloneqq \flim_r \BCr \in \Pro(\dSt_k) \ . \]
Notice that the canonical homomorphisms $\Z \twoheadrightarrow\mathsf C_r$ induce at the level of classifying stacks the map
\[ \varpi_r \colon S^1 \longrightarrow \BCr \ , \]
and all these maps assemble into a morphism
\begin{equation}
\label{eqn:S1toS1hat}
    \varpi \colon S^1 \longrightarrow \hatcircle \ . 
\end{equation}

\begin{variant}\label{variant:tame_circle}
    We let $\mathsf{char}(k)$ be the set of prime numbers that are non-invertible in $k$ and we set
    \[ \hatcircletame \coloneqq \flim_{(r,\mathsf{char}(k)) = 1} \BCr \ . \]
    Notice that there is a canonical map $\hatcircle \to \hatcircletame$.
    We denote by $\varpi_{\tame} \colon S^1 \to \hatcircletame$ the composite of $\varpi$ with the previous map.
\end{variant}

\begin{defin}[Orbifold inertia]
\label{def:IDM}
	Let $X \in \dSt_k$ be a derived stack.
	\begin{enumerate}\itemsep=0.2cm
		\item The \emph{$r$-th orbifold inertia of $X$} is the mapping stack
		\[ \inertiarth X \coloneqq \bfMap(\BCr, X) \ . \]
		
		\item The \emph{orbifold inertia of $X$} is the mapping stack
		\[ \inertiaDM X \coloneqq \bfMap( \hatcircle, X ) \ . \]

        \item The \emph{tame orbifold inertia of $X$} is the mapping stack
        \[ \inertiaDMtame(X) \coloneqq \bfMap( \hatcircletame, X) \ .  \]
	\end{enumerate}
    There is a canonical morphism $\inertiaDMtame(X)\to \inertiaDM(X)$.
\end{defin}

Unraveling the definitions, we see that
\[ \inertiaDM X \simeq \colim_r \inertiarth X \ , \]
where the filtered colimit is computed in $\dSt_k$.
Similarly, $$\inertiaDMtame(X)\simeq \colim_{(r,\mathsf{char}(k)) = 1} \inertiarth X \ .$$ 
%with the colimit ranging over all $r$ coprime to $\mathsf{char}(k)$.

\begin{notation}
	The maps $\varpi$, $\varpi_{\tame}$ and  $\varpi_r$  induce well-defined morphisms on mapping stacks:
	\[ \iota_\varpi \colon \inertiaDM X \longrightarrow \loopstack X \qquad \ , \iota_{\varpi_{\tame}} \colon \inertiaDMtame X \longrightarrow \loopstack X \qquad \text{and} \qquad \iota_r \colon \inertiarth X \longrightarrow \loopstack X \ . \]
	We also let
	\[ \pi \colon \loopstack X \longrightarrow X \ , \qquad \piDM \colon \inertiaDM X \longrightarrow X \ , \qquad \piDM_{\tame} \colon \inertiaDMtame X \longrightarrow X \ , \qquad \pi_r \colon \inertiarth X \longrightarrow X \]
	be the natural projections.
\end{notation}

\begin{defin}[Strong tameness]\label{defin:tame}
A derived \DM stack $X$ is said to be \textit{strongly tame} if all the stabilizers of the truncation of $X$ have cardinalities coprime to every prime number in $\mathsf{char}(k)$, where $\mathsf{char}(k)$ is the set of all non-invertible primes in $k$. 
\end{defin}
See Remark \ref{rem:StrongTameness} for the relation with the classical notion of tameness \cite{Abramovich_Olsson_Vistoli_Tame_stacks}.
%Following \cite{Abramovich_Olsson_Vistoli_Tame_stacks}, an Artin stack $X$ with finite inertia is said to be \emph{tame} if the (derived) pushforward functor along the projection $p \colon X \to M$ to its coarse moduli space is $t$-exact. 

The following lemma is not needed in the rest of the paper, but we included it to clarify the relation between $\inertiaDM$ and $\inertiaDMtame$.

\begin{lem}\label{lem:inertiaDM_tame_DM}
    Let $X$ be a strongly tame \DM stack over $k$.
    Then for every prime $p$ that is not invertible in $k$, and every integer $n \geqslant 1$, the canonical map
    \[ \bfMap(\mathsf{BC}_{p^n}, X) \longrightarrow \bfMap(\Spec(k),X) \simeq X \ . \]
    induced by the atlas map $\Spec(k) \to \mathsf{BC}_{p^n}$ is an equivalence on truncations.
    In particular, the canonical map
    \[ \inertiaDMtame(X) \longrightarrow \inertiaDM(X) \]
    is an equivalence on truncations.
\end{lem}

\begin{proof}
    Since $X$ is strongly tame, the condition \eqref{eq:group_orthogonality_condition} is satisfied. Therefore, the first half of the statement is a direct consequence of \cref{prop:mapping_from_B_connected}.
    As for the second half, observe that if $r$ is any integer, then it can be uniquely written as $r = r' r''$, where $r'$ is a product of primes in $\mathsf{char}(k)$ and $r''$ is coprime with $\mathsf{char}(k)$.
    Then the first half implies that the canonical map
    \[ \bfMap(\BCr, X) \simeq \bfMap( \mathsf{BC}_{r''}, \bfMap(\mathsf{BC}_{r'}, X)) \longrightarrow \bfMap(\mathsf{BC}_{r''}, X) \]
    is an equivalence on truncations.
    Let $\mathbb N$ be considered as the poset of integers ordered by divisibility and let $\mathbb N_k$ be the subposet of integers coprime with $\mathsf{char}(k)$.
    The inclusion $i \colon \mathbb N_k \hookrightarrow \mathbb N$ has a left adjoint $L$ given by
    \[ L(r) \coloneqq \frac{r}{r_k} \ , \qquad r_k \coloneqq \prod_{p \in \mathsf{char}(k)} p^{v_p(r)} \ , \]
    where $v_p(r)$ is the $p$-adic valuation of $r$.
    Standard category theory supplies a canonical identification $i_! \simeq L^\ast$, where $i_!$ is the left Kan extension along $i$ and $L^\ast$ is precomposition with $L$.
    The above argument allows therefore to identify the functor $\Phi \colon \mathbb N \to \dSt_k$ given by
    \[ r \longmapsto \trunc \bfMap(\BCr, X) \]
    with $L^\ast(i^\ast(\Phi)) \simeq i_!(i^\ast(\Phi))$.
    Since left Kan extensions do not alter colimits, the conclusion follows.
\end{proof}

\begin{warning}
    Notice that Lemma~\ref{lem:inertiaDM_tame_DM} only holds true at the level of truncations.
    To see that this is necessary, it is enough to consider the case $X = \mathbb A^1_k$.
    In this situation,
    \[ \trunc\big(\bfMap(\mathsf{BC}_p, \mathbb A^1_k)\big) \simeq \mathbb A^1_k \ . \]
    Fix a map $\mathsf{BC}_p \to \mathbb A^1_k$, which necessarily factors through $\Spec(k)$.
    Then a simple computation with cotangent complexes shows that the cotangent complex of $\bfMap(\mathsf{BC}_p, \mathbb A^1_k)$ is the dual of the group cohomology of $\mathsf C_p$ with coefficients in the trivial representation.
    This contribution persists inside $\inertiaDM(\mathbb A^1_k)$, and because of this reason, in the rest of the paper we essentially never use $\inertiaDM$ but only its tame variant.
\end{warning}

The main result of this section is the following properties of $\inertiaDMtame X$:

\begin{thm}\label{thm:basics_inertia}
	Assume that $X$ is a derived $1$-Artin stack locally almost of finite presentation.
    For every positive integer $r$ coprime to $\mathsf{char}(k)$, the following holds:
	\begin{enumerate}\itemsep=0.2cm
		\item the map
		\[ \iota_r \colon \inertiarth X \longrightarrow \loopstack X \]
		is a closed immersion.
		In particular, $\inertiarth X$ is a derived Artin stack.
		
		\item If the cotangent complex of $X$ is perfect (resp.\ has tor-amplitude in $[a,b]$), the same holds for the cotangent complex of $\inertiarth X$.
		In particular, if $X$ is smooth (resp.\ lci  \footnote{In this paper, the terminology \textit{lci} stands for \textit{derived lci}, a.k.a. \textit{quasi-smooth}, which means that the tor-amplitude of the cotangent complex belongs to $[-1,1]$. Notice that an underived stack $X$ is derived lci if and only if $X$ is lci in the classical sense, see for instance \cite[Lemma 2.4]{Porta_Yu_NQK} (whose proof works verbatim in the algebraic setting).}), the same holds for $\inertiarth X$. 
		
		\item for every $r' \mid r$ the transition map $\mathsf I^{(r')} X \to \inertiarth X$ is an open and closed immersion, and in particular $\inertiaDMtame(X)$ is a derived $1$-Artin stack locally almost of finite presentation.
	\end{enumerate}
	Assume furthermore that $X$ is a derived \DM stack locally almost of finite presentation and strongly tame.
	Then:
	\begin{enumerate}\setcounter{enumi}{3}\itemsep=0.2cm
		\item $\inertiaDMtame X$ is a strongly tame derived \DM stack locally almost of finite presentation, and the map
		\[ \iota_\varpi \colon \inertiaDMtame X \longrightarrow \loopstack X \]
		induces an isomorphism on the truncation: $\trunc\inertiaDMtame X\simeq \trunc\loopstack X\simeq \inertia (\trunc(X))$.
	\end{enumerate}
\end{thm}

\begin{cor}
\label{cor:InertiaDM-Smooth}
	Let $X$ be a smooth and strongly tame \DM stack locally almost of finite presentation.
	Then $\inertiaDMtame X$ is also smooth.
	In particular, it is underived and therefore it coincides with
	\[ \inertia X=X\times_{X \times X} X \ , \]
	where the fiber product is computed in the category of \emph{un}derived stacks $\St_k$.
\end{cor}

\begin{proof}
	The smoothness follows by combining items (2) and (3) of \cref{thm:basics_inertia}.
	The fact that it coincides with the usual inertia stack follows from item (4).
\end{proof}

\begin{rem}
	\cref{thm:basics_inertia} recovers the considerations of \cite[\S3.1]{Abramovich_GW_theory_for_DM_stacks} in the derived setting.
	The proof below might seem weirdly long compared to the arguments given in \emph{loc. cit.}
	The ultimate reason is that whereas in the classical setting a point of $\inertiarth X$ is the datum of a pair $(x,g)$ where $x$ is a point of $X$ and $g$ is an automorphism of $x$ whose order divides $r$, it is actually challenging to obtain a similar description of $\inertiarth X$ in the derived setting.
	This difficulty can be precisely quantified in terms of a basic computation in homotopy theory, see \cref{prop:homotopy_computation} below.
\end{rem}

\begin{rem}
    In general, when $X$ is underived but singular, $\inertiaDMtame X$ can be a non-trivial derived enhancement of $\inertia X$. See \Cref{example:FixedLociBecomeDerived}.
\end{rem}

%\begin{prop}
%\todo{This is wrong}
%    Let $X$ be an underived \DM stack.
%	Then $\inertiaDM X$ is an underived \DM stack and $\inertiaDM X\cong \inertia X$.
%\end{prop}
%\begin{proof}
%    As the claim that $\inertiaDM X$ is underived can be checked \'etale locally on $X$, we can assume that $X=[Y/G]$ where $Y$ is an underived scheme and $G$ is a finite group acting on $Y$. By \Cref{prop:InertiaGlobalQuotient}, $\inertiaDM X$ is a disjoint union of quotient stacks of the form $[Y^g/Z(g)]$, where $Z(g)$ is the centralizer of $g$ and $Y^g$ denotes the \textit{genuine fixed locus} of $g$, introduced in \Cref{genuinefixed}. By \Cref{lemma:GenuineFixedUnderived}, $Y^g$ is an underived scheme, hence $[Y^g/Z(g)]$ is an underived \DM stack. This shows that $\inertiaDM X$ is underived, hence coincides with its classical truncation, which is naturally isomorphic to $\inertia X$ by \Cref{thm:basics_inertia}.
%\end{proof}

\subsection{Deformation theory of orbifold inertia}

We start by some general considerations on the cotangent complex of mapping stacks.
The material covered here will be crucial in the proof of \cref{thm:basics_inertia} (and especially item (4) of that theorem), but it will also be used later on in the proof of the HKR theorem.

For the finiteness conditions appearing in the following lemma, see \Cref{Appendix:Map}.
\begin{lem}\label{lem:finiteness_BCr}
	For a positive integer $r$ coprime with $\mathsf{char}(k)$, the stack $\BCr$ is $\otimes$-universal, categorically quasi-compact and categorically perfect.
\end{lem}

\begin{proof}
    This follows from the fact that $\BCr$ is tame, but since the proof is short, we include it for the convenience of the reader.
    Let $k' \coloneqq k[T]/(T^r - 1)$.
    Since $r$ is coprime with $\mathsf{char}(k)$, the morphism $k \to k'$ is faithfully flat and étale.
    We prove below that $\QCoh((\BCr)_{k'})$ is compactly generated and that $q_\ast$ commutes with filtered colimits and preserves perfect complexes.
    Since compact generation can be tested étale-locally by \cite{Toen_Azumaya}, this implies that $\QCoh(\BCr)$ is compactly generated as well, and therefore $\otimes$-universal.
    Since $\BCr$ is tame with coarse moduli space $\Spec(k)$, we see that base-change for the diagram
    \[ \begin{tikzcd}
        (\BCr)_{k'} \arrow{r} \arrow{d}{q} & \BCr \arrow{d}{p} \\
        \Spec(k') \arrow{r}{u} & \Spec(k)
    \end{tikzcd} \]
    holds. Thus, conservativity of $u^\ast$ immediately implies that $p_\ast$ commutes with filtered colimits and preserves perfect complexes as well.
    In particular, it is categorically perfect as a consequence of Lemma~\ref{lem:categorical_properness} 

    We are left to prove the claim.
    Equivalently, we prove the statement under the additional assumption that $k$ contains all the $r$-roots of unity.
    In this case, $\mathsf C_r$ is diagonalizable, i.e.\ its Cartier dual $\mu_r$ is a constant group.
    Thus, following \cite[Theorem 5.4]{Bergh_Schnurer} (see also \cite[Theorem 5.25]{Binda_Porta_Azumaya}), we obtain a canonical decomposition
	\begin{equation}\label{eq:character_decomposition_BCr}
		\QCoh(\BCr) \simeq \prod_{\zeta \in \mu_r} \QCoh_\zeta(\BCr) \ ,
	\end{equation}
	and moreover we have canonical identifications $\QCoh_\zeta(\BCr) \simeq \Mod_k$.
	It follows that $\QCoh(\BCr)$ is compactly generated, and therefore $\BCr$ is $\otimes$-universal.
	It is also categorically quasi-compact, as the pushforward is canonically identified with taking the component indexed by $\zeta = 1$ in the decomposition  \eqref{eq:character_decomposition_BCr}.
	Finally, this operation preserves perfect complexes, and therefore we deduce that $\BCr$ is categorically perfect as a consequence of \cref{lem:categorical_properness}.
\end{proof}

\begin{lem}\label{lem:inertiaDM_etale_transitions}
	Let $r' \mid r$ be positive integers, and let $f \colon \BCr \to  \mathsf{BC}_{r'}$ be the natural homomorphism.
	Let $p \colon \mathsf{BC}_{r} \to \Spec(k)$ and $q \colon \mathsf{BC}_{r'} \to \Spec(k)$ be the structural morphisms.
    If $r$ is coprime to $\mathsf{char}(k)$, then the natural $+$-Beck-Chevalley transformation
	\[ p_+ \circ f^\ast \longrightarrow q_+ \]
	is an equivalence.
\end{lem}

\begin{proof}
    Reasoning as in the first half of the proof of Lemma~\ref{lem:finiteness_BCr}, we can assume that $k$ contains all $r$-roots of unity.
	Under the character decomposition \eqref{eq:character_decomposition_BCr} for $\QCoh(\BCr)$, the functor $f^\ast$ correspond to the map
	\[ \prod_{\mu_{r'}} \Mod_k \longrightarrow \prod_{\mu_{r}} \Mod_k \]
	induced by the inclusion $\mu_{r'} \subset \mu_{r}$.
	Given $\zeta \in \mu_{r'}$, write $k_\zeta$ for the element $k \in \Mod_k \simeq \QCoh_\zeta(\Mod_k)$.
	We therefore have $f^\ast(k_\zeta) \simeq k_\zeta$, and it suffices at this point to observe that
	\[ q_+(k_\zeta) \simeq q_\ast(k_{\zeta^{-1}})^\vee \simeq \begin{cases}
		k & \text{if } \zeta = 1 \\
		0 & \text{otherwise.} 
	\end{cases} \qedhere \]
\end{proof}

\begin{prop}\label{prop:deformation_theory_inertiarth}
	Let $X$ be a derived Artin stack and let $r$ be a positive integer coprime to $\mathsf{char}(k)$.
	Then $\inertiarth X$ is infinitesimally cohesive, nilcomplete and admits a cotangent complex.
	Moreover:
	\begin{enumerate}\itemsep=0.2cm
		\item if $\mathbb L_X$ is perfect of tor-amplitude $[a,b]$, the same holds for $\mathbb L_{\inertiarth X}$;
		\item if $r' \mid r$, the induced morphism $\mathsf I^{(r')} X \to \inertiarth X $ is formally étale.
	\end{enumerate}
\end{prop}

\begin{proof}
	Since $X$ is Artin, it is itself infinitesimally cohesive and nilcomplete.
	Therefore, the same holds for $\bfMap(F,X)$ for any derived stack $F$, in particular for $\inertiarth X$.
	The existence of the cotangent complex follows by combining \Cref{lem:finiteness_BCr} with \cref{prop:cotangent_complex_mapping_stack}.
	Again by \cref{prop:cotangent_complex_mapping_stack}, the cotangent complex of $\inertiarth X$ is computed by the formula $$\mathbb L_{\inertiarth X}\simeq \pi_+ \ev^\ast(\mathbb L_X).$$
	Since $\pi_+$ consists of taking the dual of $\mathsf C_r$-invariants (and since $r$ is assumed to be coprime to $\mathsf{char}(k)$, this operation is $t$-exact), we see that point (1) holds.
	As for statement (2), it follows from the combination of  \cref{lem:inertiaDM_etale_transitions} and \cref{prop:universal_criterion_formally_etale}.
\end{proof}

\subsection{A computation in homotopy theory}

We will reduce the proof of \cref{thm:basics_inertia} to a computation in homotopy theory.
We start by introducing some notations.

\begin{notation}
	Let $r \in \Z$ be an integer.
	We denote by
	\[ \mathsf{cov}_r \colon S^1 \to S^1 \]
	the standard degree-$r$ covering map, corresponding to the element $r \in \Z \simeq \pi_1(S^1)$.
	We write
	\[ \Gamma_r \coloneqq \cofib( \mathsf{cov}_r ) \in \Spc \ . \]
	Notice that the commutative square
	\[ \begin{tikzcd}
		\Z \arrow{r}{r} \arrow{d} & \Z \arrow{d} \\
		\ast \arrow{r} & \mathsf C_r
	\end{tikzcd} \]
	of groups induces passing to classifying stacks a commutative square
	\[ \begin{tikzcd}
		S^1 \arrow{r}{\mathsf{cov}_r} \arrow{d} & S^1 \arrow{d} \\
		\ast \arrow{r} & \BCr \ ,
	\end{tikzcd} \]
	and therefore a canonical comparison map
	\[ \gamma_r \colon \Gamma_r \longrightarrow \BCr \ . \]
\end{notation}
\begin{eg}
    For $r=2$, $\Gamma_2\simeq \mathbb{R}\mathbb{P}^2$, $\mathsf{BC}_2\simeq \mathbb{R}\mathbb{P}^{\infty}$ and the map $\gamma_2$ is the natural inclusion. 
\end{eg}

\begin{prop}\label{prop:homotopy_computation}
	\hfill
	\begin{enumerate}\itemsep=0.2cm
		\item The space $\Gamma_r$ admits the bouquet $(S^2)^{\vee r-1}$ as a universal cover.
		
		\item The map $\gamma_r$ exhibits $\BCr$ as the $1$-Postnikov truncation of $\Gamma_r$.
	\end{enumerate}
\end{prop}

\begin{proof}
	We represent $S^1$ as the complex circle, and we compute $\Gamma_r$ as the strict pushout of the diagram
	\[ \begin{tikzcd}
		S^1 \arrow{r}{z \mapsto z^r} \arrow{d} & S^1 \arrow{d} \\
		\mathbb D \arrow{r} & \Gamma_r \ ,
	\end{tikzcd} \]
	where $\mathbb D \coloneqq \{z \in \C \mid |z| \leqslant 1\}$.
	We can alternatively describe $\Gamma_r$ as follows.
	Let $B_r$ be the strict pushout
	\[ \begin{tikzcd}
		S^1 \times \mu_r \arrow{r} \arrow{d} & \mathbb D \times \mu_r \arrow{d} \\
		S^1 \arrow{r} & B_r
	\end{tikzcd} \]
	There is a canonical action of $\mu_r$ on $B_r$ given by
	\[ \zeta \cdot (z, \zeta') \coloneqq (\zeta \cdot z, \zeta \zeta') \ , \]
	and $\Gamma_r$ coincides with the quotient of $B_r$ by this action.
	Furthermore, this action is free and properly discontinuous and therefore the quotient map
	\[ B_r \longrightarrow \Gamma_r \]
	is a covering map.
	Finally, it is easy to see that $B_r$ is geometrically the gluing of $r$ copies of $\mathbb{D}$ along their boundaries and hence homotopically equivalent to the bouquet of $r-1$ copies of $S^2$.
	This proves statement (1).
	It also implies that $B_r$ is connected and
	\[ \pi_n(B_r) \simeq \begin{cases}
		\mathsf C_r & \text{if } n = 1 \\
		\pi_n( (S^2)^{\vee r-1} ) & \text{if } n \geqslant 2 \ .
	\end{cases} \]
	Statement (2) follows at once.
\end{proof}

Precomposition with $\mathsf{cov}_r$ induces for every $F \in \dSt_k$ a well-defined operation
\[ (-)^r \colon \loopstack F \longrightarrow \loopstack F \]
over $F$.
Write
\[ \delta_F \colon F \longrightarrow \loopstack F \]
for the relative diagonal.

\begin{cor}\label{prop:inertia_key_pullback}
	Assume that $F$ is a $1$-Artin derived stack.
	Then the square
	\[ \begin{tikzcd}
		\inertiarth F \arrow{r} \arrow{d} & \loopstack F \arrow{d}{(-)^{r}} \\
		F \arrow{r}{\delta_F} & \loopstack F
	\end{tikzcd} \]
	is a pullback square on truncations.
\end{cor}

\begin{proof}
	Unraveling the definitions, we see that for every $F \in \dSt_k$, the square
	\[ \begin{tikzcd}
		\bfMap(\Gamma_r, F) \arrow{r} \arrow{d} & \loopstack F \arrow{d}{(-)^r} \\
		F \arrow{r}{\delta_F} & \loopstack F
	\end{tikzcd} \]
	is a pullback square.
	Since $\trunc \colon \dSt_k \to \St_k$ preserves limits, it is therefore enough to argue that the map
	\[ \inertiarth F = \bfMap(\BCr, F) \longrightarrow \bfMap(\Gamma_r, F) \]
	becomes an equivalence after applying $\trunc$.
	Equivalently, we have to check that for every underived test affine scheme $S \in \Aff_k$, the map
	\[ \bfMap(\BCr, F)(S) \longrightarrow \bfMap(\Gamma_r, F)(S) \]
	is an equivalence.
	Since both $\BCr$ and $\Gamma_r$ are the constant stacks, we can identify this map with the map
	\[ \Map_{\Spc}(\BCr, F(S)) \longrightarrow \Map_{\Spc}(\Gamma_r, F(S)) \]
	induced by $\gamma_r \colon \Gamma_r \to \BCr$.
	Since $F$ is a $1$-Artin stack and $S$ is underived, $F(S)$ is a $1$-groupoid.
	Therefore, the conclusion follows directly from the fact that $\BCr$ coincides with the $1$-Postnikov truncation of $\Gamma_r$, as proven in \cref{prop:homotopy_computation}-(2).
\end{proof}

\subsection{Proof of \cref{thm:basics_inertia}}

We are now ready for the proof of the main theorem of this section.
We start by collecting the consequences of the results obtained so far.

\begin{cor}\label{cor:inertia_Artin}
	Let $X$ be a $1$-Artin stack locally almost of finite presentation and with separated diagonal.
	Then for every positive integer $r$ coprime to $\mathsf{char}(k)$ the following hold:
	\begin{enumerate}\itemsep=0.2cm
		\item the map
		\[ \iota_r \colon \inertiarth X \longrightarrow \loopstack X \]
		is a closed immersion.

		\item the derived stack $\inertiarth X$ is $1$-Artin and locally almost of finite presentation.
		
		\item For any $r' \mid r$, the induced morphism
		\[ \mathsf I^{(r')} X \longrightarrow \inertiarth X \]
		is an open and closed immersion, and therefore $\inertiaDMtame(X)$ is a derived $1$-Artin stack locally almost of finite presentation. 
	\end{enumerate}
\end{cor}

\begin{proof}
	For (1), it is enough to prove that $\iota_r$ is a closed immersion after passing to classical truncations.
	Thanks to \cref{prop:inertia_key_pullback}, it is enough to argue that $\delta_X \colon X \to \loopstack X$ is a closed immersion, which immediately follows from the assumption that the diagonal of $X$ is separated.
	In particular, $\trunc(\inertiarth X)$ is a $1$-Artin stack of finite presentation (in the underived sense).
	Point (2) follows combining \cref{prop:deformation_theory_inertiarth} with Lurie's representability theorem.
	As for point (3), we already know as a consequence of (1) that the morphism in question is a closed immersion.
	\Cref{prop:deformation_theory_inertiarth} implies that it is formally étale, and therefore that it must be an open immersion as well.
\end{proof}

We complete the proof of \cref{thm:basics_inertia} by showing the following:

\begin{prop}
	Let $X$ be a derived \DM stack locally almost of finite presentation.
	Then for any $r$ coprime to $\mathsf{char}(k)$, the map
	\[ \iota_r \colon \inertiarth X \longrightarrow \loopstack X \]
	induces an open immersion on truncations.
	Moreover, if $X$ is strongly tame, the morphism
	\[ \iota_\varpi^{\tame} \colon \inertiaDMtame X \longrightarrow \loopstack X \]
	is an equivalence on truncations.
\end{prop}

\begin{proof}
    First observe that
    \[ \trunc(\loopstack X) \simeq \trunc( \loopstack(\trunc(X)) \simeq \mathsf I( \trunc(X)) \ . \]
    To prove the first statement, we can therefore assume from the very beginning that $X$ is underived.
    Thanks to Proposition \ref{prop:inertia_key_pullback} it is sufficient to argue that $\delta_X \colon X \to \trunc(\loopstack X)$ is an open immersion.
    As already remarked in the proof of Proposition \ref{prop:mapping_from_B_connected}, since $X$ is a \DM stack, the diagonal $X \to X \times X$ is unramified and therefore the map $X \to \mathsf IX$ is an open immersion, see \cite[\href{https://stacks.math.columbia.edu/tag/02GE}{Tag 02GE}]{stacks-project}.
    Notice that this implies that the map $\inertiaDMtame(X) \to \loopstack X$ is an open immersion on truncations.
    
    For the second statement, since any element of the stabilizers of $X$ has finite order coprime to $\mathsf{char}(k)$, the map $\iota_\varpi^{\tame} \colon \inertiaDMtame(X) \to \loopstack X$ is an effective epimorphism. The conclusion follows.
\end{proof}

\section{HKR isomorphisms  for derived \DM stacks}
\label{HKRDMchar0} 

In this section, we establish general HKR theorems for derived \DM stacks extending earlier results of Arinkin--\caldararu--Hablicsek \cite{Arinkin_Caldararu_Hablicsek}. We obtain HKR theorems for Hochschild homology and for Hochschild cohomology. Applications to cyclic, negative cyclic and periodic cyclic homology will be discussed in Section~\ref{sec:GroupStructure}.

%The strategy is as follows:
%there is a natural morphism
%\[ S^1 \longrightarrow \Aff(S^1) \times \hatcircle , \]
%and we prove that for a tame derived \DM stack $X$ locally almost of finite presentation, when applying $\bfMap(-, X)$ to the above morphism, the induced morphism 
%\[ \bfMap(\Aff(S^1), \inertiaDMtame X)\simeq \bfMap(\Aff(S^1) \times \hatcircle, X)\xrightarrow{\simeq} \bfMap(S^1, X) \simeq \loopstack X \]
%is formally \'etale and is an isomorphism on the classical truncations, hence is an equivalence.
%In characteristic zero, we also establish a canonical identification
%\[ \bfMap(\Aff(S^1) \times \hatcircle, X) \simeq \inertiaDM X \]
%\textcolor{blue}{hence  an equivalence $$\loopstack X\simeq \mathsf T[-1] \inertiaDM X\ $$
 %}

\subsection{A twisted sector decomposition for the loop space of a derived DM stack}

We establish in Theorem \ref{thm:HKR-Geometric} a decomposition of the loop space of a tame derived DM stack indexed by the components of its tame orbifold inertia. This is one of our main results, and can be viewed as a non-linear version of the HKR equivalence for tame DM stacks that holds in arbitrary characteristics. When $k$ is a $\mathbb{Q}$-algebra, this will imply a more familiar statement comparing the loop space of a DM stack and the shifted tangent of its inertia, see Theorem \ref{thm:HKR_DM} in the next section.

%\begin{notation}
Let us write simply $\aff$ instead of $\aff_{S^1}$ for the affinization morphism 
    $  \, 
    S^1\to \Aff(S^1) 
    $ (see Recollection~\ref{recollection:affinization}). 
    Combining this with the morphism $\varpi_{\tame} \colon S^1 \to \hatcircletame$ (see \eqref{eqn:S1toS1hat}), we obtain a morphism
    \[ \hataff \coloneqq \aff \times \varpi_{\tame} \colon S^1\to \Aff(S^1) \times \hatcircletame \ . \]
    For $r \geqslant 1$, we denote by $\hataff_r$ the natural morphism $(\aff, \varpi_r) \colon S^1 \longrightarrow \Aff(S^1) \times \BCr$.
%\end{notation}

\begin{thm}[HKR for derived DM stacks: arbitrary characteristic]
\label{thm:HKR-Geometric}
    Let $k$ be a commutative ring. 
    Let $X$ be a \emph{strongly tame} derived \DM stack locally alomost of finite presentation over $k$. 
    The morphism $\hataff$  induces an equivalence of derived stacks over $X$:
    \begin{equation}
        \hataff^* \colon \bfMap(\Aff(S^1),\inertiaDMtame(X))\simeq \bfMap(\Aff(S^1) \times \hatcircletame,X) \longrightarrow \bfMap(S^1,X)\simeq \loopstack X \ . 
    \end{equation}
\end{thm}

\begin{prop}\label{prop:HKR_DM_etale}
	Fix a positive integer $r>0$ that is coprime to the characteristic of $k$.  Write $p \colon S^1 \to \Spec(k)$ and $$q_r \colon \Aff(S^1) \times \BCr \longrightarrow  \Spec(k)$$ for the structural morphisms.
	The $+$-Beck-Chevalley transformation
	\[ p_+ \hataff_r^\ast \longrightarrow q_{r,+} \ . \]
	induced by $\hataff_r \colon S^1 \longrightarrow \Aff(S^1)\times \BCr$ is an equivalence.
\end{prop}

\begin{proof}
	By Lemma \ref{lemma:CatPerfct-AffS^1}, $\Aff(S^1)$ is $\otimes$-universal, hence
	\[ \QCoh(\Aff(S^1) \times \BCr) \simeq \QCoh(\Aff(S^1)) \otimes \QCoh(\BCr) \ . \]
    As both source and target of the $+$-Beck--Chevalley transformation commute with colimits, it suffices to prove that the $+$-Beck--Chevalley transformation is an equivalence on $\cF \boxtimes \cG$ for any $\cF \in \Perf(\Aff(S^1))$ and $\cG \in \Perf(\BCr)$.  
	Fix $\cG \in \Perf(\BCr)$. 
	Thanks to Lemma \ref{lemma:CatPerfct-AffS^1}, $\cO_{\aff(S^1)}$ is a compact generator, hence we are reduced to the case where $\cF=\mathcal{O}_{\Aff(S^1)}$. 
    
    Now we prove that
	\[ p_+ \hataff_r^\ast( \cO_{\Aff(S^1)} \boxtimes \cG ) \longrightarrow q_{r,+}(\cO_{\Aff(S^1)} \boxtimes \cG ) \]
	is an equivalence, by direct computation as follows. 
    As $r$ is coprime to the characteristic of $k$, $\QCoh(\BCr)$ is a semisimple category, therefore, we can assume that $\cG$ is an irreducible representation of $\mathsf C_r$.
    It is also enough to prove the statement after a Galois extension $k \to k'$, adding all the $r$-roots of units to $k$.
    Thus, using the decomposition of $\QCoh(\BCr)$ given in equation \eqref{eq:character_decomposition_BCr}, we can assume that $\cG = k_\zeta$ is the irreducible $\mathsf C_r$-representation associated to the character $\zeta \in \mu_r$ of $\mathsf C_r$.
    In this case, $\hataff_r^\ast(\cO_{\Aff(S^1)} \boxtimes \cG) \simeq \varpi_r^\ast(k_\zeta)$ is the local system on $S^1$ associated to $k$ with monodromy given by multiplication by $\zeta$, and
    \[ p_{+}( \varpi_r^\ast(k_\zeta)) \simeq \cofib( 1 - \zeta \colon k \to k ) \simeq \begin{cases}
        k \oplus k[1] & \text{if } \zeta = 1 \\
        0 & \text{otherwise.}
    \end{cases}  \]
    In both cases, this matches $q_{r,+}(\cO_{\Aff(S^1)} \boxtimes k_\zeta)$, whence the conclusion.
\end{proof}

\begin{proof}[Proof of Theorem \ref{thm:HKR-Geometric}]
We prove the theorem in two steps: first, we show that $\hataff^*$ is formally \'etale; second, we show  that $\hataff^*$ induces an equivalence on the classical truncations.

For the first step, as $X$ is locally almost of finite presentation, so is $\inertiaDMtame X$ by Theorem \ref{thm:basics_inertia}.
Unraveling the definition of $\hatcircletame$, to show that $\hataff^\ast$ is formally étale we have to check that
\[ \hataff_r^*\colon \bfMap(\Aff(S^1), \inertiarth(X))=\bfMap(\Aff(S^1) \times \BCr, X) \longrightarrow \bfMap(S^1, X)=\loopstack X \]
is formally \'etale for each $r$ coprime to the characteristic of $k$. 
To this end, we invoke the criterion in Proposition \ref{prop:universal_criterion_formally_etale}, and conclude by Proposition \ref{prop:HKR_DM_etale}.

For the second step, observe that we can factor $\hataff^\ast$ as the following composition
\begin{center}
    \begin{tblr}{Q[c]Q[c]Q[l]Q[l]}
        $\bfMap(\Aff(S^1), \inertiaDMtame(X))$ & $\simeq$ & $\bfMap(\mathsf B\mathsf{Fix}, \inertiaDMtame(X))$ & by Example~\ref{eg:Betti} \\
        & $\longrightarrow$ & $\inertiaDMtame(X)$ & induced by $\Spec(\Z) \to \mathsf B\mathsf{Fix}$ \\
        & $\longrightarrow$ & $\loopstack(X)$ & induced by $\varpi_\tame \colon S^1 \to \hatcircletame$ \\
    \end{tblr}
\end{center}
Recall from Theorem~\ref{thm:basics_inertia}-(4) that $\inertiaDMtame(X)$ is a strongly tame derived \DM stack and from Lemma~\ref{lemma:AffS^1=BG} that $\mathsf{Fix}$ is an affine fiberwise-connected group scheme over $k$.
Thus, we can apply Proposition~\ref{prop:mapping_from_B_connected} (see Remark~\ref{rmk:ConditionToApply}) and deduce that the morphism on the second line is an isomorphism on truncations.
On the other hand, Theorem~\ref{thm:basics_inertia}-(4) guarantees moreover that the morphism on the third line is an equivalence on the truncations, whence the conclusion.

\end{proof}

\subsection{Multiplicative HKR isomorphism for Hochschild homology}
\label{HKR isomorphism for Hochschild homology}
 When the base ring is of characteristic 0, using Example \ref{lemma:AffS^1=BGa} we can  restate Theorem \ref{thm:HKR-Geometric}  in the following more usual form:

\begin{thm}[HKR for derived DM stacks: characteristic 0]
\label{thm:HKR_DM}
    Let $X$ be a derived \DM stack locally almost of finite presentation over a commutative $\mathbb{Q}$-algebra.
    The maps
    \[ \mathbb D_{-1} \times \hatcircle \longrightarrow \BGa \times \hatcircle \longleftarrow S^1 \]
    induce equivalences
    \[ \mathsf T[-1] \inertiaDM(X) \longleftarrow \bfMap(\BGa, \inertiaDM(X)) \longrightarrow \loopstack X \]
    of derived stacks over $X$.
    In particular, $\mathsf T[-1]\inertiaDM(X) \simeq \loopstack X$ as a stack over $X$.
\end{thm}
\begin{proof}
    We combine Theorem~\ref{thm:HKR-Geometric} with \eqref{eqn:Formality-weak} in Example~\ref{lemma:AffS^1=BGa}.
    %and Corollary~\ref{cor:unipotent_loops_DM}.
\end{proof}

Recall that for a derived \DM stack $X$, its \textit{Hochschild homology} is defined as the following complex of $k$-modules, viewed as an object in the stable $\infty$-category $\Mod_k$:
\begin{equation}
	\HH_*(X)\coloneqq\Gamma(X\times X, \Delta_*\mathcal{O}_X\otimes \Delta_*\mathcal{O}_X),
\end{equation}
where $\Delta\colon X\to X\times X$ is the diagonal morphism. 
%As it happens for schemes, 
Derived base change for the derived pullback square
\begin{equation*}
    \xymatrix{
    \loopstack X \ar[d]_{p} \ar[r]^{p}& X \ar[d]^{\Delta}\\
    X\ar[r]_-{\Delta} & X\times X\ .
    }
\end{equation*}
yields the usual interpretation of $\HH_\ast(X)$ as functions on $\loopstack X$:
\[ \HH_{*}(X) \simeq \Gamma(X, \Delta^*\Delta_*\mathcal{O}_X) \simeq \Gamma(\loopstack X, \mathcal{O}_{\loopstack X}). \]
In particular, it inherits an algebra structure coming from the right hand side.
In more classical terms, for any integer $i$, the $i$-th \textit{Hochschild homology group} of $X$ is defined as
\begin{equation*}
	\HH_i(X)\coloneqq H^{-i}(\HH_*(X))\simeq H^{-i}(X\times X, \Delta_*\mathcal{O}_X\otimes \Delta_*\mathcal{O}_X)\simeq H^{-i}(\loopstack X, \mathcal{O}_{\loopstack X}).
\end{equation*}
Consequently, $\bigoplus_i \HH_i(X)$ admits a natural graded algebra structure.

\begin{thm}
\label{thm:HKR-HH} 
Let $X$ be a derived Deligne--Mumford stack locally almost of finite presentation over  $k$. Then we have an equivalence in the $\infty$-category $\mathsf{Alg}(\Mod_k)$: 
\begin{equation}
	\HH_{*}(X) \simeq \Gamma(\inertiaDM X, \Sym(\bbL_{\inertiaDM X}[1])).
\end{equation}
where the Hochschild homology $\HH_{*}(X)$ is equipped with its natural algebra structure as derived functions on $\loopstack X$ and $\Gamma(\inertiaDM X, \Sym(\bbL_{\inertiaDM X}[1]))$ is equipped with the natural algebra structure induced from the algebra structure on the symmetric algebra.
\end{thm}
\begin{proof}
	Write $\pi\colon \mathsf T[-1] \inertiaDM X\to\inertiaDM X$ for the natural projection.
    Since the equivalence of \Cref{thm:HKR_DM} is over $X$, it induces a canonical equivalence $\pi_*(\mathcal{O}_{\mathsf T[-1] \inertiaDM X})\simeq \Sym(\bbL_{\inertiaDM X}[1])$ in $\mathsf{Alg}(\QCoh(X))$.
    Applying the lax-monoidal $\infty$-functor of global sections, we conclude.
    %Thanks to \Cref{thm:HKR_DM}, we have an equivalence of derived stacks over $X$:
	%\begin{equation}
	%	\mathsf T[-1] \inertiaDM X \xrightarrow{\simeq}  \loopstack X.
    %\end{equation}
    %Therefore, by taking (derived) functions, we have an isomorphism of algebras:
    %\begin{equation}
    % \Gamma (\loopstack X, \mathcal{O}_{\loopstack X})\simeq \Gamma(\mathsf T[-1] \inertiaDM X, \mathcal{O}_{\mathsf T[-1] \inertiaDM X}).
    %\end{equation}
    %We conclude by noting that $\pi_*(\mathcal{O}_{\mathsf T[-1] \inertiaDM X})\simeq \Sym(\bbL_{\inertiaDM X}[1])$  as algebras, where $\pi\colon \mathsf T[-1] \inertiaDM X\to\inertiaDM X$ is the natural projection.
\end{proof}

\subsection{HKR isomorphism for Hochschild cohomology}
\label{HKR isomorphism for Hochschild cohomology}
\Cref{thm:HKR_DM} also allows us to obtain a statement for Hochschild \emph{co}homology.
To formulate it let us denote by
\[ p \colon \loopstack X \longrightarrow X \ , \qquad q \colon \mathsf T[-1]\inertiaDM X \longrightarrow X \]
the natural projections.
Since both $p$ and $q$ are isomorphisms on the truncation, they are proper and therefore the induce adjunctions
\[ p_\ast \colon \IndCoh(\loopstack X) \leftrightarrows \IndCoh(X) \colon q^! \ , \qquad q_\ast \colon \IndCoh( \mathsf T[-1]\inertiaDM X ) \leftrightarrows \IndCoh(X) \colon q^! \ ,  \]
see \cite[\S3.3.7]{Gaitsgory_IndCoh}.
Given $\cF \in \IndCoh(X)$, we set
\[ \mathcal{HH}^\ast(X;\cF) \coloneqq p_\ast p^!(\cF) \ , \]
and we refer to it as the \textit{Hochschild cohomology sheaf of $X$ with coefficients in $\cF$}, whose (derived) global sections recover the \textit{Hochschild cohomology of $X$ with coefficients in $\cF$}:
\[ \mathsf{HH}^*(X;\cF)\coloneqq \Gamma(X, \mathcal{HH}^*(X;\cF))=\Gamma(\loopstack X, p^!\cF).\]

\begin{notation}
    Assume that $X$ has finite tor-amplitude over the classical base ring $k$.
    Then $\cO_X$ is bounded and therefore $\cO_X \in \Cohb(X) \subset \IndCoh(X)$.
    In this case, we simply set
    \[ \mathcal{HH}^\ast(X) \coloneqq \mathcal{HH}^\ast(X;\cO_X) \qquad \text{and} \qquad \mathsf{HH}^\ast(X) \coloneqq \mathsf{HH}^\ast(X;\cO_X) \ . \]
\end{notation}

\begin{cor}
\label{Hochcoh}
    Let $X$ be a derived DM stack locally almost of finite presentation over a commutative $\mathbb{Q}$-algebra $k$.
    For every $\cF \in \IndCoh(X)$ there is a canonical equivalence
    \[ \mathcal{HH}^\ast(X;\cF) \simeq q_\ast q^!(\cF)\ , \]
    and a canonical isomorphism in the stable $\infty$-category $\Mod_k$
    \[\mathsf{HH}^\ast(X;\cF) \simeq \Gamma(\mathsf T[-1] \inertiaDM X,q^!(\cF))\ . \]
    Moreover if $X$ is smooth then we have canonical isomorphisms
    \begin{align}
        \label{eqn:HHForLCI}
        \begin{split}
             \mathsf{HH}^\ast(X)&\simeq \Gamma (\inertiaDM X, \Sym(\mathbb{L}_{\inertiaDM X}[1])\otimes i^*\omega_X^\vee[-\dim(X)])\\
        &\simeq \Gamma (\inertiaDM X, \Sym(\mathbb{T}_{\inertiaDM X}[-1])\otimes \det \mathsf{N}[-c]),
        \end{split}
    \end{align}
    where $\omega_X$ is the dualizing sheaf of $X$, $i\colon \inertiaDM X\to X$ is the canonical morphism, 
    $$\mathsf N\coloneqq\cofib(\mathbb{T}_{\inertiaDM X}\to i^*\mathbb{T}_X) \ , \quad  c\coloneqq\dim(X)-\dim(\inertiaDM X) \ , $$
    and $\dim(-)$ is the dimension viewed as a locally constant $\mathbb{Z}$-valued function. 
\end{cor}
\begin{proof}
    Since $p_\ast$ and $q_\ast$ are canonically identified by \cref{thm:HKR_DM}, we immediately obtain the first assertion. Applying $\Gamma$ yields the second assertion. 
    Assume now that $X$ is a smooth DM stack of finite presentation. By \Cref{thm:basics_inertia}, $\inertiaDM X=\inertiarth X$ for some $r$, which is also smooth. 
    Using \cite[Chapter 9, Corollary 7.3.2 and Proposition 7.3.4]{Gaitsgory_Rozenblyum_II}, we compute that 
    $$p^!\cO_X\simeq \det(\mathbb{L}_p)[\operatorname{vdim}(p)]\simeq p^*\det(\mathbb{L}_{\Delta})[-\dim(X)]\simeq p^*\omega_X^\vee[-\dim(X)].$$ Therefore, by Theorem \ref{thm:HKR_DM},
    $$q^!\cO_X\simeq q^*\omega_X^\vee[-\dim(X)].$$
    We decompose the morphism $q$ into the composition $$q\colon \mathsf T[-1]\inertiaDM X\xrightarrow{\pi} \inertiaDM X \xrightarrow{i} X.$$
    %The standard fiber sequence of cotangent complexes for the map $\pi$ shows that the dualizing sheaf $\omega_{\mathsf T[-1]\inertiaDM X}$ is trivial. As a result, the relative dualizing sheaf $\omega_{q}$ is isomorphic to $q^*\omega_X^\vee$ and $\dim(q)=-\dim(X)$. Therefore,
    %\[q^!\mathcal{O}_X\simeq \omega_q[\dim(q)]\simeq q^*\omega_X^\vee[-\dim(X)]\]
    We conclude by the following computation: 
    \begin{align*}
        \mathsf{HH}^\ast(X) &\simeq 
    \Gamma(\mathsf T[-1] \inertiaDM X,q^!(\cO_X))\\
    &\simeq \Gamma(\mathsf T[-1] \inertiaDM X, q^*\omega_X^\vee[-\dim(X)] )\\
    &\simeq \Gamma(\inertiaDM X, \pi_*\pi^*i^*\omega_X^\vee[-\dim(X)] )\\
    &\simeq \Gamma(\inertiaDM X, \Sym(\mathbb{L}_{\inertiaDM X}[1])\otimes  i^*\omega_X^\vee[-\dim(X)] ),
    \end{align*}
    where the last step uses the projection formula and the fact that $\pi_*\mathcal{O}_{\inertiaDM X}\simeq \Sym(\mathbb{L}_{\inertiaDM X}[1])$.
\end{proof}

In the classical case, Corollary \ref{Hochcoh} can be reformulated in a much more concrete way, generalizing  \cite[Corollary 1.17]{Arinkin_Caldararu_Hablicsek} to the case where $X$ is not necessarily a global quotient by a finite group. The following consequence is one instance:

\begin{cor}
\label{cor:HochschildCohomology}
    Assume that $X$ is a smooth (hence underived) DM stack over a commutative $\mathbb{Q}$-algebra $k$. Consider the connected-component decomposition of $\inertiaDM X = \inertia X$: 
$$
\mathsf I X = \bigsqcup_{i \in I} Z_i,
$$
where $I$ is the set of connected components of $\inertia X$. 
Let $c_i=\dim (X)-\dim (Z_i)$, and let $\omega_{{Z_i}/X}$ be the relative dualizing sheaf of the natural map $Z_i \to X$.
Then we have an isomorphism of graded vector spaces:
    $$
    \mathsf{HH}^\ast(X) \simeq \bigoplus_{i \in I} \bigoplus_{p + q= \ast} H^{p-c_i}\Big(Z_i, \bigwedge^q \mathsf T_{Z_i} \otimes \omega_{Z_i/X}\Big).
    $$
\end{cor}
\begin{proof}
    We apply \eqref{eqn:HHForLCI} and use the relative canonical bundle formula $\omega_X^\vee|_{Z_i}\simeq \omega_{Z_i}^\vee\otimes \omega_{Z_i/X}$ and the canonical isomorphism $\Omega_{Z_i}^p\otimes \omega_{Z_i}^\vee\simeq \bigwedge^{\dim(Z_i)-p}\mathsf{T}_{Z_i}$.
\end{proof}

\section{Examples and applications}
\label{sec:examples}

In this section, we apply our main results from the previous section to compute Hochschild (co)homology to various situations naturally arising in geometry.
Here's an analytic breakdown of the examples that will be considered:
\begin{enumerate}\itemsep=0.2cm
\item In \S\ref{subsec:GlobalQuotients} we generalize the main results of Arinkin--\caldararu--Hablicsek \cite{Arinkin_Caldararu_Hablicsek} by removing their smoothness assumption; in other words, we consider DM stacks arising as quotients of singular (and possibly derived) schemes by the action of a finite group.
An important phenomenon is highlighted in Example~\ref{example:FixedLociBecomeDerived}: in the singular but underived setting, the orbifold inertia might introduce an extra derived structure.

\item In \S\ref{subsec:InertiaOfMappingStacks} we analyze the orbifold inertia of mapping stacks; this enables us to present concrete examples where the inertia of a \emph{derived} DM stack can be computed explicitly.

\item In \S\ref{sec:examplesbeyond} we generalize \cite{Arinkin_Caldararu_Hablicsek} by dropping the global quotient assumption. We focus on much studied classical examples, i.e.  weighted projective lines and more generally root stacks.

\item Finally in \S\ref{subsec:QuotientByAlgGroups} we give a complete treatment of DM stacks that can be presented as global quotients of algebraic (as opposed to finite) groups.
\end{enumerate}

\subsection*{Assumption} Throughout  this section  we work over a commutative $\mathbb{Q}$-algebra $k$.

\begin{rem}
    This assumption ensures that we can exploit Theorem~\ref{thm:HKR-HH} and its consequences for Hochschild (co)homology.
    It is nevertheless not entirely needed for the geometric characterizations of $\inertiaDM X$ given in Proposition~\ref{prop:InertiaGlobalQuotient} and in Lemma~\ref{lem:inertiamapping}, with the caveat that one should use the tame variant of the construction.
    As this is somehow a minor point, we leave the details to the interested readers.
\end{rem}

\subsection{Global quotients: possibly derived and singular}
\label{subsec:GlobalQuotients}
We compute the orbifold inertia of the global quotient of a derived scheme by a finite group, in terms of the (derived) fixed loci of this action. Fixed loci for classical stacks have been studied by Romagny \cite{Romagny-GroupAction}. The derived structure leads to some subtleties in the definition of fixed locus. Let us first briefly clarify this. We refer to \cite[Appendix A]{AKLPR} for more details (but beware  we denote by $Y^G$ what is denoted by $Y^{hG}$ in \textit{loc.~cit}).

\begin{defin}[Derived fixed loci]
\label{genuinefixed}
Let $Y$ be a derived stack equipped with an action of an algebraic group $G$. Hence we have a structural morphism $[Y/G]\to \mathsf{B}G$. The \emph{derived fixed locus} of the $G$-action on $Y$ is defined as the following \textit{section stack}:
\[ Y^G \coloneqq \mathbf{Sect}_{\mathsf BG}([Y/G]) \coloneqq \Spec(k) \times_{\bfMap(\mathsf BG, \mathsf BG)} \bfMap(\mathsf BG, [Y/G]) \ , \]
where $\Spec(k)\to \bfMap(\mathsf BG, \mathsf BG)$ is given by the identity of $\mathsf BG$, and the mapping stacks are computed in the category of derived stacks.

Let $g$ be a finite-order automorphism of a derived stack $Y$. Let $\langle g \rangle$ be the finite cyclic group generated by $g$. 
Then the \textit{derived fixed locus} of $g$ on $Y$,
denote by $Y^g$,
is defined to be the  derived fixed locus of the action of $\langle g\rangle$ on $Y$:
$$Y^g\coloneqq Y^{\langle g \rangle}=\mathbf{Sect}_{\mathsf B\langle g\rangle}([Y/\langle g\rangle]).$$ 
\end{defin}

\begin{rem}[Weil restriction]
    Let $p\colon \mathsf BG\to \Spec(k)$ be the structural morphism. The base-change functor $p^*\colon \dSt_k \to \dSt_{\mathsf BG}$ admits a right adjoint functor $p_* \colon\dSt_{\mathsf BG}\to \dSt_k$, called the \textit{Weil restriction} (\cite[\S 19.1.2]{SAG}). 
    For any derived stack $Y$ equipped with an action of $G$, we get an object $[Y/G]\in \dSt_{\mathsf BG}$.
    By \cite[Proposition 19.1.2.2 \& Construction 19.1.2.3]{SAG}, the derived fixed loci $Y^G$ can be equivalently interpreted via the Weil restriction:
    \begin{equation}
        Y^G\simeq p_*([Y/G]).
    \end{equation}
    Using atlas of $\mathsf BG$ and of $[Y/G]$, one can show that 
    \begin{equation}
       Y^G \simeq \lim_{n \in \mathbf \Delta} \bfMap(G^n, Y) \ .
    \end{equation}
    Notice that the action of $G$ on $Y$ on the RHS is entirely encoded in the transition maps of the diagram.
\end{rem}

\begin{rem}[Residual action by centralizer]
    Let the notation be as in \Cref{genuinefixed}. For any $g\in G$, let $Z(g)\coloneqq Z_G(g)$ be the centralizer of $g$ in $G$. Then the derived fixed locus $Y^g$ admits a canonical residual action of $Z(g)$. Indeed, the action of $Z(g)$ on $Y$ clearly descends to an action on $[Y/\langle g\rangle]$ that respects the structural morphism $[Y/\langle g\rangle]\to \mathsf B\langle g\rangle$. Hence $\bfMap(\mathsf B\langle g\rangle, [Y/\langle g\rangle])$ admits a $Z(g)$-action that respects the structural morphism to $\bfMap(\mathsf B\langle g\rangle, \mathsf B\langle g\rangle)$. By definition of $Y^g$ as the section stack, we get a $Z(g)$-action on it.
\end{rem}

\begin{rem}[Conjugation action]
\label{rem:ConjugationAction}
Let $G$ be a finite group acting on a derived stack $Y$. For any $g, h\in G$, it is clear from the definition that we have a canonical isomorphism 
\begin{equation}
    h.\colon Y^g \xrightarrow{\simeq} Y^{hgh^{-1}},
\end{equation}
which is given by $x\mapsto  h.x$ on the level of functor of points.
These isomorphisms assemble into the so-called \textit{conjugation} action of $G$ on $\bigsqcup_{g\in G} Y^g$, where $G$ acts on the indexing set by conjugation.
\end{rem}

 \begin{prop}
 \label{prop:InertiaGlobalQuotient}
 Let $Y$ be a finitely presented derived scheme equipped with an action of a finite group $G$. Then we have isomorphisms
 $$\inertiaDM [Y/G] \simeq \bigsqcup_{[g] \in G/G} [Y^g/Z(g)]\simeq \left[\left(\bigsqcup_{g\in G}Y^g \right)/ G\right],$$
 where $G/G$ denotes the set of conjugacy classes of $G$, and in the last stack, the $G$-action is the conjugation action of \Cref{rem:ConjugationAction}.
  \end{prop} 
\begin{proof}
The second isomorphism is clear, let us prove the first isomorphism. 
Let $g\in G$ be an element of order $r$. We have pullback diagrams 
\begin{equation}
\label{diagramfixed}
\begin{gathered}
\xymatrix{ Y^g  \ar[r] \ar[d]& \mathbf{Map}( \mathsf B \langle g \rangle,  [Y/ \langle g \rangle ]) \ar[d] \ar[r] &  \mathbf{Map}(\mathsf B \langle g \rangle,  [Y/ Z(g)  ]) \ar[d]\ \\
\Spec(k)  \ar[r]^-{\id} & \mathbf{Map}(  \mathsf B \langle g \rangle , \mathsf B\langle g \rangle ) \ar[r] & \mathbf{Map}(\mathsf B \langle g \rangle , \mathsf B Z( g ) )}
\end{gathered}
\end{equation}
The left square is a pullback square is by definition. The right square is a pullback square since it is obtained by applying  
$\mathbf{Map}(\mathsf B \langle g \rangle,  -)$ to the following pullback square:
\begin{equation}
\label{diagramfixed}
\begin{gathered}
\xymatrix{ [Y/ \langle g \rangle ] \ar[d] \ar[r] &    [Y/ Z(g)]    \ar[d]  \\
 \mathsf B\langle g \rangle \ar[r] &  \mathsf B Z( g ) }
\end{gathered}
\end{equation}
 Now since the cotangent complex of $\mathbf{Map}( \mathsf B \langle g \rangle, \mathsf B Z(g) )$ is trivial, the bottom horizontal map $$\Spec(k) \rightarrow \mathbf{Map}(\mathsf B \langle g \rangle, \mathsf B Z(g))$$ is \'etale. Therefore the upper horizontal  map  $Y^g  \rightarrow  \mathbf{Map}(\mathsf B \langle g \rangle, [Y/Z(g)])$ is also \'etale. 

Now note that the top map factors through the quotient $[Y^g/Z(g)]$. This follows from the fact that $\mathrm{pt} \to  \mathbf{Map}( \mathsf B \langle g \rangle , \mathsf B Z( g ) )$ factors through $\mathsf B Z(g)$. Hence  $$[Y^g/Z(g)] \to  \mathbf{Map}(\mathsf B \langle g \rangle, [Y/ Z(g)  ])  \to \inertia^{(r)} [Y/G]$$  is also \'etale, where $r=|g|$. In order to conclude the proof we only have to show that the above morphisms induce an isomorphism between the truncation of $\inertiaDM [Y/G] $ and $ \bigsqcup_{[g] \in G/G} [\trunc (Y)^g/Z(g)].$ This follows from the following computation 
$$\trunc(\inertiaDM [Y/G] ) \simeq [\trunc Y/G] \times_{ [\trunc Y/G] \times[ \trunc Y/G]} [\trunc Y/G]  \simeq \inertia [\trunc Y/G] \simeq   \bigsqcup_{[g] \in G/G} [\trunc (Y)^g/Z(g)],$$
where we used that the fiber products commute  with  truncation functor $\trunc$.
\end{proof}

The graded $k$-module
$ \ 
\bigoplus_{i \in \mathbb{Z}} \HH_{-i}([Y/G]) 
 $, 
obtained by taking the cohomology of the algebra object $\HH_*([Y/G])$, inherits a natural structure of graded $k$-algebra. The following corollary recovers and generalizes \cite[Corollary 1.7]{Arinkin_Caldararu_Hablicsek}. 
\begin{cor}
\label{cor:HH-GlobalQuotientDerived}
    Let $Y$ be a finitely presented derived scheme equipped with an action of a finite group $G$. Then 
    \begin{enumerate}
        \item 
    We have an isomorphism of graded algebras
    \begin{align*}
                \bigoplus_{i \in \mathbb{Z}} \HH_{-i}([Y/G])&\simeq \bigoplus_{i\in \mathbb{Z}}\left(\bigoplus_{g\in G}\bigoplus_{q-p=i}H^q(Y^g, \Omega^p_{Y^{g}})\right)^G\\
                &\simeq \bigoplus_{i\in \mathbb{Z}}\bigoplus_{[g]\in G/G} \bigoplus_{q-p=i}H^q(Y^g, \Omega_{Y^g}^p)^{Z(g)}.
    \end{align*}
    where $\Omega^p\coloneqq\bigwedge^p\mathbb{L}$ stands for the $p$-th term of the \emph{derived} de Rham complex.
    \item If moreover $Y$ is smooth, we have an isomorphism of vector spaces:
    \begin{equation*}
        \HH^i([Y/G])\simeq \bigoplus_{[g]\in G/G}\bigoplus_{q+p=i}H^q(Y^g, \bigwedge^p \mathsf{T}_{Y^g}\otimes \det(\mathsf N_g)[-c_g])^{Z(g)},
    \end{equation*}
    where $\mathsf N_g$ is the normal bundle $\mathsf N_{Y^g/Y}$ and $c_g$ is its rank; in other words, $\mathsf N_g[-1]$ denotes the relative tangent complex of the canonical morphism $Y^g\to Y$.
    
    \end{enumerate}
\end{cor}
\begin{proof}
    We apply \Cref{thm:HKR-HH} and \Cref{Hochcoh} (or \Cref{cor:HochschildCohomology}) to $X=[Y/G]$ and combine with \Cref{prop:InertiaGlobalQuotient}.
\end{proof}

We have the following somewhat surprising observation: even for an underived global quotient DM stack $X$, when $X$ is singular, its orbifold inertia $\inertiaDM X$ can carry a non-trivial derived structure. Thus, even in the classical setting, the failure of  being an isomorphism for the natural map $\inertia X \to \inertiaDM X$ contains information on the singularities of the stack. We give here a simple example to illustrate this point:

\begin{eg}
    \label{example:FixedLociBecomeDerived}
    Consider the underived affine scheme $Y \coloneqq \Spec(A)$ where 
    \[ A \coloneqq k[x,y]/ (xy) \ . \]
    Let $G \coloneqq \mathbb{Z}/2 \mathbb{Z} = \{1, \varphi\}$ and consider the  action on $Y$ given  by $\varphi \cdot \overline{x} \coloneqq -\overline{x} $ and $\varphi \cdot \overline{y} \coloneqq -\overline{y}$.
    We claim that $Y^G$ has a non-trivial derived structure.
    In virtue of \cref{prop:InertiaGlobalQuotient}, $\inertiaDM([Y/G])$ also has a non-trivial derived structure.

    \medskip
    
    Recall that since $Y$ is affine, the truncation of $\trunc(Y^G)$ coincides with the spectrum of $A_G$, the (classical) $G$-coinvariants of $A$. A simple computation shows that, as $\operatorname{char}(k)\neq 2$,
    \[ A_G= A/(\overline{x}-(-\overline{x}), \overline{y}-(-\overline{y})) = A/(\overline{x}, \overline{y}) \simeq  k \ . \]
    Hence $\trunc(Y^G)$ is isomorphic to $\Spec (k)$, which corresponds to the unique (reduced) singular point $p_0$ of $Y$.
    We are going to show that $Y^G$ is not equivalent to  $\trunc(Y^G)$, hence it must be derived.

    \medskip

    Notice that the following square, where $f(x, y)\coloneqq xy$,
    \[ \begin{tikzcd}
        Y \arrow{r}{j} \arrow{d} & \mathbb A^2_k \arrow{d}{f} \\
        \Spec(k) \arrow{r}{0} & \mathbb A^1_k
    \end{tikzcd}\]
    is a derived pullback square by the flatness of $f$.
    In particular,
    \[ \mathbb L_Y \simeq \cofib( \delta \colon j^\ast f^\ast \mathbb L_{\mathbb A^1_k} \to j^\ast \mathbb L_{\mathbb A^2_k} ) \ . \]
    Now, $j^\ast f^\ast \mathbb L_{\mathbb A^1_k} \simeq \cO_Y \mathsf dt$ and $j^\ast \mathbb L_{\mathbb A^2_k} \simeq \cO_Y \mathsf dx \oplus \cO_X \mathsf dy$, where $t$ denotes the affine coordinate on $\mathbb A^1_k$, and $x$ and $y$ denote the affine coordinates in $\mathbb A^2_k$.
    The map $\delta$ between them is induced by $\mathsf dt \mapsto \mathsf d(xy) = \overline{x} \mathsf dy + \overline{y} dx $.
    Notice that $\delta$ is injective, but its cofiber is not locally free.
    The fiber at the singular point $p_0$ is the zero map
    \[ k \mathsf dt \longrightarrow k \mathsf dx \oplus k \mathsf dy \ . \]
    Thus,
    \[ \pi_1( p_0^\ast \mathbb L_Y ) \simeq  k \mathsf dt \ . \]
    
    \medskip

    Observe now that letting $G$ act on $\mathbb A^2$ by multiplication by $-1$ on both coordinates and trivially on $\mathbb A^1$, the above diagram becomes $G$-equivariant.
    In particular, $\mathbb L_Y$ acquires a canonical $G$-equivariant structure. It follows from \cite[Corollaries A.27 and A.30]{AKLPR} that
    \[ \mathbb L_{Y^G} \simeq (\mathbb L_Y |_{Y^G})_G \ . \]
    Since $G$ acts trivially on $\mathsf dt$ by construction, the above computation allows to conclude that 
    \[ \pi_1( \mathbb L_{Y^G} ) \simeq k \mathsf dt \ , \]
    whereas $\mathbb L_{\trunc(Y^G)} = 0$. In particular, $Y^G$ has a non-trivial derived structure.
\end{eg}

In addition to \Cref{genuinefixed}, the following notion of fixed locus is also frequently used in derived algebraic geometry\ :
\begin{defin}[Naive derived fixed loci]
\label{def:DerivedFixedLoci}
   Let $Y$ be a derived stack and $g$ an automorphism of $Y$. We define the \textit{naive derived fixed locus} $Y^{\mathbb{R} g}$ via the (derived) pullback diagram
\begin{equation}
	\xymatrix{
		Y^{\mathbb{R} g} \ar[r] \ar[d] & Y \ar[d]^-\Delta \\
		Y \ar[r]^-{\Delta_g} & Y \times Y
	}
\end{equation}
where $\Delta_g$ is the $g$-twisted diagonal map (i.e.~the graph of the automorphism of $Y$ given by $g$), which is given by the following formula on functor of points:
\begin{equation}
	\Delta_g: Y \to Y \times Y, \quad \Delta_g(y) \coloneqq (y, gy).
\end{equation}
\end{defin}

The following lemma clarifies the relation between the derived fixed loci (\Cref{genuinefixed}) and the naive derived fixed loci (\Cref{def:DerivedFixedLoci})\ :
\begin{lem}\label{lemma: fixed points}
Let $Y$ de a finitely presented derived algebraic space equipped with a finite-order automorphism $g$. Then
the derived fixed locus $Y^g$ is a derived algebraic space, and we have 
$$
\mathsf{T}[-1]Y^g \simeq \loopstack Y^g \simeq Y^{\mathbb{R} g}.
$$
\end{lem} 
\begin{proof}
We give here a proof  that relies on our main result \Cref{thm:HKR_DM}. A direct proof would certainly be desirable. 
Let $r$ be the order of $g$, and let $G=\langle g\rangle$ be the cyclic group generated by $g$ acting naturally on $Y$.
By \Cref{thm:HKR_DM} and \Cref{prop:InertiaGlobalQuotient},
we have $$ \loopstack [Y/G] \simeq   \mathsf T[-1] \inertiaDM [Y/G] \simeq \mathsf T[-1] \bigsqcup_{k=1}^r [Y^{g^k}/G] \simeq  \bigsqcup_{k=1}^r [\mathsf T[-1] Y^{g^k}/G]$$ 

On the other hand, $$\loopstack [Y/G] \simeq [Y/G] \times_{[Y/G] \times [Y/G] } [Y/G]\simeq  \bigsqcup_{k=1}^r [Y^{\mathbb{R} g^k}/G].$$
Since the HKR isomorphism induces identity on the classical truncation, we obtain an equivalence 
\begin{equation}
    [\mathsf T[-1]Y^g/G]\simeq [Y^{\mathbb{R}g}/G].
\end{equation}
As $G$ is generated by $g$, hence the $G$-action on $Y^g$ is trivial. It follows that $\mathsf T[-1]Y^g\simeq Y^{\mathbb{R}g}$.

Since $Y^{\mathbb{R}g}$ is a derived algebraic space by construction, $\mathsf T[-1]Y^g$ is also a derived algebraic space. Therefore, $Y^g$, as a closed subscheme of $\mathsf T[-1]Y^g$, must be a derived scheme. Finally, applying \Cref{thm:HKR_DM} to $Y^g$, we obtain $\mathsf T[-1]Y^g\simeq \loopstack Y^g$.
\end{proof}

\subsection{Mapping stacks}
\label{subsec:InertiaOfMappingStacks} 
%If a derived DM stack arises as a mapping stack, its orbifold inertia is computable as a mapping stack into some orbifold inertia stack. We will deduce from this formula an explicit  presentation  of the derived orbifold inertia stack of some DM stacks with non-trivial derived structure. We remark that  \Cref{prop:InertiaGlobalQuotient} also applies to   derived DM stacks, and so in principle it can be used to do concrete  calculations;  however, it is formulated in terms genuine fixed point loci, which    are difficult to compute in general in the derived setting.

Let us consider a pair given by a derived stack $S$ and a derived DM stack $Y$ such  that  $X \coloneqq \bfMap(S, Y)$
is still a derived DM stack. Then we have the following simple formula for its derived orbifold inertia stack. 
\begin{lem} \label{lem:inertiamapping}
    Let $S, Y \in \dSt_k$ and assume that:
    \begin{enumerate}\itemsep=0.2cm
        \item $S$ is a finite colimit of quasi-compact Artin stacks;

        \item $Y$ is a derived \DM stack;

        \item the mapping stack $\bfMap(S,Y)$ is again a \DM stack.
    \end{enumerate}
    Then
    \begin{equation} \label{inertiamappingstack}
        \inertiaDM \bfMap(S,Y)  \simeq \bfMap(S, \inertiaDM Y) \ .
    \end{equation}
\end{lem}

\begin{proof}
    By definition we have that
    \begin{align*}
       \inertiaDM X  & = \colim_r \bfMap(\BCr, X) \\
       & \simeq \colim_r \bfMap(\BCr, \bfMap(S, Y) ) \\
       & \simeq \colim_r \bfMap(S, \bfMap(\BCr, Y) ) \ .
    \end{align*}
    Notice that the transition maps $\bfMap(\BCr, Y) \to \bfMap(\mathsf{BC}_{r'},Y)$ are open immersions thanks to Theorem~\ref{thm:basics_inertia}-(3).
    Since the colimit is filtered, assumption (1) reduces us to show that the colimit can enter inside $\bfMap(S,-)$ when $S$ is a quasi-compact Artin stack.
    This follows from \cite[Lemma~A.1-(2)]{Porta_Sala_Categorified}.
\end{proof}

There are many interesting  examples of derived DM stacks that arise in such a way.
For instance: 
\begin{enumerate}\itemsep=0.2cm
\item We can take $S$ to be $S^1$.  
Then 
$X =  \loopstack Y$ 
is also a derived DM stack. More generally,   we can take $S$ to be the Betti stack of any  finite CW complex. 
\item We can take $S$ to be $ \mathsf B \mathbb{G}_a$. In this case   
$X = \bfMap(\mathsf B \mathbb{G}_a, Y) \simeq  \mathsf T[-1] Y $, which is still a derived DM stack.
\item 
Examples where the source $S$ is an algebraic curve can also be given, on condition of introducing stability conditions, or restricting to selected components of the full mapping stack (such as the sub-stack of stable maps $\overline{\cM}_{g,n}(X)$ \cite{Schurg_Toen_Vezzosi_Determinant} or the stack of almost constant maps considered in \cite{sibilla2023equivariant}). 
\end{enumerate}

In these cases, Lemma \ref{lem:inertiamapping} often  allows us   to compute explicitly the derived orbifold inertia stack of a \emph{derived} DM stack.  We illustrate this point with an example.

\begin{eg}
    Let $Y = [V/G]$ be a global quotient DM stack, such that $V$ is smooth (hence underived) and $G$ is a finite group. We set $S=B \mathbb{G}_a$. Then we have that $X \simeq \mathsf T[-1] [V/G]$ is equivalent to $[\mathsf T[-1] V/G]$, where the action of $G$ on $\mathsf T[-1] V$ is induced from the $G$-action on $V$.
    Proposition \ref{prop:InertiaGlobalQuotient} gives a formula for  $\inertiaDM X$, but it requires to   calculate the derived fixed loci from first principles, a subtle task in the case of actions on derived schemes.  
    Thanks to Lemma~\ref{lem:inertiamapping} we can write 
    \[ \inertiaDM X \simeq \mathsf T[-1] ( \inertiaDM [V/G] ) \simeq \mathsf T[-1] ( \mathsf I [V/G] ) \simeq \bigsqcup_{[g] \in G/G} \mathsf T[-1][V^g/Z(g)] \]
    where the second equivalence comes from the fact that the derived orbifold inertia of a smooth classical DM stack coincides with the classical inertia stack. So in particular we can describe   $\inertiaDM X$ in terms of the \emph{classical} fixed loci of the $G$-action $V^g$, which are readily computable. 
\end{eg}

\subsection{Beyond global quotients (I) : weighted projective lines and root stacks}
\label{sec:examplesbeyond}

Already in the category of underived classical stacks, most DM stacks are not of the form $[Y/G]$ with $G$ a finite group acting on an algebraic space $Y$. Indeed, by \cite[Proposition 6]{Prill}, any complex variety with only quotient singularities (i.e.~\textit{orbifold} in the classical sense as in \cite{Satake-orbifold}) admits a canonical smooth DM stack structure with only non-trivial stabilizers in codimension $\geq 2$, and certainly most varieties with quotient singularities are not quotients of smooth varieties by finite groups.

Such non-global quotient examples arise naturally.  Perhaps the best-known examples are Thurston's football and teardop, which also happen to be the examples that led historically to the  notion of  orbifold. 

\begin{eg}[Thurston's football and teardrop]
\label{eg:Football_Teardrop}
Given two positive integers $p,q$ that are coprime to each other,
let $X=\mathbb{P}^1(p,q)=\operatorname{Proj}(k[x,y])$ be the weighted projective line with $\deg(x)=p, \deg(y)=q$. The underlying orbifold is often referred to as \textit{Thurston's football} (or \textit{teardrop} when $q=1$); see the pictures below. By \cite[Proposition 5.1 and Corollary 5.9]{Kai-Noohi}, $X$ is \textit{not} of the form of a global quotient of a scheme by a finite group.  

\begin{center}
%%%%%%%%%%%%%%Football picture%%%%%%%%
\begin{tikzpicture}[>=Latex, line cap=round, line join=round, scale=1.1]

\tikzset{
  conepoint/.style={circle, inner sep=1.2pt, draw=black, fill=black},
  label/.style={font=\footnotesize}
}

%====================================================
% Football S^2(p,q)
%====================================================
\begin{scope}[xshift=-3cm]
  % draw rugby-ball oval
  \draw[thick] (0,-1.8) .. controls (1.2,-0.9) and (1.2,0.9) .. (0,1.8)
               .. controls (-1.2,0.9) and (-1.2,-0.9) .. (0,-1.8);

  % Cone points top and bottom
  \coordinate (N) at (0,1.8);
  \coordinate (S) at (0,-1.8);
  \node[conepoint] at (N) {};
  \node[conepoint] at (S) {};

  % labels for cone angles
  \node[label] at ($(N)+(0,0.3)$) {$\mu_{p}$};
  \node[label] at ($(S)+(0,-0.3)$) {$\mu_{q}$};

  % caption
  \node[font=\small] at (0,-2.5) {Football $\mathbb{P}^1(p,q)$};
\end{scope}

%====================================================
% Teardrop S^2(p)
%====================================================
\begin{scope}[xshift=3cm]
  % Smooth teardrop outline (single Bézier loop)
  \draw[thick]
    (0,1.8)                % top (cone point)
    .. controls (1.3,1.0) and (1.6,-1.0) .. (0,-1.0)  % right side to bottom
    .. controls (-1.6,-1.0) and (-1.3,1.0) .. (0,1.8); % left side back to top

  % cone point at top
  \coordinate (C) at (0,1.8);
  \node[conepoint] at (C) {};

  % label
  \node[label] at ($(C)+(0,0.3)$) {$\mu_{p}$};

  \node[font=\small] at (0,-2.5) {Teardrop $\mathbb{P}^1(p,1)$};
\end{scope}

\end{tikzpicture}
%%%%%%%End of football picture%%%%%%%%%%%%
\end{center}
By \Cref{cor:InertiaDM-Smooth}, 
\begin{equation}
    \inertiaDM X\simeq \inertia X\simeq X\cup (\inertia(\mathsf{BC}_p\sqcup \mathsf{BC}_q))\simeq X \cup ([\mathsf{C}_p/\mathsf{C}_p]\sqcup [\mathsf{C}_q/\mathsf{C}_q]) \simeq X \sqcup   \bigsqcup^{p-1}\mathsf{BC}_p \sqcup \bigsqcup^{q-1}\mathsf{BC}_q
\end{equation}
Consequently, by \cref{thm:HKR-HH}, we have
$$\mathsf{HH}_0(X)\simeq H^0(X, \mathcal{O}_X)\oplus H^1(X, \mathbb L_X) \oplus H^0(\mathsf{BC}_p, \mathcal{O}_{\mathsf{BC}_p})^{p-1}\oplus H^0(\mathsf{BC}_q, \mathcal{O}_{\mathsf{BC}_q})^{q-1} \simeq k^{p+q},$$
and $\mathsf{HH}_i(X)=0$ for any $i\neq 0$ since $H^1(X,\mathcal{O}_X)=0$ and $H^0(X, \mathbb{L}_X)=0$.

Similarly, one can compute the Hochschild cohomology using \cref{cor:HochschildCohomology}: since $\omega_X|_{\mathsf {BC}_p}$ is $\chi_{\zeta_p}$, the 1-dimensional representation of $C_p$ given by $\zeta_{p}$, we have:
$$\mathsf{HH}^*(X) \simeq H^*(X, \mathcal{O}_X)\oplus H^*(X, \mathbb T_X[-1]) \oplus H^*(\mathsf{BC}_p, \chi_{\zeta_p^{-1}})^{p-1}\oplus H^*(\mathsf{BC}_q,  \chi_{\zeta_q^{-1}})^{q-1}.$$
As we are in characteristic zero, the last two summands vanish and $H^*(X, \mathcal{O}_X)\simeq H^*(\mathbb{P}^1, \mathcal{O})$. Hence,
\begin{align*}
        \mathsf{HH}^0(X)&\simeq H^0(X, \mathcal{O}_X)\simeq k\\
        \mathsf{HH}^1(X)&\simeq H^1(X, \mathcal{O}_X)\oplus H^0(X, \mathbb T_X)\simeq k\\
        \mathsf{HH}^2(X)&\simeq  H^1(X, \mathbb T_X) =0 \\
        \mathsf{HH}^3(X)& =0.
\end{align*}

\end{eg}

\begin{eg}[Weighted projective line]
    \label{eg:WeightedProjLine}
    A natural generalization of \Cref{eg:Football_Teardrop} is the so-called \textit{weighted projective line}, introduced by Geigle and Lenzing in \cite{Geigle-Lenzing_Weighted_Projective_Lines}. Given a collection of points $\underline{x}\coloneqq(x_1, \cdots, x_t)$ with $x_i\in \mathbb{P}^1$, and a collection of weights $\underline{p}\coloneqq(p_1, \cdots, p_t)$ with $p_i\in \mathbb{Z}_{\geq 2}$, the weighted projective line $X\coloneqq\mathbb{X}(\underline{x}, \underline{p})$, defined in \cite{Geigle-Lenzing_Weighted_Projective_Lines}, is obtained by attaching weight $p_i$ to the point $x_i$; we depict it as follows:

\begin{center}
\begin{tikzpicture}
  % Draw the line
  \draw[-] (-1,0) -- (6,0);

  % Define some points
  \foreach \x/\above/\below in {
    0/{$x_1$}/{$\mu_{p_1}$},
    1/{$x_2$}/{$\mu_{p_2}$},
    5/{$x_t$}/{$\mu_{p_t}$}
  }{
    % Draw point
    \fill (\x,0) circle (2pt);
    % Label above
    \node[above] at (\x,0) {\above};
    % Label below
    \node[below] at (\x,0) {\below};
  }
 \node[above] at (3,0) {$\cdots$};
\end{tikzpicture}

\vspace{0.5em}
{Weighted projective line}
\end{center}
As in \Cref{eg:Football_Teardrop}, by \Cref{cor:InertiaDM-Smooth}, we have 
\begin{equation}
     \inertiaDM X\simeq X \sqcup  \bigsqcup_{i=1}^t \bigsqcup^{p_i-1}\mathsf{BC}_{p_i}.
\end{equation}
Therefore, 
$$\mathsf{HH}_0(X)=H^0(X, \mathcal{O}_X)\oplus H^1(X, \mathbb L_X) \oplus \bigoplus_{i=1}^t H^0(\mathsf{BC}_{p_i}, \mathcal{O}_{\mathsf{BC}_{p_i}})^{\oplus (p_i-1)}=k^{(\sum_{i}p_i)-t+2},$$
and $\mathsf{HH}_i(X)=0$ for any $i\neq 0$.

Similar to \cref{eg:WeightedProjLine}, the Hochschild cohomology $X$ can be computed using  \cref{cor:HochschildCohomology} as follows:
\begin{equation}
\label{eqn:HHWPL}
    \mathsf{HH}^*(X) \simeq H^*(X, \mathcal{O}_X)\oplus H^*(X, \mathbb T_X[-1])=
    \begin{cases}
        k\oplus k^2[-1] &\quad \text{ if } t=1\\
        k \oplus k[-1]&\quad \text{ if } t=2\\
        k&\quad \text{ if } t=3\\
        k \oplus k^{t-3}[-2]&\quad \text{ if } t\geq 4\\
    \end{cases}
\end{equation}
\end{eg}
The computation \eqref{eqn:HHWPL} was previously obtained by Happel \cite{Happel-HHPiecewiseHereditaryAlg} using piecewise hereditary algebras. Our geometric method is apparently different, and can be easily adapted beyond the case of piecewise hereditary algebras, e.g.~to compute the Hochschild homology and cohomology of stacky curves (see \cite{Schremmer-HHofWPL-Masterthesis} for an indirect approach using \cite{Arinkin_Caldararu_Hablicsek}), or more generally, to root stacks, as we now turn to.

\begin{eg}[Root stacks]
\Cref{eg:Football_Teardrop} and \Cref{eg:WeightedProjLine} are one-dimensional instances of a general construction called \textit{root stack}. Root stacks are DM stacks carrying universal roots of Cartier divisors  (and more generally log structures) on schemes. These objects have been first introduced by Cadman \cite{cadman2007using}, and have been much studied since; see \cite{talpo2018infinite} and references therein for an overview of the subject. For concreteness, let us restrict to the case of the root stack of a scheme $X$  equipped with a single smooth Cartier divisor $D$. This is the same as the datum of a line bundle $\cO_X(D)$ on $X$ and a section $\sigma$, and is therefore classified by a morphism 
$$
X \to [\mathbb{A}^1/\mathbb{G}_m]
$$
The $n$-th root stack of the pair $(X,D)$ is denoted  $\sqrt[n]{(X, D)}$ and    is defined by the pullback square
\[ \xymatrix{\sqrt[n]{(X, D)} \ar[r] \ar[d]_-\pi & [\mathbb{A}^1/\mathbb{G}_m] \ar[d]^{(-)^n} \\
X \ar[r] & [\mathbb{A}^1/\mathbb{G}_m], }
\]
where the right vertical arrow is given by $z\mapsto z^n$. 
The stack $\sqrt[n]{(X, D)}$ can also be defined via a universal property. Indeed, the root stack is universal among stacks over $X$ carrying an $n$-th root of  $\cO(D)$ and of the section $\sigma$: we denote these universal roots by 
$$
\cO(D_n) \quad \text{and} \quad \sqrt[n]{\sigma} \in \cO(D_n)
$$
The zero-locus of $\sqrt[n]{\sigma}$ is the stacky divisor $D_n$, which fits in the following diagram
$$
\xymatrix{
D_n \ar[r] \ar[dr] & \widetilde{D_n} \ar[r] \ar[d] & \sqrt[n]{(X, D)} \ar[d] \\
& D \ar[r] & X
}
$$
The  square on the right is a fiber product; the divisor $D_n$
 is the reduction of $\widetilde{D_n}$, which is the pullback of $D$ to $\sqrt[n]{(X, D)}$; the map $D_n\to D$ is a $\mu_n$-gerbe. If the line bundle $\cO_X(D)$ admits an $n$-th root, then $\sqrt[n]{(X, D)}$ can be presented as a global $\mu_n$-quotient but this fails in general.    The teardrop $\mathbb{P}^1(p,1)$ is equivalent to the root stack $\sqrt[p]{(X, D)}$ when $X=\mathbb{P}^1$ and $D$ is equal to a point. 
 
The inertia stack of $\sqrt[n]{(X, D)}$ can be computed as
\begin{equation}
\label{eqn:InertiaOfRootStack}
    \inertiaDM  \sqrt[n]{(X, D)}  \simeq \sqrt[n]{(X, D)} \sqcup \big ( \bigsqcup_{i=1}^{n-1} D_n \big ).
\end{equation}
Thanks to \Cref{thm:HKR_DM}, we obtain a decomposition
\begin{equation}
\label{eq:LoopOfRootStack}
    \loopstack  \sqrt[n]{(X, D)}   \simeq  \mathsf T[-1]   \sqrt[n]{(X, D)} 
  \sqcup \big ( \bigsqcup_{i=1}^{n-1} \mathsf T[-1]  D_n \big ).
\end{equation}

\begin{prop}
\label{prop:HHRootStack}
    Let $D$ be a smooth divisor on a smooth scheme $X$. Let $n$ be a positive integer. We have canonical isomorphisms:
    \begin{enumerate}\itemsep=0.2cm
        \item $\Gamma( \mathsf T[-1]   \sqrt[n]{(X, D)} , \mathcal{O}) \simeq \Gamma(\mathsf T[-1]X , \mathcal{O})$.
        \item  $\Gamma( \mathsf T[-1] D_n , \mathcal{O}) \simeq \Gamma(\mathsf T[-1] D, \mathcal{O} )$.
    \end{enumerate}
\end{prop}
\begin{proof}

By the functoriality of the (shifted) tangent, $\sqrt[n]{(X, D)}\to X$ induces a morphism $\pi\colon \mathsf T[-1]  \sqrt[n]{(X, D)}\to \mathsf T[-1]  X$, thus a natural morphism
  \begin{equation}
\label{rootroot2}
  \mathcal{O}_{\mathsf T[-1]X}\to  \pi_*( \mathcal{O}_{ \mathsf T[-1]   \sqrt[n]{(X, D)} }) .
  \end{equation}
  To prove (1), it suffices to show that \eqref{rootroot2} is an isomorphism.
We will give two proofs.
%In fact, we shall sketch two proofs of this fact. One argument, which we give first, depends on the local structure on the root stack. It has the advantage of being more  explicit and, further, to also  apply to the Hochschild cohomology of the root stack. However, for completeness, we   shall also present a second argument which is conceptually more natural. 

Our first proof is via local computation: it suffices to check that \eqref{rootroot2} is an isomorphism \'etale locally on $X$. Hence we can assume that the pair $(X,D)$ is  of the form
$$
(X, D)= (\mathbb{A}^1 \times S, \{0 \} \times S ),
$$
where $S$ is a smooth variety of dimension $\dim(X)-1$. In this case, the root stack   
$\sqrt[n]{(X, D)}$ is isomorphic to a stack of the form
$$
\sqrt[n]{(\mathbb{A}^1, \{0\})} \times S \simeq [\mathbb{A}^1/\mu_n] \times S 
$$
where $\mu_n$ acts by multiplication of roots of unity.

As $\mathsf T[-1] ([\mathbb{A}^1/\mu_n]\times S) \simeq \mathsf T[-1] [\mathbb{A}^1/\mu_n]\times  \mathsf T[-1]  S$, we are reduced to the case where $$(X, D)= (\mathbb{A}^1, \{0\}) \ .$$ 
%Since $[\mathbb{A}^1/\mu_n]$ is a global quotient DM stack, we have that
Note that
$$
\mathsf T[-1] [\mathbb{A}^1/\mu_n] \simeq 
[\mathsf T[-1] \mathbb{A}^1 / \mu_n]
%\simeq [\Omega_0 \mathbb{A}^1 \times \mathbb{A}^1/\mu_n]
$$
where $\mu_n$ acts diagonally.
Let $t$ be the variable of $\mathbb{A}^1$ and $z$ be the variable of $\mathbb{A}^1\coloneqq\mathbb{A}^1/\mu_n$; $z=t^n$. Then we conclude by the following computation:
\begin{align*}
    \Gamma([\mathsf T[-1] \mathbb{A}^1 / \mu_n], \mathcal{O}) & \simeq \Gamma(\mathbb{A}^1, \Sym(\mathbb{L}_{\mathbb{A}^1}[1]))^{\mu_n}\\
    & \simeq \Gamma(\mathbb{A}^1, \mathcal{O})^{\mu_n} \oplus \Gamma(\mathbb{A}^1, \mathbb{L}_{\mathbb{A}^1}[1])^{\mu_n} \\
    & \simeq k[t^n]\oplus k[t^n]\mathsf{d}(t^n)\\
    & \simeq k[z]\oplus k[z]\mathsf{d}z\\
    & \simeq \Gamma(\mathbb{A}^1, \mathcal{O})  \oplus \Gamma(\mathbb{A}^1, \mathbb{L}_{\mathbb{A}^1}[1]) \\
    & \simeq \Gamma(\mathsf T[-1] \mathbb{A}^1 , \mathcal{O}).
\end{align*}
%We can describe this object more explicitly. Indeed  $\mathsf T[-1] \mathbb{A}^1$ splits as the product $\Omega_0 \mathbb{A}^1 \times \mathbb{A}^1$, and the $\mu_n$-action preserves the splitting. Thus we can write 
%\begin{equation}
%    \mathsf T[-1] [\mathbb{A}^1/\mu_n] \simeq [\Omega_0 \mathbb{A}^1 /\mu_n] \times [\mathbb{A}^1 /\mu_n] 
%\end{equation}

%By the above discussion checking that   (\ref{rootroot2}) is an equivalence for the pair $(\mathbb{A}^1, \{0\})$ means showing that the push-forward along the map 
%$$
%f: [\Omega_0 \mathbb{A}^1 /\mu_n] \times [\mathbb{A}^1 /\mu_n] \to \Omega_0 \mathbb{A}^1 \times \mathbb{A}^1
%$$
%sends the structure sheaf to the structure sheaf.  This is clear, since $f$ is a coarse moduli map. This shows that the map of quasi-coherent sheaves (\ref{rootroot2}) is an equivalence \'etale locally, hence it must be a global equivalence. In particular, this implies that  equivalence   (\ref{rootroot}) holds, which is what we wanted to show.  

Now we  give an alternative proof  that (\ref{rootroot2}) is an equivalence. By the definition of the root stack, and the fact that the shifted tangent preserves fiber products, we obtain a pullback square
\[ \xymatrix{\mathsf T[-1] \sqrt[n]{(X, D)} \ar[r] \ar[d]_-\pi & \mathsf T[-1]  [\mathbb{A}^1/\mathbb{G}_m] \ar[d]^{p} \\
\mathsf T[-1]X \ar[r] & \mathsf T[-1] [\mathbb{A}^1/\mathbb{G}_m] }
\]
By base change,   
%it is enough to prove a universal form of the comparison statement we are interested in. Namely, 
it is sufficient to show that the push-forward along the  map 
$$
p: \mathsf T[-1] [\mathbb{A}^1/\mathbb{G}_m] \to \mathsf T[-1] [\mathbb{A}^1/\mathbb{G}_m]
$$
sends the structure sheaf to the structure sheaf.

It is not difficult to describe explicitly the stack $\mathsf T[-1] [\mathbb{A}^1/\mathbb{G}_m]$.  Indeed, if $Y$ is a scheme with an action of an algebraic  group $G$, with Lie algebra $\mathfrak{g}$, it is always possible to construct a global atlas for $T[-1][Y/G]$. Namely, let $\mathsf T[-1]_G Y$ be the following fiber product 
\[ \xymatrix{\mathsf T[-1]_G Y \ar[r] \ar[d] & \mathsf Y  \ar[d]  \\
Y \times \mathfrak{g} \ar[r] &  \mathsf T Y }
\]
where the right vertical arrow is the zero section, and the bottom arrow is given by the infinitesimal action. Then $\mathsf T[-1]_G Y$ carries a natural $G$-action, and we have an equivalence\footnote{This  is a  small variation on the well-known construction of the atlas  of  the loop space of a global quotient stack which is well documented in the literature, see  for instance Proposition 2.1.8 in \cite{chen2020equivariant}; the details of the construction in the shifted tangent case will be spelled out in the forthcoming \cite{Rychlewiczinpreparation}.} 
$$
[\mathsf T[-1]_G Y / G] \simeq \mathsf T[-1][Y/G]
$$

It is easy to work out this construction in the case of the multiplicative action of $\mathbb{G}_m$ on $\mathbb{A}^1$. 
Consider the variety 
$$
C = \{ xy = 0 \} \subset \mathbb{A}^1_x \times \mathbb{A}^1_y 
$$
and equip it with a $\mathbb{G}_m$-action which scales the $x$-direction with weight $1$, and acts trivially on the $y$-coordinate. One can see that 
$$
\mathsf T[-1]_{\mathbb{G}_m} \mathbb{A}^1 \simeq C \quad \text{and} \quad \mathsf T[-1][\mathbb{A}^1/\mathbb{G}_m] \simeq  [C/\mathbb{G}_m]
$$
In particular, let 
$$
\iota: \mathsf T[-1] B \mathbb{G}_m \simeq \mathbb{A}^1 \times \mathsf B \mathbb{G}_m \to \mathsf T[-1] [\mathbb{A}^1/\mathbb{G}_m]
$$
be the inclusion. Then the map 
$$
\iota^*: \mathcal{O}(\mathsf T[-1] [\mathbb{A}^1/\mathbb{G}_m]) \to \mathcal{O}(\mathsf T[-1] \mathsf B \mathbb{G}_m)
$$
is an equivalence.  
By applying again base change, we  reduce to prove that the derived push-forward along the map
$$
q=p|_{ \mathsf T[-1] \mathsf B \mathbb{G}_m }:  \mathsf T[-1] \mathsf B \mathbb{G}_m 
\simeq \mathbb{A}^1 \times \mathsf B \mathbb{G}_m 
\to \mathsf T[-1] \mathsf B \mathbb{G}_m \simeq \mathbb{A}^1 \times \mathsf B \mathbb{G}_m
$$
sends the structure sheaf to the structure sheaf. Note that $q$ is the map induced at the level shifted tangent bundles  by the  map 
$$
(-)^n: \mathsf B \mathbb{G}_m \to \mathsf B \mathbb{G}_m
$$
which is formally \'etale. Thus its differential, i.e. the induced map between the (shifted) tangents, is an equivalence. That is, $q$ is an equivalence when restricted to  the linear factor 
$$
\mathbb{A}^1 \simeq \mathsf T[-1]_* \mathsf B \mathbb{G}_m \subset \mathsf T[-1] \mathsf B \mathbb{G}_m 
\simeq \mathbb{A}^1 \times \mathsf B \mathbb{G}_m 
$$
This implies in particular that, as claimed, there is an equivalence 
$$
q_*( \mathcal{O}_{(\mathsf T[-1] \mathsf B \mathbb{G}_m )} ) \simeq \mathcal{O}_{(\mathsf T[-1] \mathsf B \mathbb{G}_m )}
$$
which is what we needed to show. 

To show (2), it suffices to notice that 
$$
\mathsf T[-1] D_n \to \mathsf T[-1] D
$$
is a $\mu_n$-gerbe hence the derived push-forwards sends the structure sheaf to structure sheaf. 
\end{proof}

Combining \eqref{eq:LoopOfRootStack} and  \Cref{prop:HHRootStack}, we obtain the following consequence:
\begin{cor}
\label{cor:HHRootStack}
Let $D$ be a smooth divisor in a smooth scheme $X$. Let $n$ be a positive integer. 
We have a canonical isomorphism 
\begin{equation}
    \mathsf{HH}_*( \sqrt[n]{(X, D)} ) \simeq H^*(X, \Sym(\mathbb{L}_X[1]))\oplus H^*(D, \Sym(\mathbb{L}_D[1]))^{\oplus n-1}.
\end{equation}
In other words, we have a natural isomorphism of Hochschild homology groups that is compatible with HKR decompositions:
\begin{equation}
\label{eqn:HHRootStack}
    \mathsf{HH}_*( \sqrt[n]{(X, D)} ) \simeq \mathsf{HH}_*( X ) \oplus  \bigoplus_{i=1}^{n-1} \mathsf{HH}_*( D ).
\end{equation}

\end{cor}
\begin{rem}
    The  decomposition \eqref{eqn:HHRootStack} of the Hochschild homology in \Cref{cor:HHRootStack}, regardless of the compatibility with HKR decompositions, was already known through non-commutative methods. Indeed, Hochschild homology being additive with respect to semi-orthogonal decompositions, \eqref{eqn:HHRootStack} follows from the natural semi-orthogonal decomposition of the category of perfect complexes on the root stack $\sqrt[n]{(X, D)}$:
    \begin{equation}
    \label{eqn:SOD}
        \Perf (\sqrt[n]{(X, D)}) \simeq \langle\Perf(X), \underbrace{ \Perf(D), \cdots, \Perf(D)}_{n-1 \text{ copies } } \rangle .
    \end{equation}
    See \cite{ishii2015special}, \cite{bergh2016geometricity}, \cite{bodzenta2024root} and \cite{scherotzke2020parabolic} for details.
\end{rem}

Note that computation of Hochschild cohomology is often much more challenging than that of Hochschild cohomology, since it is not functorial and not additive with respect to semi-orthogonal decompositions. Nevertheless, our main theorem also allows us to compute the Hochschild cohomology of root stacks. We stress that the following result is \textit{not} a consequence of \eqref{eqn:SOD}.

\begin{prop}[Hochschild cohomology of root stacks]
\label{prop:HHcohRootStack}
    Let $D$ be a smooth divisor on a smooth scheme $X$. Let $n$ be a positive integer. 
    \begin{equation}
    \label{eqn:HHcoh-1}
        \mathsf{HH}^*(\sqrt[n]{(X, D)}) \simeq \Gamma(\sqrt[n]{(X, D)}, \Sym(\mathbb{T}_{\sqrt[n]{(X, D)}}[-1])) \simeq\Gamma(\mathsf T^*[1]\sqrt[n]{(X, D)}, \mathcal{O}).
    \end{equation}
    Moreover, if $n=2$, 
    \begin{equation}
    \label{eqn:HHcoh-2}
        \mathsf{HH}^*(\sqrt{(X, D)}) \simeq \mathsf{Fib} (\Gamma(X, \Sym(\mathbb{T}_X[-1]))  \to  \Gamma(D, \Sym(\mathbb{T}_D(D)[-1])[-1])).
    \end{equation}
\end{prop}
\begin{proof}
    By \Cref{Hochcoh} and \eqref{eqn:InertiaOfRootStack}, we have a decomposition:
    \begin{equation}
        \mathsf{HH}^*(\sqrt[n]{(X, D)}) \simeq \Gamma(\sqrt[n]{(X, D)}, \Sym(\mathbb{T}_{\sqrt[n]{(X, D)}}[-1])) \oplus \Gamma(D_n, \Sym(\mathbb{T}_{D_n}[-1])\otimes \mathsf{N}[-1]),
    \end{equation}
    where $\mathsf N$ is the normal bundle of $D_n$ in $\sqrt[n]{(X, D)}$.
    
    To show \eqref{eqn:HHcoh-1}, it suffices to prove the vanishing of the sumand $\Gamma(D_n, \Sym(\mathbb{T}_{D_n}[-1])\otimes \mathsf{N}[-1])$, for which it is enough to show that 
    $f_*(\Sym(\mathbb{T}_{D_n}[-1])\otimes \mathsf{N})=0$ on $D$, where $f\colon D_n\to D$ is the natural $\mu_n$-gerbe map. Since this vanishing can be checked \'etale locally on $D$, we are reduced to the case where the line bundle $\mathcal{O}_X(D)$ has an $n$-th root $L$. In this case, let $\pi\colon \widetilde{X}\to X$ be the $n$-fold cyclic cover of $X$ branched along $D$ and let $R$ be the ramification divisor, we have  $\sqrt[n]{(X, D)}\simeq [\widetilde{X}/\mu_n]$, $D_n\simeq [R/\mu_n]$ and $\pi$ restricts to an isomorphism $\phi\colon R\to D$.
    Therefore
    \begin{equation*}
        f_*(\Sym(\mathbb{T}_{D_n}[-1])\otimes \mathsf{N}) \simeq \phi_*(\Sym(\mathbb{T}_{R}[-1])\otimes \mathsf{N}_{R/\widetilde{X}})^{\mu_n}\simeq \Sym(\mathbb{T}_{D}[-1])\otimes (\phi_*\mathsf{N}_{R/\widetilde{X}})^{\mu_n}.
    \end{equation*}
    However, $\phi_*\mathsf{N}_{R/\widetilde{X}}\simeq L|_D$ and the $\mu_n$-action is by multiplication by roots of unity. Consequently, $\phi_*\mathsf{N}_{R/\widetilde{X}}^{\mu_n}=0$.

    Having proved \eqref{eqn:HHcoh-1}, in order to show \eqref{eqn:HHcoh-2}, it suffices to establish that when $n=2$, for any non-negative integer $p$, the following fiber sequence in $\QCoh(X)$ (it is in fact a short exact sequence of coherent sheaves on $X$) :
    \begin{equation}
    \label{eqn:SESTangentDoubleCover}
        0\to \pi_*\left(\bigwedge^p \mathbb{T}_{\sqrt{(X, D)}}\right) \to \bigwedge^p \mathbb{T}_X\to i_*\left(\bigwedge^{p-1} \mathbb{T}_D\right)\otimes \mathcal{O}_X(D)\to 0.
    \end{equation}
    where $i\colon D\to X$ is the closed immersion.
    Since all the morphisms in \eqref{eqn:SESTangentDoubleCover} are naturally defined, it suffices to check the exactness \'etale locally, hence we are reduced to the case where where the line bundle $\mathcal{O}_X(D)$ has an $n$-th root $L$. Let $\pi\colon \widetilde{X}\to X$ be the associated double cover and $\phi\colon R \to D$ be the isomorphism as before. A local computation yields for any $p\geq 1$ the following short exact sequence of coherent sheaves on $\widetilde{X}$:
    \begin{equation}
    \label{eqn:SESforOmega}
        0\to \pi^*\Omega^p_X \to \Omega^p_{\widetilde{X}} \to i'_*\Omega^{p-1}_R\otimes \mathcal{O}_{\widetilde{X}}(-R)\to 0 ,
    \end{equation}
    where $i'\colon R\to \widetilde{X}$ is the closed immersion.
    Applying the functor of derived dual to \eqref{eqn:SESforOmega} and using the Grothendieck--Serre duality, we get the following short exact sequence on $\widetilde{X}$
     \begin{equation}
         \label{eqn:SESforTangent}
        0\to \bigwedge^p \mathbb{T}_{\widetilde{X}} \to \pi^*\left(\bigwedge^p \mathbb{T}_X\right) \to i'_*\left(\bigwedge^{p-1} \mathbb{T}_R\right)\otimes \mathcal{O}_{\widetilde{X}}(2R)\to 0.
    \end{equation}
    Then \eqref{eqn:SESTangentDoubleCover} can be deduced from \eqref{eqn:SESforTangent} by applying $\pi_*$ and taking $\mu_2$-invariants.
\end{proof}

\end{eg}

\subsection{Beyond global quotients (II) : quotients by algebraic groups}
\label{subsec:QuotientByAlgGroups}
Although most DM stacks are not global quotients by finite groups, in practice, by working with quotient stacks by linear algebraic groups, we can already deal with a fairly large class of DM stacks. For instance, the weighted projective line $\mathbb{P}^1(p,q)$ in \Cref{eg:WeightedProjLine} is equivalent to the quotient stack $[\mathbb{A}^2/\mathbb{G}_m]$ where the 2-dimensional representation of $\mathbb{G}_m$ is of weight $p, q$. More generally, in \cite[Theorem 2.18]{Edidin-Hassett-Kresch-Vistoli}, it is shown that any smooth DM stack with trivial generic stabilizer is a quotient stack of the form $[Y/G]$ with $G$ a linear algebraic group acting on an algebraic space $Y$. Motivated by this consideration, we compute the inertia stack as well as the Hochschild homology of derived DM stacks that are quotient stacks by linear algebraic groups.

\begin{prop}
\label{prop:QuotientByAlgebraicGroup}
    Let $Y$ be a finitely presented derived \DM stack equipped with an action of a linear algebraic group $G$. Assume that the quotient stack $[Y/G]$ is a separated derived \DM stack. Then there are only finitely many conjugacy classes in $G$ with non-empty fixed loci and they are semi-simple; we denote this finite set by $\mathscr{C}$.   
    Then we have an isomorphism
    \begin{equation}
       \inertiaDM [Y/G] \simeq \bigsqcup_{[g] \in \mathscr{C}} [Y^g/Z(g)]\ ,
    \end{equation}
    where $Y^g$ is the derived fixed locus and $Z(g)$ is the centralizer of $g$ in $G$.\\
    Consequently, 
    \begin{equation}
        \loopstack [Y/G] \simeq \bigsqcup_{[g] \in \mathscr{C}} \mathsf T[-1][Y^{g}/Z(g)]\ .
    \end{equation}
\end{prop}

\begin{proof}
%The proof is a slight refinement of the argument for \Cref{prop:InertiaGlobalQuotient}.
Since $X\coloneqq [Y/G]$ is a separated \DM stack and $G$ is affine, the morphism $Y\times G\to Y\times Y$ is a finite morphism. 
Therefore, any $[g]\in \mathscr{C}$ must be of finite order, hence semi-simple (as we are in characteristic zero).
By \Cref{thm:basics_inertia}, the classical truncation of $\inertiaDM X$ is nothing but $\inertia [\trunc(Y)/G]\simeq [\inertia_G\trunc(Y)/G]$, where $\inertia_G\trunc(Y)=\trunc(Y\times_{Y\times Y} (Y\times G))$ is the classical universal stabilizer (a.k.a.~the inertia scheme) which is finite over $Y$. 
By \cite[Lemma 2.10 and Proposition 2.11]{Edidin_Jarvis_Kimura} the set $\mathscr{C}$ is finite and consists of conjugacy classes of semi-simple elements. 
 
For any $[g]\in \mathscr{C}$, i.e.~$Y^g$ is non-empty. Let $r$ be the order of $g$.
In the commutative diagram,
\begin{equation}
\label{Diag:FixedLocus}
\begin{gathered}
\xymatrix{ Y^g  \ar[r] \ar[d]& \mathbf{Map}( \mathsf B \langle g \rangle,  [Y/ \langle g \rangle ]) \ar[d] \ar[r] &  \mathbf{Map}(\mathsf B \langle g \rangle,  [Y/ G  ]) \ar[d]\ \\
\Spec(k)  \ar[r]^-{\id} & \mathbf{Map}(  \mathsf B \langle g \rangle , \mathsf B\langle g \rangle ) \ar[r] & \mathbf{Map}(\mathsf B \langle g \rangle , \mathsf B G )}
\end{gathered}
\end{equation}
the left square is a pullback square by definition, and the right square is a pullback square since it is obtained by applying  
$\mathbf{Map}(\mathsf B \langle g \rangle,  -)$ to the following pullback square:
\begin{equation*}
\begin{gathered}
\xymatrix{ [Y/ \langle g \rangle ] \ar[d] \ar[r] &    [Y/ G]    \ar[d]  \\
 \mathsf B\langle g \rangle \ar[r] &  \mathsf BG }
\end{gathered}
\end{equation*}
Note that the multiplication map $Z(g)\times \langle g\rangle\to G$ is a morphism of group by the definition of centralizer. This induces a morphism $\mathsf BZ(g)\to \mathbf{Map}(\mathsf B \langle g \rangle , \mathsf B G )$ through which the composition of the bottom row of \eqref{Diag:FixedLocus} factors. As a result, we have a pullback square:
\begin{equation}
\label{diag:GlobalQuotientEtaleChart}
    \begin{gathered}
\xymatrix{ 
[Y^g/Z(g)]  \ar[r] \ar[d] &  \mathbf{Map}(\mathsf B \langle g \rangle,  [Y/ G ]) \simeq \inertiarth [Y/G] \ar[d]\ \\
\mathsf BZ(g)  \ar[r] & \mathbf{Map}(\mathsf B \langle g \rangle , \mathsf B G )
}
\end{gathered}
\end{equation}
We claim that the bottom arrow (hence also the top arrow) is \'etale. Indeed, 
\begin{equation*}
    \mathbb{L}_{\mathsf BZ(g)}\simeq \Lie(Z(g))^\vee[-1]\ ,
\end{equation*}
as the adjoint $Z(g)$-module, and
\begin{equation*}
    \mathbb{L}_{\mathbf{Map}(\mathsf B \langle g \rangle , \mathsf B G )}|_{\mathsf BZ(g)}\simeq \Gamma (\mathsf B\langle g \rangle, \Lie(G)^\vee[-1])\simeq (\Lie(G)^g)^\vee[-1]
\end{equation*}
as $\mathsf B\langle g \rangle$ has no higher cohomology (we are in characteristic zero). Here $g$ acts on $\Lie(G)$ by conjugation.
Now the canonical isomorphism of $Z(g)$-modules $$\Lie(G)^g\simeq \Lie(Z(g))$$ identifies the above two cotangent complexes and we conclude that the top morphism in \eqref{diag:GlobalQuotientEtaleChart} is \'etale. 

When $[g]$ runs through the finite set $\mathscr{C}$, we get an \'etale morphism 
\begin{equation}
\label{eqn:InertiaOfQuotient}
    \bigsqcup_{[g]\in \mathscr{C}} [Y^g/Z(g)] \to \inertiaDM [Y/G].
\end{equation}
To check that it is an equivalence, it suffices to check that it induces an isomorphism on the classical truncation.
To this end, it is enough to observe that that the following morphism is an isomorphism (see also \cite[Lemma 2.3]{Kinjo_Multiplicative_dimensional_reduction}):
\begin{align*}
        \bigsqcup_{[g]\in \mathscr{C}} G\times^{Z(g)} \trunc(Y)^g &\to \inertia_G\trunc(Y)\\
        (h,y)&\mapsto (hgh^{-1}, hy)
\end{align*}
where $Z(g)$ acts on $G\times Y^g$ via $t.(h,y)\coloneqq(ht^{-1}, ty)$. 
The last assertion in the statement follows from \Cref{thm:HKR_DM}.
\end{proof}

\begin{cor}
\label{cor:HHofQuotientByAlgGp}
    Let $Y$ be a finitely presented derived algebraic space equipped with an action of a linear algebraic group $G$. Assume that the quotient stack $[Y/G]$ is a separated \DM stack. Then there is an isomorphism of graded algebras 
    \begin{equation*}
       \bigoplus_{i \in \mathbb{Z}} \HH_{-i}([Y/G])\simeq \bigoplus_{i\in \mathbb{Z}}\bigoplus_{[g]\in \mathscr{C}} \bigoplus_{q-p=i}H^q([Y^g/Z(g)], \Omega^p_{[Y^g/Z(g)]})\ ,
    \end{equation*}
     where $\Omega^p\coloneqq\bigwedge^p\mathbb{L}$ stands for the $p$-th term of the \emph{derived} de Rham complex, and $Y^g$ is the \emph{derived} fixed locus.
\end{cor}
\begin{proof}
It follows from the combination of \Cref{thm:HKR-HH} and \Cref{prop:QuotientByAlgebraicGroup}.
\end{proof}

\section{Circle action and de Rham differential}
\label{sec:GroupStructure}

We assume throughout this section that $k$ is a $\Q$-algebra.
This has two main consequences, both of which fail away from characteristics $0$:
\begin{enumerate}\itemsep=0.2cm
    \item As observed in Example~\ref{eg:Betti} and Example~\ref{affd-1}, both $\Aff(S^1)$ and $\Aff(\bbD_{-1})$ become equivalent to $\bbG_a$.
    This is the main input for the HKR isomorphism discussed in Section \ref{HKR isomorphism for Hochschild homology} and \ref{HKR isomorphism for Hochschild cohomology}.
    
    \item $\mathsf C^\ast_{\mathsf{sing}}(S^1;k)$ is formal as a Hopf algebra, as recalled in \S\ref{sec:formality}. Equivalently, $\Aff(S^1)$ and $\Aff( \bbD_{-1} )$ are equivalent as group objects in coaffine stacks.
    This is the key to obtain HKR equivalences for cyclic, negative cyclic and periodic cyclic homology.
\end{enumerate}
In this section, we focus on this second phenomenon, and extend the classical results known for schemes to \DM stacks.

\subsection{Formality} \label{sec:formality}

Recall from Example~\ref{lemma:AffS^1=BGa} that there are natural morphisms
\begin{equation}\label{eq:zig_zag_formality}
    \begin{tikzcd}
        \bbD_{-1} \arrow{r}{u} & \BGa & S^1 \arrow{l}[swap]{\aff}
    \end{tikzcd}
\end{equation}
inducing isomorphisms on the affinizations.
Moreover, both $S^1$ and $\mathsf B\bbG_a$ have canonical group structures, and the above map $\aff$, that arises by delooping the morphism $\Z \to \bbG_a$, is also a map of groups.
Additionally, $\Aff(\bbD_{-1}) = \Spec^{\mathsf{nc}}(k \oplus k[-1])$ carries a canonical group structure, corresponding to the natural Hopf-algebra structure on the split square-zero extension $k \oplus k[-1]$.
In order to argue that $u$ respects the group structure, we prove:

\begin{lem}
    The space
    \[ \{\BGa\} \times_{(\dSt_k)_\ast^\simeq} \mathsf{Mon}_{\mathbb E_\infty}^{\mathsf{gp}}(\dSt_k)^\simeq  \]
    of $\mathbb E_\infty$-group structures on $\BGa$ is contractible.
\end{lem}

\begin{proof}
    The canonical atlas $\Spec(k) \to \BGa$ induces an isomorphism on $\pi_0$, and therefore $\BGa \in \dSt_k^{\ge 1}$.
	Furthermore that the inclusion $\dSt_k^{\geqslant 1} \hookrightarrow \dSt_k$ commutes with products.
	We can therefore consider the following ladder of pullbacks
	\[ \begin{tikzcd}
		X \arrow{r} \arrow{d} & \mathrm{Mon}_{\mathbb E_\infty}^{\mathrm{gp}}(\dSt_k^{\ge 1}) \arrow{r} \arrow{d} & \mathrm{Mon}_{\mathbb E_\infty}^{\mathrm{gp}}(\dSt_k) \arrow{d} \\
		\{*\} \arrow{r}{\BGa} \arrow{r} & (\dSt_k)_\ast^{\ge 1} \arrow{r} & (\dSt_k)_\ast
	\end{tikzcd} \ , \]
    and conclude that
    \[ X^\simeq \simeq \{\BGa\} \times_{(\dSt_k)_\ast^\simeq} \mathsf{Mon}_{\mathbb E_\infty}^{\mathsf{gp}}(\dSt_k)^\simeq \]
    is the space of $\mathbb E_\infty$-group structures on $\BGa$.
    We focus on the left square.
	Consider the commutative rectangle
	\[ \begin{tikzcd}
		X \arrow{r} \arrow{d} & \mathrm{Mon}_{\mathbb E_\infty}^{\mathrm{gp}}(\dSt_k^{\ge 1}) \arrow{r}{\Omega} \arrow{d} & \mathrm{Mon}_{\mathbb E_\infty}^{\mathrm{gp}}( \mathrm{Mon}_{\mathbb E_1}^{\mathrm{gp}}(\dSt_k) ) \arrow{d} \\
		\{*\} \arrow{r}{\BGa} & (\dSt_k)_\ast^{\ge 1} \arrow{r}{\Omega} & \mathrm{Mon}_{\mathbb E_1}^{\mathrm{gp}}(\dSt_k) 
	\end{tikzcd} \]
	May's delooping theorem \cite[Theorem 5.2.6.15]{HA} implies that the horizontal morphism in the square on the right are equivalences.
	In particular, the square in question is a pullback.
	As a consequence, it is enough to compute the outer pullback.
	
	Observe now that $\Omega( \BGa ) \simeq \bbG_a$ and that this is a discrete object in $\dSt_k$.
	Furthermore, the induced $\mathbb E_1$-group structure on $\bbG_a$ coincides with the additive one.
	We now observe that, since $\bbG_a$ is discrete and since the $\mathbb E_1$-structure is fixed, being $\mathbb E_\infty$ is now a property rather than a structure.
	In other words, we see that $X$ is either empty or contractible.
	As the additive group structure on $\bbG_a$ is commutative, we see that it is indeed the latter case.
\end{proof}

\begin{cor}\label{cor:formality_Hopf}
    The zig-zag of \eqref{eq:zig_zag_formality} induces an equivalence
    \begin{equation}\label{eq:formality}
    \phi \colon \mathsf C^\ast_{\mathsf{sing}}(S^1;k) \simeq k \oplus k[-1]
    \end{equation}
    as cocommutative Hopf algebras.
    Equivalently, upon passing to the affinization it induces an equivalence of derived $\mathbb E_\infty$-group stacks
    \[ \begin{tikzcd}
        \Aff(\bbD_{-1}) \arrow{r}{\overline{u}} & \BGa & \Aff(S^1) \arrow{l}[swap]{\overline{\aff}} \ .
    \end{tikzcd} \]
\end{cor}

\begin{rem}
    Robalo proved in \cite{Robalo-ChoicesHKR} that the space of equivalences of the form \eqref{eq:formality}, subjected to certain naturality constraints is reduced to $\{\lambda \phi\}_{\lambda \in k^\times}$.
    Here $\lambda \in k^\times$ acts on $k \oplus k[-1]$ by $\lambda \cdot (a,b) \coloneqq (a,\lambda b)$, which provides indeed an auto-equivalence of Hopf algebras.
\end{rem}

\begin{recollection}[Toën-Vezzosi's perspective]\label{recollection:Toen_Vezzosi}
    Let us recall, following \cite{Toen_Vezzosi_S1_algebras}, how the formality of Corollary~\ref{cor:formality_Hopf} allows to make the HKR isomorphism into a more structured statement, taking into account the natural $S^1$-action on the Hochschild side and the de Rham differential on the de Rham side.
    Consider the categories of $S^1$- and mixed algebras, defined respectively as
    \[ S^1\textrm{-}\CAlg_k \coloneqq \mathsf{CoMod}_{\mathsf C^\ast_{\mathsf{sing}}(S^1;k)}(\CAlg_k) \qquad \text{resp.} \qquad \varepsilon \textrm{-} \CAlg_k \coloneqq \mathsf{CoMod}_{k\oplus k[-1]}(\CAlg_k) \]
    Then the formality \eqref{eq:formality} induces an equivalence $A_\phi$ making the following triangle commutative:
    \begin{equation}
    \label{TVS^10dR}
    \begin{tikzcd}[column sep=small]
        S^1\textrm{-}\mathsf{CAlg}_k \arrow{rr}{A_\phi}[swap]{\sim} \arrow{dr}[swap]{U_{S^1}} & & \varepsilon\textrm{-}\mathsf{CAlg}_k \arrow{dl}{U_\varepsilon} \\
        {} & \mathsf{CAlg}_k \ ,
    \end{tikzcd} 
    \end{equation}
    where $U_{S^1}$ and $U_\varepsilon$ are the natural forgetful functors.
    The main result of \cite{Toen_Vezzosi_S1_algebras} asserts that comonadic functors $U_{S^1}$ and $U_\varepsilon$ are also monadic, and their left adjoints $L_{S^1}$ and $L_\varepsilon$ admit the following description: for all $A\in \mathsf{CAlg}_k$,
    \[ U_{S^1} L_{S^1}(A) \simeq S^1 \otimes_k A \simeq A \otimes_{A \otimes_k A} A \qquad \text{resp.} \qquad U_\varepsilon L_\varepsilon(A) \simeq \mathsf{DR}(A) \coloneqq \Sym_A( \mathbb L_{A/k}[1] ) \]
    It follows that $A_\phi$ exchanges the $S^1$-action (i.e.\ $\mathsf C^\ast_{\mathsf{sing}}(S^1;k)$-coaction) on Hochschild homology of $A$ with the mixed structure (i.e.\ $(k\oplus k[-1])$-coaction) on $\mathsf{DR}(A)$ given by the de Rham differential. 
\end{recollection}

\subsection{HKR isomorphism for cyclic, negative cyclic and periodic cyclic homology}
\label{HKR isomorphism for cyclic homology}

The goal of this section is to generalize the picture of Recollection~\ref{recollection:Toen_Vezzosi} to the \DM setting.

\begin{construction}
    Let $X$ be a derived \DM stack  locally almost of finite presentation over $k$.
    Recall from \cref{thm:HKR_DM} that the morphism $\hataff \colon S^1 \to \BGa \times \hatcircle$ induces the following commutative triangle:
    \begin{equation}\label{eq:HKR_over_X}
    \begin{tikzcd}[column sep=small]
        \mathsf T[-1]\inertiaDM X \arrow{rr}{\hataff^\ast} \arrow{dr}[swap]{\pi_1} & & \loopstack X \arrow{dl}{\pi_2} \\
        {} & X
    \end{tikzcd}
    \end{equation}
    Moreover, since $\loopstack X \simeq \bfMap(S^1,X)$, we see that $\loopstack X$ carries a canonical $(S^1 \times X)$-action over $X$.
    Similarly, $\mathsf T[-1]\inertiaDM X \simeq \Map(\BGa \times \hatcircle, X)$ carries a canonical $(\BGa \times X)$-action over $X$, induced by the left action of $\BGa$ on $\BGa \times \hatcircle$.
    The morphism of groups $\aff \colon S^1 \to \BGa$ allows to see $\mathsf T[-1]\inertiaDM X$ as equipped with a $(S^1 \times X)$-action over $X$.
\end{construction}

\begin{prop} \label{S1derham}
    The equivalence $\hataff^\ast$ supplied by \cref{thm:HKR_DM} can be promoted to an equivalence of $(S^1 \times X)$-derived \DM stacks over $X$.
\end{prop}

\begin{proof}
    Consider $\BGa \times \hatcircle$ together with its natural left $\BGa$-action.
    The morphism of groups $\aff \colon S^1 \to \BGa$ allows to see $\BGa \times \hatcircle$ as a left $S^1$-module in $\dSt_k$, and the map $\hataff \colon S^1 \to \BGa \times \hatcircle$ becomes a morphism of left $S^1$-modules (where $S^1$ is considered as a left module over itself, with action given by the group multiplication).
    Since the $(S^1 \times X)$-actions on $\mathsf T[-1]\inertiaDM X$ and on $\loopstack X$ are induced by the universal properties of the mapping stacks, the conclusion follows from this analysis.
\end{proof}

We now extract some more concrete consequences.

\begin{rem}\label{recollection:S1_algebras_sheaves}
    Let $X$ be a derived \DM stack.
    Let $\mathsf I\mathscr X$ be the small étale topos of $\inertiaDM X$.
    Taking sheaves on $\mathsf I\mathscr X$ with coefficients in $S^1\textrm{-}\CAlg_k$ and $\varepsilon\textrm{-}\CAlg_k$ and following \cite{hoyois2015homotopy}, we obtain the following commutative diagram:
    \begin{equation} \label{diag:Intertwinings}
        \xymatrix{
            S^1\textrm{-}\mathsf{CAlg}_k(\mathsf I\mathscr X) \ar@/_3pc/[dd]_-{\Gamma(-)} \ar[rrrr]^-{A_\phi}_-{\sim}   \ar[d] & &&& \varepsilon\textrm{-}\mathsf{CAlg}_k(\mathsf I\mathscr X)  \ar[d] \ar@/^3pc/[dd]^-{\Gamma(-)} \\
            S^1\textrm{-}\mathsf{CAlg}_k \ar[rrrr]^-{A_\phi}_-{\sim}  \ar[d]_-{F}  & & & & \varepsilon\textrm{-}\mathsf{CAlg}_k \ar[d]^-{F} \\
            S^1\textrm{-}\mathsf{Mod}_k \ar[rrrr]^-{A_\phi}_-{\sim}  \ar@/^1.5pc/[ddrr]^-{a} \ar@/_2.5pc/[ddrr]_-{b}  & & & &  \varepsilon\textrm{-}\mathsf{Mod}_k \ar@/^2.5pc/[ddll]^-{ \beta } \ar@/_1.5pc/[ddll]_-{ \alpha } \\ &&&& \\
            & & \mathsf{Mod}_k \ar@/_/[uurr]_-{ \mathsf{triv}_\varepsilon} \ar@/^/[uull]^-{ \mathsf{triv}_{S^1}} & & 
        }  
    \end{equation}
Here, $\Gamma(-)$ is the functor of derived global sections, $F$ is the functor forgetting the algebra structure, the functors $\mathsf{triv}_{S^1}$ and 
$\mathsf{triv}_{\varepsilon}$ 
send a $k$-module to its trivial representation.  The functors $a$ and $b$ are respectively the (homotopy) $S^1$-invariants and $S^1$-coinvariants 
    $$a \coloneqq (-)^{S^1} \quad \text{and} \quad 
    b \coloneqq (-)_{S^1}\ .
    $$
    The functors $\alpha$ and $\beta$ are given by the following formulas
    $$
    \alpha \coloneqq \Map_{\varepsilon\textrm{-}\mathsf{Mod}_k}(k, -) \quad \text{and} \quad 
    \beta \coloneqq k \otimes_{k[\varepsilon]} (-)\ .
    $$
We have adjunctions
$$
b \dashv \mathsf{triv}_{S^1} \dashv 
a \quad \text{and} \quad 
\beta \dashv \mathsf{triv}_{\varepsilon} \dashv \alpha\ . 
$$
The top two squares are commutative on the nose. As for the bottom triangle, $A_\phi$ intertwines $\mathsf{triv}_{S^1}$ and 
$\mathsf{triv}_{\varepsilon}$:
\begin{equation*}
    A_\phi \circ \mathsf{triv}_{S^1} \simeq \mathsf{triv}_{\varepsilon}\ .
\end{equation*}
Hence $A_\phi$ also intertwines their left and right adjoint functors:
\begin{equation}
\label{eqn:IntertwineAdjoints}
    \alpha\circ A_\phi\simeq a \  \text{ and  }\   \beta\circ A_\phi\simeq b\ .
\end{equation}
Further, we have the \emph{norm map}, which are natural transformations
$$
\nu: (-)_{S^1}[1] \Rightarrow (-)^{S^1} \  \text{ and  }\  \nu:  k \otimes_{k[\varepsilon]} (-) [1] \Rightarrow \Map_{\varepsilon\textrm{-}\mathsf{Mod}_k}(k, -).
$$
We refer to the literature for the former, and recall that the latter is induced by the following morphism of $k[\varepsilon]$-bimodules, where $k[1]$ is the 1-dimensional vector space with basis $\varepsilon$:
\begin{equation}
    \label{eqn:NormMapDef}
k[1] \to k[\varepsilon]\ ,
\end{equation}
together with adjunction:
\begin{equation*}
    \Map_k(k[1]\otimes_{k[\varepsilon]} - , \Map_{k[\varepsilon]}(k, -))\simeq \Map_{k[\varepsilon]}(k\otimes_k k[1]\otimes_{k[\varepsilon]} -, - )\simeq \Map_{k[\varepsilon]}(k[1]\otimes_{k[\varepsilon]}- ,k[\varepsilon]\otimes_{k[\varepsilon]}-).
\end{equation*}
\end{rem}

\begin{cor}
\label{corS^1=dR}
    In the setting of diagram \eqref{diag:Intertwinings}, the Connes operator on $\pi_{2,\ast}(\cO_{\loopstack X})$ corresponds, via the top horizontal equivalence $A_\phi$, to the action of the de Rham differential on $\pi_{1,\ast}(\cO_{\mathsf T[-1]\inertiaDM X})$.
\end{cor}

\begin{proof}
    Invoking \cref{recollection:Toen_Vezzosi}, we deduce that the $\BGa$-action on $\mathsf T[-1]\inertiaDM X$ corresponds to the $(k \oplus k[-1])$-coaction on $\cO_{\mathsf T[-1]\inertiaDM X} \simeq \mathsf{DR}(\cO_{\inertiaDM X})$ given by the de Rham differential of $\inertiaDM X$.
    Similarly, the $S^1$-action on $\loopstack X$ corresponds to the free $\mathsf C^\ast_{\mathsf{sing}}(S^1;k)$-coaction on $\cO_{\loopstack X} \simeq S^1 \otimes \cO_{\inertiaDM X}$.
\end{proof}

\begin{defin}
    Let $X$ be a derived stack. The  \emph{cyclic homology} of $X$ is defined to be the homotopy orbits of the natural $S^1$-action on $\HH_*(X)$:
    $$
\HC(X)\coloneqq (\mathsf{HH}_*(X))_{S^1}\ . 
    $$
    The \emph{negative cyclic homology} of $X$ is defined to be the homotopy fixed points for this action:
    $$
    \HN(X)\coloneqq(\mathsf{HH}_*(X))^{S^1}\ . 
    $$
    The \emph{periodic cyclic homology} of $X$ is defined to be the Tate fixed points of this action, i.e.~the homotopy cofiber of the norm map:
    $$
    \HP(X)\coloneqq (\mathsf{HH}_*(X))^{t S^1}\coloneqq \cofib( \HC(X)[1] \stackrel{\nu} \longrightarrow \HN(X) )\ . 
    $$
It is shown in  \cite{hoyois2015homotopy} that the above definitions coincide with the classical ones, which can be found for example in  \cite{loday2013cyclic}. 
\end{defin}
%Recall that 
%$$
%k^{S^1} \simeq k[u]
%$$
%is the free commutative algebra on generator $u$ of degree two.  Let  
%$k[[u]]$ be the commutative algebra of power series in $u$. 
%These definitions are standard by now, but differ from the  approach in the classical references such as \cite{loday2013cyclic}. A proof that the two points of view are in fact equivalent is contained in   \cite{hoyois2015homotopy}. 
%We also  remark that the definition we gave of  $\HP(X)$ can be reformulated by saying that  $\HP(X)$ is given by the Tate fixed points of the $S^1$-action on  $\mathsf{HH}_*(X)$  

Given a derived DM stack $X$, we define the derived de Rham complex of $\inertiaDM X$ as the totalization of the following mixed graded algebra:
\begin{equation}
\label{eqn:DRIX}
    \DR(\inertiaDM X)\coloneqq \left[\cO_{\inertiaDM X} \xrightarrow{\mathsf{d}_{\dR}} \mathbb{L}_{\inertiaDM X}\xrightarrow{\mathsf{d}_{\dR}} \bigwedge^2\mathbb{L}_{\inertiaDM X}\xrightarrow{\mathsf{d}_{\dR}} \cdots \right],
\end{equation}
where $\cO_{\inertiaDM X} $ is placed in cohomological degree 0, $\bigwedge^i\mathbb{L}_{\inertiaDM X}$ stands for the $i$-th derived exterior power the cotangent complex of $\inertiaDM X$, and the mixed structure is given by  the de Rham differential operator $\mathsf{d}_{\dR}$.
The de Rham complex \eqref{eqn:DRIX} is endowed with a natural Hodge filtration given as follows: for $i\geq 0$,
\begin{equation}
\label{eqn:HodgeFil}
    \Fil_{\mathsf{Hdg}}^i\DR(\inertiaDM X)\coloneqq \left[\cdots \to 0 \to \bigwedge^i\mathbb{L}_{\inertiaDM X}\xrightarrow{\mathsf{d}_{\dR}}\bigwedge^{i+1}\mathbb{L}_{\inertiaDM X} \xrightarrow{\mathsf{d}_{\dR}}\cdots \right],
\end{equation}
and $\Fil_{\mathsf{Hdg}}^{i}\DR(\inertiaDM X)=\DR(\inertiaDM X)$ for $i\leq 0$.

\begin{cor}
\label{cor:HCHNHP}
Let $X$ be a derived DM stack  locally almost of finite presentation.
Denote  $\DR\coloneqq\DR(\inertiaDM X)$ the derived de Rham complex of $\inertiaDM X$.
For any integer $i$, let $\DR^{\geq i}\coloneqq \Fil_{\mathsf{Hdg}}^i\DR$ and $\DR^{\leq i}\coloneqq \cofib(\DR^{\geq i} \to \DR)$.
There are canonical isomorphisms of graded vector spaces:
\begin{align*} 
    \HC(X) & \simeq \bigoplus_{i\geq 0}  \Gamma \left(\inertiaDM X, \DR^{\leq i} [2i]\right)\ ;\\
    \HN(X) &\simeq \prod_{i\in \mathbb{Z}}\Gamma \left(\inertiaDM X,  \DR^{\geq i} [2i]\right)\simeq  \prod_{i\leq 0} H^*_{\dR}(\inertiaDM X)[2i] \times \prod_{i=1}^\infty \Gamma\left(\inertiaDM X, \DR^{\geq i} [2i]\right)\ ;\\
    \HP(X) &\simeq \prod_{i\in \mathbb{Z}} \Gamma \left(\inertiaDM X, \DR[2i]\right) \simeq \prod_{i\in \mathbb{Z}} H^*_{\dR}(\inertiaDM X)[2i]\ .
\end{align*}
\end{cor}
\begin{proof}
    By the commutative diagram \eqref{diag:Intertwinings}, especially \eqref{eqn:IntertwineAdjoints}, we have 
    \begin{align}
    \begin{split}
        \label{eqn:HCHN}
        \HC(X)& \simeq b( \Gamma (\cO_{\loopstack X}))\simeq \beta(A_{\phi}(\Gamma (\cO_{\loopstack X})))\simeq  k\otimes_{k[\varepsilon]}\Gamma (\inertiaDM X, \Sym(\bbL_{\inertiaDM X}[1]))\ ;\\
        \HN(X)& \simeq a( \Gamma (\cO_{\loopstack X}))\simeq \alpha(A_{\phi}(\Gamma (\cO_{\loopstack X})))\simeq \mathsf \Map_{k[\varepsilon]}(k, \Gamma (\inertiaDM X, \Sym(\bbL_{\inertiaDM X}[1])))\ .
    \end{split}
    \end{align}
    Here, the $k[\varepsilon]$-structure on $\Gamma (\Sym(\bbL_{\inertiaDM X}[1]))$ is given by the de Rham differential operator $$\mathsf d_{\dR}\colon \Sym(\bbL_{\inertiaDM X}[1])\to \Sym(\bbL_{\inertiaDM X}[1])[-1].$$
    Resolving the object $k\in \varepsilon\textrm{-}\Mod_k$ by the mixed module
    \begin{equation*}
        \cdots \to k[\varepsilon][2]\xrightarrow{\cdot \varepsilon} k[\varepsilon][1]\xrightarrow{\cdot \varepsilon} k[\varepsilon],
    \end{equation*}
    and combining with \eqref{eqn:HCHN},
   we obtain the claimed formulas for $\HC(X)$ and $\HN(X)$.  

   Combining the description of the norm map in \eqref{eqn:NormMapDef} with the above resolution of $k$, one obtains that 
   the norm map $\HC(X)[1]\to \HN(X)$ is induced by:
   \begin{equation}
       \bigoplus_{i> 0} \DR^{\leq i-1}[2i-1] \xrightarrow{\mathsf d_{\dR}}\prod_{i\in \mathbb{Z}} \DR^{\geq i}[2i]
   \end{equation}
   whose cofiber is $\prod_{i\in \mathbb{Z}} \DR[2i]$, since for $i\leq 0$, $\DR^{\geq i}[2i] = \DR[2i]$. The formula for $\HP(X)$ follows. 
\end{proof}

The next  result  is an immediate consequence of Corollary \ref{cor:HCHNHP}. We  record it as an independent statement as it provides an even more explicit description of the periodic cyclic homology of $X$, which might be of independent interest. 
\begin{cor}
Let $X$ be a derived DM stack locally almost of finite presentation. Let 
$ \inertia   \trunc(X) $ 
be the classical inertia stack of $\trunc(X)$, and let 
$ \, 
|\inertia   \trunc(X) |
$ 
be its coarse moduli space. 
Then there are equivalences  
$$\HP(X)\simeq \HP(\trunc(X)) \simeq \prod_{i\in \mathbb{Z}} H^*_{\dR}(|\inertia   \trunc(X) |)[2i].$$
\end{cor}
\begin{proof}
  Note that  periodic cyclic homology, as de Rham cohomology, is nil-invariant. This implies the first equivalence. Next, by the  last equivalence of Corollary \ref{cor:HCHNHP}, we can write 
   $$
    \HP(\trunc(X)) \simeq  \prod_{i\in \mathbb{Z}} H^*_{\dR}(\inertia \trunc(X))[2i]
   $$
   So the only thing left to prove, is that the latter is equivalent to the $\mathbb{Z}_2$-periodic de Rham cohomology of the coarse moduli space. This is a general property of the de Rham cohomology of DM stacks in characteristic $0$. Let us sketch the argument for a general DM stack $Y$, which we can assume underived. By \'etale descent for de Rham cohomology we reduce to the global quotient case, 
   $ \, 
   Y = [S/G]
   $, 
   where $S$ is a scheme and $G$ is a finite group. Let us denote $|Y|$ the coarse moduli space of $Y$. Using again \'etale descent, we see that the de Rham cohomology of $Y$ is equivalent to the $G$-invariant part of the de Rham cohomology of $S$. But the latter is equivalent to the de Rham cohomology of the coarse moduli space  $|Y|$. See for instance Proposition 36 of \cite{behrend2004cohomology} for a similar argument. 
   %because  there is a canonical identification between the (derived) de Rham cohomology of a derived stack and that of its classical truncation. Indeed, by definition, for any derived stack $Y$, $Y_{\dR}\simeq \trunc(Y)_{\dR}$. Applying $\Gamma(-, \cO)$, one obtains  that $H^*_{\dR}(Y)\simeq H^*_{\dR}(\trunc(Y))$. Taking $Y=\inertiaDM X$, since $\trunc(\inertiaDM X)\simeq \inertia \trunc(X)$ by \cref{thm:basics_inertia}, and we get the last claim in the statement.
\end{proof}

\begin{rem}
We can give also an alternative, Koszul dual, description of $\HC(X)$, $\HN(X)$ and $\HP(X)$. Again, this follows easily  from the commutative  diagram \eqref{diag:Intertwinings}, but we shall not include all the details.  Let $\, k[u] \simeq k^{S^1}$ 
 be the free commutative algebra on a generator $u$ of degree two. 
Then we have isomorphisms 
$$
\HC(X) \simeq  \big ( \Gamma ( \Sym(\bbL_{\inertiaDM X}[1])) \otimes_k \left ( k(u)/u \cdot k[u] \right ) , u \cdot \ddR \big ) \ , 
$$
    $$
\HN(X) \simeq  \big ( \Gamma ( \Sym(\bbL_{\inertiaDM X}[1])) \otimes_k k[u] , u \cdot \ddR \big ) \ .
$$
Periodic cyclic homology is equivalent to the Tate fixed points of the $S^1$-action on $\mathsf{HH}(X)$, 
 $$
 \HP(X) \simeq (\mathsf{HH}(X))^{S^1} \otimes_{k[u]}k[u, u^{-1}]$$
and in particular we have that 
$$
\HP(X) \simeq  \big ( \Gamma ( \Sym(\bbL_{\inertiaDM X}[1])) \otimes_k k(u) , u \cdot \ddR \big ) \ .
$$ 
\end{rem}
%By construction these three invariants fit into a cofiber sequence. Hence, in order to  compute  them, it is enough to work with two of them; indeed, the third one can be recovered as the cone of a natural map. We shall focus on negative cyclic homology and periodic cyclic homology. 

%Let $Y$ be a quasi-compact and quasi-separated classical scheme. We denote by 
%$$
%\mathsf{H}^{\ast, \mathbb{Z}_2}_{dR}(Y)
%$$
%its $\mathbb{Z}_2$-periodic algebraic de Rham cohomology in the sense of Hartshorne \cite{hartshorne1975rham}, see also \cite{bhatt2012completions} for a more modern treatment. Now let $X$ be a quasi-compact and quasi-separated 
%One has that 
%$$
%\HP(X) \simeq 
%$$

%There is a canonical isomorphism between the negative cyclic homology \personal{Maybe it is rather periodic cyclic homology.} of $X$ and the derived de Rham cohomology of $\inertiaDM X$:
%\begin{equation}
 %   \HN_i(X)\simeq \begin{cases}
        %H^{\operatorname{even}}_{\mathsf {dR}}(\inertiaDM X) \qquad \text{if $i$ is even;}\\
        %H^{\operatorname{odd}}_{\mathsf {dR}}(\inertiaDM X) \qquad \text{if $i$ is odd.}
   % \end{cases}
%\end{equation}

\bibliographystyle{amsplain}
\bibliography{References.bib}

@article{ben2023between,
  title={Between coherent and constructible local Langlands correspondences},
  author={Ben-Zvi, David and Chen, Harrison and Helm, David and Nadler, David},
  journal={arXiv preprint arXiv:2302.00039},
  year={2023}
}

@article {Gaitsgory_IndCoh,
    AUTHOR = {Gaitsgory, D.},
     TITLE = {ind-coherent sheaves},
   JOURNAL = {Mosc. Math. J.},
  FJOURNAL = {Moscow Mathematical Journal},
    VOLUME = {13},
      YEAR = {2013},
    NUMBER = {3},
     PAGES = {399--528, 553},
      ISSN = {1609-3321,1609-4514},
   MRCLASS = {14F05 (14D24)},
  MRNUMBER = {3136100},
MRREVIEWER = {Cristian\ V.\ Anghel},
       DOI = {10.17323/1609-4514-2013-13-3-399-528},
       URL = {https://doi.org/10.17323/1609-4514-2013-13-3-399-528},
}

@article {Mathew_Kaledin_Degeneration,
    AUTHOR = {Mathew, A.},
     TITLE = {Kaledin's degeneration theorem and topological {H}ochschild
              homology},
   JOURNAL = {Geom. Topol.},
  FJOURNAL = {Geometry \& Topology},
    VOLUME = {24},
      YEAR = {2020},
    NUMBER = {6},
     PAGES = {2675--2708},
      ISSN = {1465-3060,1364-0380},
   MRCLASS = {16E40 (14A22 55P43)},
  MRNUMBER = {4194301},
MRREVIEWER = {Jiao\ Zhang},
       DOI = {10.2140/gt.2020.24.2675},
       URL = {https://doi.org/10.2140/gt.2020.24.2675},
}

@article {Kaledin_Degeneration,
    AUTHOR = {Kaledin, D.},
     TITLE = {Non-commutative {H}odge-to-de {R}ham degeneration via the
              method of {D}eligne-{I}llusie},
   JOURNAL = {Pure Appl. Math. Q.},
  FJOURNAL = {Pure and Applied Mathematics Quarterly},
    VOLUME = {4},
      YEAR = {2008},
    NUMBER = {3, Special Issue: In honor of Fedor Bogomolov. Part 2},
     PAGES = {785--875},
      ISSN = {1558-8599,1558-8602},
   MRCLASS = {14A22 (14C30 16E40 16S38)},
  MRNUMBER = {2435845},
MRREVIEWER = {Oliver\ R\"{o}ndigs},
       DOI = {10.4310/PAMQ.2008.v4.n3.a8},
       URL = {https://doi.org/10.4310/PAMQ.2008.v4.n3.a8},
}

@article {Schurg_Toen_Vezzosi_Determinant,
    AUTHOR = {Sch\"{u}rg, T. and To\"{e}n, B. and Vezzosi, G.},
     TITLE = {Derived algebraic geometry, determinants of perfect complexes,
              and applications to obstruction theories for maps and
              complexes},
   JOURNAL = {J. Reine Angew. Math.},
  FJOURNAL = {Journal f\"{u}r die Reine und Angewandte Mathematik. [Crelle's
              Journal]},
    VOLUME = {702},
      YEAR = {2015},
     PAGES = {1--40},
      ISSN = {0075-4102,1435-5345},
   MRCLASS = {14D20 (14F05)},
  MRNUMBER = {3341464},
MRREVIEWER = {Nicolas\ Perrin},
       DOI = {10.1515/crelle-2013-0037},
       URL = {https://doi.org/10.1515/crelle-2013-0037},
}

@article {Porta_Sala_Categorified,
    AUTHOR = {Porta, M. and Sala, F.},
     TITLE = {Two-dimensional categorified {H}all algebras},
   JOURNAL = {J. Eur. Math. Soc. (JEMS)},
  FJOURNAL = {Journal of the European Mathematical Society (JEMS)},
    VOLUME = {25},
      YEAR = {2023},
    NUMBER = {3},
     PAGES = {1113--1205},
      ISSN = {1435-9855,1435-9863},
   MRCLASS = {14F08 (14H60 14J60 17B37 55P99)},
  MRNUMBER = {4577961},
MRREVIEWER = {Joan\ Pons-Llopis},
       DOI = {10.4171/jems/1303},
       URL = {https://doi.org/10.4171/jems/1303},
}

@article {Toen_Azumaya,
    AUTHOR = {To{\"e}n, Bertrand},
     TITLE = {Derived {A}zumaya algebras and generators for twisted derived categories},
   JOURNAL = {Invent. Math.},
  FJOURNAL = {Inventiones Mathematicae},
    VOLUME = {189},
      YEAR = {2012},
    NUMBER = {3},
     PAGES = {581--652},
      ISSN = {0020-9910,1432-1297},
   MRCLASS = {18E30 (14F05 14F20 16E35 18D10 18G55)},
  MRNUMBER = {2957304},
MRREVIEWER = {Julia\ Bergner},
       DOI = {10.1007/s00222-011-0372-1},
       URL = {https://doi.org/10.1007/s00222-011-0372-1},
}

@article {Toen_Revisited_schematization,
    AUTHOR = {To{\"e}n, Bertrand},
     TITLE = {Le problème de la schématisation de {G}rothendieck revisité},
   JOURNAL = {\'{E}pijournal G\'{e}om. Alg\'{e}brique},
  FJOURNAL = {\'{E}pijournal de G\'{e}om\'{e}trie Alg\'{e}brique. EPIGA},
    VOLUME = {4},
      YEAR = {2020},
     PAGES = {Art. 14, 21},
      ISSN = {2491-6765},
   MRCLASS = {55P15 (14A20 18N60 55P20 55P60)},
  MRNUMBER = {4191417},
MRREVIEWER = {J\'{e}r\^{o}me\ Scherer},
       DOI = {10.46298/epiga.2020.volume4.6060},
       URL = {https://doi.org/10.46298/epiga.2020.volume4.6060},
 }

@article{mondal2026unipotent,
  title={Unipotent homotopy theory of schemes},
  author={Mondal, Shubhodip and Reinecke, Emanuel},
  journal={Journal of the American Mathematical Society},
  volume={39},
  number={1},
  pages={205--312},
  year={2026}
}

@article {HPT_Homotopy_invariance,
    AUTHOR = {Haine, Peter J. and Porta, Mauro and Teyssier, Jean-Baptiste},
     TITLE = {The homotopy-invariance of constructible sheaves},
   JOURNAL = {Homology Homotopy Appl.},
  FJOURNAL = {Homology, Homotopy and Applications},
    VOLUME = {25},
      YEAR = {2023},
    NUMBER = {2},
     PAGES = {97--128},
      ISSN = {1532-0073,1532-0081},
   MRCLASS = {55N30 (18F20 32S60 55N05 55P55)},
  MRNUMBER = {4651175},
MRREVIEWER = {Daniel\ Dugger},
       DOI = {10.4310/hha.2023.v25.n2.a6},
       URL = {https://doi.org/10.4310/hha.2023.v25.n2.a6},
}

@misc{Raksit,
      title={Hochschild homology and the derived de Rham complex revisited}, 
      author={Arpon Raksit},
      year={2026},
      journal={arXiv preprint arXiv:2007.02576}
}

@article {Ragimov_Schlank,
    AUTHOR = {Ragimov, Shaul and Schlank, Tomer M.},
     TITLE = {The {$\infty$}-categorical reflection theorem and
              applications},
   JOURNAL = {J. Topol.},
  FJOURNAL = {Journal of Topology},
    VOLUME = {19},
      YEAR = {2026},
    NUMBER = {1},
     PAGES = {Paper No. e70060, 58},
      ISSN = {1753-8416,1753-8424},
   MRCLASS = {18N55},
  MRNUMBER = {5009303},
       DOI = {10.1112/topo.70060},
       URL = {https://doi.org/10.1112/topo.70060},
}

@article{Brantner_Magidson_Nuiten,
    title={Formal Integration of Derived Foliations}, 
    author={Brantner, Lukas and Magidson, Kirill and Nuiten, Joost},
    year={2025},
    journal={arXiv preprint arXiv:2502.05257},
    year={2025}
}

@book {Gaitsgory_Rozenblyum_I,
    AUTHOR = {Gaitsgory, Dennis and Rozenblyum, Nick},
     TITLE = {A study in derived algebraic geometry. {V}ol. {I}.
              {C}orrespondences and duality},
    SERIES = {Mathematical Surveys and Monographs},
    VOLUME = {221},
 PUBLISHER = {American Mathematical Society, Providence, RI},
      YEAR = {2017},
     PAGES = {xl+533pp},
      ISBN = {978-1-4704-3569-1},
   MRCLASS = {14F05 (18D05 18G55)},
  MRNUMBER = {3701352},
MRREVIEWER = {Adrian\ Langer},
       DOI = {10.1090/surv/221.1},
       URL = {https://doi.org/10.1090/surv/221.1},
}

@book {Gaitsgory_Rozenblyum_II,
    AUTHOR = {Gaitsgory, Dennis and Rozenblyum, Nick},
     TITLE = {A study in derived algebraic geometry. {V}ol. {II}.
              {D}eformations, {L}ie theory and formal geometry},
    SERIES = {Mathematical Surveys and Monographs},
    VOLUME = {221},
 PUBLISHER = {American Mathematical Society, Providence, RI},
      YEAR = {2017},
     PAGES = {xxxv+436},
      ISBN = {978-1-4704-3570-7},
   MRCLASS = {14F05 (18D05 18G55)},
  MRNUMBER = {3701353},
MRREVIEWER = {Adrian\ Langer},
       DOI = {10.1090/surv/221.2},
       URL = {https://doi.org/10.1090/surv/221.2},
}

@article {Abramovich_Olsson_Vistoli_Tame_stacks,
    AUTHOR = {Abramovich, Dan and Olsson, Martin and Vistoli, Angelo},
     TITLE = {Tame stacks in positive characteristic},
   JOURNAL = {Ann. Inst. Fourier (Grenoble)},
  FJOURNAL = {Universit\'{e} de Grenoble. Annales de l'Institut Fourier},
    VOLUME = {58},
      YEAR = {2008},
    NUMBER = {4},
     PAGES = {1057--1091},
      ISSN = {0373-0956,1777-5310},
   MRCLASS = {14A20 (14L15)},
  MRNUMBER = {2427954},
MRREVIEWER = {Hsian-Hua\ Tseng},
       DOI = {10.5802/aif.2378},
       URL = {https://doi.org/10.5802/aif.2378},
}

@article{behrend2004cohomology,
  title={Cohomology of stacks},
  author={Behrend, Kai},
  journal={Intersection theory and moduli, ICTP Lect. Notes},
  volume={19},
  pages={249--294},
  year={2004}
}

@book{loday2013cyclic,
  title={Cyclic homology},
  author={Loday, Jean-Louis},
  volume={301},
  year={2013},
  publisher={Springer Science \& Business Media}
}

@article{hoyois2015homotopy,
  title={The homotopy fixed points of the circle action on {H}ochschild homology},
  author={Hoyois, Marc},
  journal={arXiv preprint arXiv:1506.07123},
  year={2015}
}

@book {Andre_Homologie,
    AUTHOR = {André, M.},
     TITLE = {Homologie des algèbres commutatives},
    SERIES = {Die Grundlehren der mathematischen Wissenschaften, Band 206},
 PUBLISHER = {Springer-Verlag, Berlin-New York},
      YEAR = {1974},
     PAGES = {xv+341},
   MRCLASS = {18H20 (12GXX 13DXX 14D15)},
  MRNUMBER = {352220},
MRREVIEWER = {R. M. Fossum},
}

@inproceedings {Quillen_Homology_of_commutative_rings,
    AUTHOR = {Quillen, D.},
     TITLE = {On the (co-) homology of commutative rings},
 BOOKTITLE = {Applications of {C}ategorical {A}lgebra ({P}roc. {S}ympos.
              {P}ure {M}ath., {V}ol. {XVII}, {N}ew {Y}ork, 1968)},
    SERIES = {Proc. Sympos. Pure Math., XVII},
     PAGES = {65--87},
 PUBLISHER = {Amer. Math. Soc., Providence, RI},
      YEAR = {1970},
   MRCLASS = {13.90 (18.00)},
  MRNUMBER = {257068},
MRREVIEWER = {S. Yuan},
}

@article{Kinjo_Multiplicative_dimensional_reduction,
  title={Multiplicative dimensional reduction},
  author={Kinjo, T.},
  journal={arXiv preprint arXiv:2511.16342},
  year={2025}
}

@article {AKLPR,
    AUTHOR = {Aranha, D.  and Khan, A. A. and Latyntsev, A. and Park, H. and Ravi, C.},
     TITLE = {Virtual localization revisited},
   JOURNAL = {Adv. Math.},
  FJOURNAL = {Advances in Mathematics},
    VOLUME = {479},
      YEAR = {2025},
    NUMBER = {part A},
     PAGES = {Paper No. 110434, 53},
      ISSN = {0001-8708},
   MRCLASS = {14D23 (14A20 14A30 14C15 14L30 14N35)},
  MRNUMBER = {4933128},
       DOI = {10.1016/j.aim.2025.110434},
       URL = {https://doi.org/10.1016/j.aim.2025.110434},
}

@inproceedings{antieau2021counterexamples,
  title={Counterexamples to {H}ochschild-{K}ostant-{R}osenberg in characteristic p},
  author={Antieau, Benjamin and Bhatt, Bhargav and Mathew, Akhil},
  booktitle={Forum of Mathematics, Sigma},
  volume={9},
  pages={e49},
  year={2021},
  organization={Cambridge University Press}
}

@article{sibilla2023equivariant,
  title={Equivariant elliptic cohomology and mapping stacks I},
  author={Sibilla, Nicolo and Tomasini, Paolo},
  journal={arXiv preprint arXiv:2303.10146},
  year={2023}
}

@article{Rychlewiczinpreparation,
  title={In preparation},
  author={Rychlewicz, Kamil and Scherotzke, Sarah and Sibilla, Nicolò and Tomasini, Paolo},
}

@article{chen2020equivariant,
  title={Equivariant localization and completion in cyclic homology and derived loop spaces},
  author={Chen, Harrison},
  journal={Advances in Mathematics},
  volume={364},
  pages={107005},
  year={2020},
  publisher={Elsevier}
}

@article{bodzenta2024root,
  title={Root stacks and periodic decompositions},
  author={Bodzenta, Agnieszka and Donovan, Will},
  journal={manuscripta mathematica},
  volume={175},
  number={1},
  pages={53--73},
  year={2024},
  publisher={Springer}
}

@article{scherotzke2020parabolic,
  title={Parabolic semi-orthogonal decompositions and kummer flat invariants of log schemes},
  author={Scherotzke, Sarah and Sibilla, Nicolo and Talpo, Mattia},
  journal={Documenta Mathematica},
  volume={25},
  pages={955--1009},
  year={2020}
}

@article{bergh2016geometricity,
  title={Geometricity for derived categories of algebraic stacks},
  author={Bergh, Daniel and Lunts, Valery A and Schn{\"u}rer, Olaf M},
  journal={Selecta Mathematica},
  volume={22},
  number={4},
  pages={2535--2568},
  year={2016},
  publisher={Springer}
}

@article{ishii2015special,
  title={The special McKay correspondence and exceptional collections},
  author={Ishii, Akira and Ueda, Kazushi},
  journal={Tohoku Mathematical Journal, Second Series},
  volume={67},
  number={4},
  pages={585--609},
  year={2015},
  publisher={Mathematical Institute, Tohoku University}
}

@article{talpo2018infinite,
  title={Infinite root stacks and quasi-coherent sheaves on logarithmic schemes},
  author={Talpo, Mattia and Vistoli, Angelo},
  journal={Proceedings of the London Mathematical Society},
  volume={116},
  number={5},
  pages={1187--1243},
  year={2018},
  publisher={Wiley Online Library}
}

@article{cadman2007using,
  title={Using stacks to impose tangency conditions on curves},
  author={Cadman, Charles},
  journal={American journal of mathematics},
  volume={129},
  number={2},
  pages={405--427},
  year={2007},
  publisher={Johns Hopkins University Press}
}

@article{connes1985non,
  title={Non-commutative differential geometry},
  author={Connes, Alain},
  journal={Publications Mathematiques de l'IHES},
  volume={62},
  pages={41--144},
  year={1985}
}

@article{tsygan1983homology,
  title={The homology of matrix Lie algebras over rings and the Hochschild homology},
  author={Tsygan, Boris L},
  journal={Russian Mathematical Surveys},
  volume={38},
  number={2},
  pages={198},
  year={1983},
  publisher={IOP Publishing}
}

@article {Arinkin_Caldararu_Hablicsek,
    AUTHOR = {Arinkin, D. and C\u{a}ld\u{a}raru, A. and Hablicsek, M.},
     TITLE = {Formality of derived intersections and the orbifold {HKR}
              isomorphism},
   JOURNAL = {J. Algebra},
  FJOURNAL = {Journal of Algebra},
    VOLUME = {540},
      YEAR = {2019},
     PAGES = {100--120},
      ISSN = {0021-8693},
   MRCLASS = {14F08},
  MRNUMBER = {4003476},
MRREVIEWER = {Mee Seong Im},
       DOI = {10.1016/j.jalgebra.2019.08.002},
       URL = {https://doi.org/10.1016/j.jalgebra.2019.08.002},
}

@article {BenZvi_Nadler_Loop_spaces_and_connections,
    AUTHOR = {Ben-Zvi, D. and Nadler, D.},
     TITLE = {Loop spaces and connections},
   JOURNAL = {J. Topol.},
  FJOURNAL = {Journal of Topology},
    VOLUME = {5},
      YEAR = {2012},
    NUMBER = {2},
     PAGES = {377--430},
      ISSN = {1753-8416},
   MRCLASS = {14F05 (14F10 19D55 55P35)},
  MRNUMBER = {2928082},
MRREVIEWER = {Julia Bergner},
       DOI = {10.1112/jtopol/jts007},
       URL = {https://doi.org/10.1112/jtopol/jts007},
}

@article {HKR,
    AUTHOR = {Hochschild, G. and Kostant, B. and Rosenberg, A.},
     TITLE = {Differential forms on regular affine algebras},
   JOURNAL = {Trans. Amer. Math. Soc.},
  FJOURNAL = {Transactions of the American Mathematical Society},
    VOLUME = {102},
      YEAR = {1962},
     PAGES = {383--408},
      ISSN = {0002-9947},
   MRCLASS = {18.20 (14.52)},
  MRNUMBER = {142598},
MRREVIEWER = {J. W. Gray},
       DOI = {10.2307/1993614},
       URL = {https://doi.org/10.2307/1993614},
}

@article {MRT,
    AUTHOR = {Moulinos, Tasos and Robalo, Marco and To\"{e}n, Bertrand},
     TITLE = {A universal {H}ochschild-{K}ostant-{R}osenberg theorem},
   JOURNAL = {Geom. Topol.},
  FJOURNAL = {Geometry \& Topology},
    VOLUME = {26},
      YEAR = {2022},
    NUMBER = {2},
     PAGES = {777--874},
      ISSN = {1465-3060},
   MRCLASS = {14A30 (19D55)},
  MRNUMBER = {4444269},
MRREVIEWER = {Rui Miguel Saramago},
       DOI = {10.2140/gt.2022.26.777},
       URL = {https://doi.org/10.2140/gt.2022.26.777},
}

@misc{stacks-project,
	author       = {The {Stacks project authors}},
	title        = {The Stacks project},
	howpublished = {\url{https://stacks.math.columbia.edu}},
	year         = {2025},
}

@article{Rozenblyum_Connections_infinitesimal_Hecke,
  title={Connections on moduli spaces and infinitesimal Hecke modifications},
  author={Rozenblyum, Nick},
  journal={arXiv preprint arXiv:2108.07745},
  year={2021}
}

@article{Porta_HDR,
	title={Derived methods in moduli theory},
	author={Porta, M.},
	journal={HAL, tel-04984799},
	year={2025}
}

@article{Binda_Porta_Azumaya,
	title={{GAGA} problems for the {B}rauer group via derived geometry},
	author={Binda, F. and Porta, M.},
	journal={arXiv preprint arXiv:2107.03914},
	year={2021}
}

@article {Bergh_Schnurer,
	AUTHOR = {Bergh, D. and Schn\"{u}rer, O. M.},
	TITLE = {Decompositions of derived categories of gerbes and of families
	of {B}rauer-{S}everi varieties},
	JOURNAL = {Doc. Math.},
	FJOURNAL = {Documenta Mathematica},
	VOLUME = {26},
	YEAR = {2021},
	PAGES = {1465--1500},
	ISSN = {1431-0635},
	MRCLASS = {14F08 (14A20)},
	MRNUMBER = {4334847},
	MRREVIEWER = {Giulio Bresciani},
}

@article{Naef_Safronov,
	title={Simple homotopy invariance of the loop coproduct},
	author={Naef, F. and Safronov, P.},
	journal={arXiv preprint arXiv:2406.19326},
	year={2024}
}

@article {Toen_Champs_affines,
	AUTHOR = {To{\"e}n, B.},
	TITLE = {Champs affines},
	JOURNAL = {Selecta Math. (N.S.)},
	FJOURNAL = {Selecta Mathematica. New Series},
	VOLUME = {12},
	YEAR = {2006},
	NUMBER = {1},
	PAGES = {39--135},
	ISSN = {1022-1824},
	MRCLASS = {14F35 (14A20 18F10 55U35)},
	MRNUMBER = {2244263},
	MRREVIEWER = {Mark Hovey},
	DOI = {10.1007/s00029-006-0019-z},
	URL = {https://doi.org/10.1007/s00029-006-0019-z},
}

@article {Porta_Yu_Hom,
	AUTHOR = {Porta, M. and Yu, T. Y.},
	TITLE = {Derived {H}om spaces in rigid analytic geometry},
	JOURNAL = {Publ. Res. Inst. Math. Sci.},
	FJOURNAL = {Publications of the Research Institute for Mathematical
	Sciences},
	VOLUME = {57},
	YEAR = {2021},
	NUMBER = {3-4},
	PAGES = {921--958},
	ISSN = {0034-5318},
	MRCLASS = {14G22 (14D23 18G80)},
	MRNUMBER = {4322003},
	DOI = {10.4171/prims/57-3-7},
	URL = {https://doi.org/10.4171/prims/57-3-7},
}

@article{Porta_Yu_NQK,
	author={Porta, M. and Yu, T. Y.},
	TITLE = {Non-archimedean quantum {K}-invariants},
	JOURNAL = {Ann. Sci. \'{E}c. Norm. Sup\'{e}r. (4)},
	FJOURNAL = {Annales Scientifiques de l'\'{E}cole Normale Sup\'{e}rieure. Quatri\`eme
	S\'{e}rie},
	VOLUME = {57},
	YEAR = {2024},
	NUMBER = {3},
	PAGES = {713--786},
	ISSN = {0012-9593},
	MRCLASS = {14G22 (18N65)},
	MRNUMBER = {4773296},
}

@article {Halpern_Leistner_Preygel_Categorical_properness,
	AUTHOR = {Halpern-Leistner, D. and Preygel, A.},
	TITLE = {Mapping stacks and categorical notions of properness},
	JOURNAL = {Compos. Math.},
	FJOURNAL = {Compositio Mathematica},
	VOLUME = {159},
	YEAR = {2023},
	NUMBER = {3},
	PAGES = {530--589},
	ISSN = {0010-437X},
	MRCLASS = {14A20 (14A30)},
	MRNUMBER = {4560539},
	MRREVIEWER = {Jon Eivind Vatne},
	DOI = {10.1112/S0010437X22007667},
	URL = {https://doi.org/10.1112/S0010437X22007667},
}

@article {Abramovich_GW_theory_for_DM_stacks,
	AUTHOR = {Abramovich, D. and Graber, T. and Vistoli, A.},
	TITLE = {Gromov-{W}itten theory of {D}eligne-{M}umford stacks},
	JOURNAL = {Amer. J. Math.},
	FJOURNAL = {American Journal of Mathematics},
	VOLUME = {130},
	YEAR = {2008},
	NUMBER = {5},
	PAGES = {1337--1398},
	ISSN = {0002-9327},
	MRCLASS = {14N35 (14A20 53D45)},
	MRNUMBER = {2450211},
	MRREVIEWER = {Johannes Walcher},
	DOI = {10.1353/ajm.0.0017},
	URL = {https://doi.org/10.1353/ajm.0.0017},
}

@incollection {Abramovich_Orbifold_quantum_product,
	AUTHOR = {Abramovich, D. and Graber, T. and Vistoli, A.},
	TITLE = {Algebraic orbifold quantum products},
	BOOKTITLE = {Orbifolds in mathematics and physics ({M}adison, {WI}, 2001)},
	SERIES = {Contemp. Math.},
	VOLUME = {310},
	PAGES = {1--24},
	PUBLISHER = {Amer. Math. Soc., Providence, RI},
	YEAR = {2002},
	MRCLASS = {14N35 (14A20 53D45)},
	MRNUMBER = {1950940},
	MRREVIEWER = {Tyler J. Jarvis},
	DOI = {10.1090/conm/310/05397},
	URL = {https://doi.org/10.1090/conm/310/05397},
}

@article {Porta_Comparison,
	AUTHOR = {Porta, M.},
	TITLE = {Comparison results for derived {D}eligne-{M}umford stacks},
	JOURNAL = {Pacific J. Math.},
	FJOURNAL = {Pacific Journal of Mathematics},
	VOLUME = {287},
	YEAR = {2017},
	NUMBER = {1},
	PAGES = {177--197},
	ISSN = {0030-8730},
	MRCLASS = {14A20 (18B25 18D10 18F20 18G55)},
	MRNUMBER = {3613438},
	MRREVIEWER = {Lennart Meier},
	DOI = {10.2140/pjm.2017.287.177},
	URL = {https://doi.org/10.2140/pjm.2017.287.177}
}

@article{BenZvi_Nadler_Drinfeld_centers,
	author = {Ben-Zvi, D. and Francis, J. and Nadler, D.},
	doi = {10.1090/S0894-0347-10-00669-7},
	fjournal = {Journal of the American Mathematical Society},
	issn = {0894-0347},
	journal = {J. Amer. Math. Soc.},
	mrclass = {14D23 (14F05 18D10 18E30)},
	mrnumber = {2669705},
	mrreviewer = {Andrei D. Halanay},
	number = {4},
	pages = {909--966},
	title = {Integral transforms and {D}rinfeld centers in derived algebraic geometry},
	url = {https://doi-org/10.1090/S0894-0347-10-00669-7},
	volume = {23},
	year = {2010}}

@book{HTT,
	address = {Princeton, NJ},
	author = {Lurie, J.},
	isbn = {978-0-691-14049-0; 0-691-14049-9},
	key = {2},
	mrclass = {18-02 (18B25 18E35 18G30 18G55 55U40)},
	mrnumber = {2522659 (2010j:18001)},
	mrreviewer = {Mark Hovey},
	pages = {xviii+925},
	publisher = {Princeton University Press},
	series = {Annals of Mathematics Studies},
	shorthand = {HTT},
	title = {Higher topos theory},
	volume = {170},
	year = {2009}}

@unpublished{HA,
	author = {Lurie, J.},
	ids = {Lurie_Higher_algebra, Lurie_HA},
	key = {3},
	month = {September},
	note = {\href{http://www.math.ias.edu/~lurie/papers/HA.pdf}{\nolinkurl{math.ias.edu/~lurie/papers/HA.pdf}}},
	shorthand = {HA},
	title = {Higher Algebra},
	year = {2017}}

@unpublished{SAG,
	author = {Lurie, J.},
	key = {4},
	month = {February},
	note = {\href{http://www.math.ias.edu/~lurie/papers/SAG-rootfile.pdf}{\nolinkurl{math.ias.edu/~lurie/papers/SAG-rootfile.pdf}}},
	shorthand = {SAG},
	title = {Spectral Algebraic Geometry},
	year = {2018}}

@article{HAG-II,
	author = {To\"en, B. and Vezzosi, G.},
	coden = {MAMCAU},
	fjournal = {Memoirs of the American Mathematical Society},
	issn = {0065-9266},
	journal = {Mem. Amer. Math. Soc.},
	mrclass = {14A20 (18F10 18G55 55P42 55U40)},
	mrnumber = {MR2394633},
	number = {902},
	pages = {x+224},
	title = {Homotopical algebraic geometry. {II}. {G}eometric stacks and applications},
	volume = {193},
	year = {2008}}

@unpublished{G,
	author = {Lurie, J.},
	eprint = {0905.0462},
	month = {October},
	note = {Preprint from the web page of the author},
	title = {$(\infty,2)$-categories and the {G}oodwillie Calculus {I}},
	url = {http://arxiv.org/abs/0905.0462},
	year = {2009}}

@article{Toen_Vezzosi_S1_algebras,
	author = {To\"{e}n, B. and Vezzosi, G.},
	doi = {10.1112/S0010437X11005501},
	fjournal = {Compositio Mathematica},
	issn = {0010-437X,1570-5846},
	journal = {Compos. Math.},
	mrclass = {18G55 (14F40 16E45 18G30 55U10)},
	mrnumber = {2862069},
	mrreviewer = {Kenneth\ A.\ Brown},
	number = {6},
	pages = {1979--2000},
	title = {Alg\`ebres simpliciales {$S^1$}-\'{e}quivariantes, th\'{e}orie de de {R}ham et th\'{e}or\`emes {HKR} multiplicatifs},
	url = {https://doi.org/10.1112/S0010437X11005501},
	volume = {147},
	year = {2011}}

@article {Edidin_Jarvis_Kimura,
    AUTHOR = {Edidin, D. and Jarvis, T. J. and Kimura, T.},
     TITLE = {Logarithmic trace and orbifold products},
   JOURNAL = {Duke Math. J.},
  FJOURNAL = {Duke Mathematical Journal},
    VOLUME = {153},
      YEAR = {2010},
    NUMBER = {3},
     PAGES = {427--473},
      ISSN = {0012-7094,1547-7398},
   MRCLASS = {14L30 (14A20 14C15 14N35 19L47 55N91)},
  MRNUMBER = {2667422},
MRREVIEWER = {Hsian-Hua\ Tseng},
       DOI = {10.1215/00127094-2010-028},
       URL = {https://doi.org/10.1215/00127094-2010-028},
}

@article {Prill,
    AUTHOR = {Prill, David},
     TITLE = {Local classification of quotients of complex manifolds by
              discontinuous groups},
   JOURNAL = {Duke Math. J.},
  FJOURNAL = {Duke Mathematical Journal},
    VOLUME = {34},
      YEAR = {1967},
     PAGES = {375--386},
      ISSN = {0012-7094,1547-7398},
   MRCLASS = {32.65 (10.23)},
  MRNUMBER = {210944},
MRREVIEWER = {E.\ Brieskorn},
       URL = {http://projecteuclid.org/euclid.dmj/1077377006},
}

@article {Satake-orbifold,
    AUTHOR = {Satake, I.},
     TITLE = {On a generalization of the notion of manifold},
   JOURNAL = {Proc. Nat. Acad. Sci. U.S.A.},
  FJOURNAL = {Proceedings of the National Academy of Sciences of the United
              States of America},
    VOLUME = {42},
      YEAR = {1956},
     PAGES = {359--363},
      ISSN = {0027-8424},
   MRCLASS = {55.0X},
  MRNUMBER = {79769},
MRREVIEWER = {H.\ Samelson},
       DOI = {10.1073/pnas.42.6.359},
       URL = {https://doi.org/10.1073/pnas.42.6.359},
}

@article {Edidin-Hassett-Kresch-Vistoli,
    AUTHOR = {Edidin, Dan and Hassett, Brendan and Kresch, Andrew and
              Vistoli, Angelo},
     TITLE = {Brauer groups and quotient stacks},
   JOURNAL = {Amer. J. Math.},
  FJOURNAL = {American Journal of Mathematics},
    VOLUME = {123},
      YEAR = {2001},
    NUMBER = {4},
     PAGES = {761--777},
      ISSN = {0002-9327,1080-6377},
   MRCLASS = {14A20 (14F22)},
  MRNUMBER = {1844577},
MRREVIEWER = {Gabriele\ Vezzosi},
       URL =
              {http://muse.jhu.edu/journals/american_journal_of_mathematics/v123/123.4edidin.pdf},
}

@article {Kai-Noohi,
    AUTHOR = {Behrend, Kai and Noohi, Behrang},
     TITLE = {Uniformization of {D}eligne-{M}umford curves},
   JOURNAL = {J. Reine Angew. Math.},
  FJOURNAL = {Journal f\"ur die Reine und Angewandte Mathematik. [Crelle's
              Journal]},
    VOLUME = {599},
      YEAR = {2006},
     PAGES = {111--153},
      ISSN = {0075-4102,1435-5345},
   MRCLASS = {14D20 (14A20 14H30)},
  MRNUMBER = {2279100},
MRREVIEWER = {Hsian-Hua\ Tseng},
       DOI = {10.1515/CRELLE.2006.080},
       URL = {https://doi.org/10.1515/CRELLE.2006.080},
}

@article {Swan-HH,
    AUTHOR = {Swan, Richard G.},
     TITLE = {Hochschild cohomology of quasiprojective schemes},
   JOURNAL = {J. Pure Appl. Algebra},
  FJOURNAL = {Journal of Pure and Applied Algebra},
    VOLUME = {110},
      YEAR = {1996},
    NUMBER = {1},
     PAGES = {57--80},
      ISSN = {0022-4049,1873-1376},
   MRCLASS = {19D55 (14F99 18G40 18G60)},
  MRNUMBER = {1390671},
MRREVIEWER = {Sue\ Geller},
       DOI = {10.1016/0022-4049(95)00091-7},
       URL = {https://doi.org/10.1016/0022-4049(95)00091-7},
}

@article {GellerWeibel-EtaleDescentHH,
    AUTHOR = {Weibel, Charles A. and Geller, Susan C.},
     TITLE = {\'Etale descent for {H}ochschild and cyclic homology},
   JOURNAL = {Comment. Math. Helv.},
  FJOURNAL = {Commentarii Mathematici Helvetici},
    VOLUME = {66},
      YEAR = {1991},
    NUMBER = {3},
     PAGES = {368--388},
      ISSN = {0010-2571,1420-8946},
   MRCLASS = {19E99 (13D03 18G50 19D55)},
  MRNUMBER = {1120653},
MRREVIEWER = {Jean-Marc\ Piveteau},
       DOI = {10.1007/BF02566656},
       URL = {https://doi.org/10.1007/BF02566656},
}

@incollection {Geigle-Lenzing_Weighted_Projective_Lines,
    AUTHOR = {Geigle, Werner and Lenzing, Helmut},
     TITLE = {A class of weighted projective curves arising in
              representation theory of finite-dimensional algebras},
 BOOKTITLE = {Singularities, representation of algebras, and vector bundles
              ({L}ambrecht, 1985)},
    SERIES = {Lecture Notes in Math.},
    VOLUME = {1273},
     PAGES = {265--297},
 PUBLISHER = {Springer, Berlin},
      YEAR = {1987},
      ISBN = {3-540-18263-2},
   MRCLASS = {14H45 (14F05 16A64)},
  MRNUMBER = {915180},
MRREVIEWER = {Alfred\ G.\ Wiedemann},
       DOI = {10.1007/BFb0078849},
       URL = {https://doi.org/10.1007/BFb0078849},
}

@article{Okawa-Sano-NCRigidity,
      title={On the noncommutative rigidity of the moduli stack of stable pointed curves}, 
      author={Shinnosuke Okawa and Taro Sano},
      year={2019},
      eprint={1412.7060},
      archivePrefix={arXiv},
      primaryClass={math.AG},
      url={https://arxiv.org/abs/1412.7060}, 
}

@article {Hacking-MgnRigid,
    AUTHOR = {Hacking, Paul},
     TITLE = {The moduli space of curves is rigid},
   JOURNAL = {Algebra Number Theory},
  FJOURNAL = {Algebra \& Number Theory},
    VOLUME = {2},
      YEAR = {2008},
    NUMBER = {7},
     PAGES = {809--818},
      ISSN = {1937-0652,1944-7833},
   MRCLASS = {14H10 (14B12)},
  MRNUMBER = {2460695},
MRREVIEWER = {Gavril\ Farkas},
       DOI = {10.2140/ant.2008.2.809},
       URL = {https://doi-org.scd-rproxy.u-strasbg.fr/10.2140/ant.2008.2.809},
}

@article {Romagny-GroupAction,
    AUTHOR = {Romagny, Matthieu},
     TITLE = {Group actions on stacks and applications},
   JOURNAL = {Michigan Math. J.},
  FJOURNAL = {Michigan Mathematical Journal},
    VOLUME = {53},
      YEAR = {2005},
    NUMBER = {1},
     PAGES = {209--236},
      ISSN = {0026-2285,1945-2365},
   MRCLASS = {14A20 (14H10)},
  MRNUMBER = {2125542},
MRREVIEWER = {Ivan\ S.\ Kausz},
       DOI = {10.1307/mmj/1114021093},
       URL = {https://doi-org.scd-rproxy.u-strasbg.fr/10.1307/mmj/1114021093},
}

@article{Robalo-ChoicesHKR,
      title={Choices of {HKR} isomorphisms}, 
      author={Marco Robalo},
   JOURNAL = {Math. Res. Lett, to appear},
      year={2025},
      eprint={2310.05859},
      archivePrefix={arXiv},
      primaryClass={math.AG},
      url={https://arxiv.org/abs/2310.05859}, 
}

@article {Happel-HHPiecewiseHereditaryAlg,
    AUTHOR = {Happel, Dieter},
     TITLE = {Hochschild cohomology of piecewise hereditary algebras},
   JOURNAL = {Colloq. Math.},
  FJOURNAL = {Colloquium Mathematicum},
    VOLUME = {78},
      YEAR = {1998},
    NUMBER = {2},
     PAGES = {261--266},
      ISSN = {0010-1354,1730-6302},
   MRCLASS = {16E40 (16G10 16G60)},
  MRNUMBER = {1659132},
MRREVIEWER = {Teimuraz\ Pirashvili},
       DOI = {10.4064/cm-78-2-261-266},
       URL = {https://doi.org/10.4064/cm-78-2-261-266},
}

@article{Schremmer-HHofWPL-Masterthesis,
      title={Weighted projective lines and {H}ochschild cohomology}, 
      author={Felix Schremmer},
    journal= {Master thesis, arXiv: 2512.08414},
      year={2025},
      eprint={2512.08414},
      archivePrefix={arXiv},
      primaryClass={math.AG},
      url={https://arxiv.org/abs/2512.08414}, 
}

@article{ScherotzkeSibillaTomasini-FourierMukai,
      title={Fourier--Mukai equivalences for formal groups and elliptic Hochschild homology}, 
      author={Sarah Scherotzke and Nicolò Sibilla and Paolo Tomasini},
      year={2025},
      eprint={2505.00172},
      archivePrefix={arXiv},
      primaryClass={math.AG},
      url={https://arxiv.org/abs/2505.00172}, 
}

@article {Brantner-Mathew-PartitionLie,
    AUTHOR = {Brantner, D. Lukas B. and Mathew, Akhil},
     TITLE = {Deformation theory and partition {L}ie algebras},
   JOURNAL = {Acta Math.},
  FJOURNAL = {Acta Mathematica},
    VOLUME = {235},
      YEAR = {2025},
    NUMBER = {1},
     PAGES = {1--148},
      ISSN = {0001-5962,1871-2509},
   MRCLASS = {14D06 (17B55 32G13)},
  MRNUMBER = {4995932},
       DOI = {10.4310/acta.2025.v235.n1.a1},
       URL = {https://doi.org/10.4310/acta.2025.v235.n1.a1},
}

@book {Hazewinkel,
    AUTHOR = {Hazewinkel, Michiel},
     TITLE = {Formal groups and applications},
    SERIES = {Pure and Applied Mathematics},
    VOLUME = {78},
 PUBLISHER = {Academic Press, Inc. [Harcourt Brace Jovanovich, Publishers],
              New York-London},
      YEAR = {1978},
     PAGES = {xxii+573},
      ISBN = {0-12-335150-2},
   MRCLASS = {14L05 (10D30 12B25 55N22)},
  MRNUMBER = {506881},
MRREVIEWER = {Robert\ A.\ Morris},
}

@article {FPSS-FoldedFilteredCircle,
    AUTHOR = {Lie Fu and Mauro Porta and Sarah Scherotzke and Nicolò Sibilla},
    TITLE = {Folded filtered circle and {H}ochschild--{K}ostant--{R}osenberg filtration for {D}eligne--{M}umford stacks},
    year={2026},
    Journal={In preparation},
}

@article{PTVV,
  author = {Pantev, Tony and To\"en, Bertrand and Vaqui\'e, Michel and Vezzosi, Gabriele},
  title = {Shifted symplectic structures},
  journal = {Publ. Math. Inst. Hautes \'Etudes Sci.},
  fjournal = {Publications Math\'ematiques. Institut de Hautes \'Etudes Scientifiques},
  volume = {117},
  year = {2013},
  pages = {271--328},
  issn = {0073-8301,1618-1913},
  mrclass = {14F05 (14A15 18F20 18G30 53D05 53D12)},
  mrnumber = {3090262},
  mrreviewer = {Andrey\ Yu.\ Lazarev},
  doi = {10.1007/s10240-013-0054-1},
  url = {https://doi-org.scd-rproxy.u-strasbg.fr/10.1007/s10240-013-0054-1},
}

@article{CPTVV,
  author = {Calaque, Damien and Pantev, Tony and To\"en, Bertrand and Vaqui\'e, Michel and Vezzosi, Gabriele},
  title = {Shifted {P}oisson structures and deformation quantization},
  journal = {J. Topol.},
  fjournal = {Journal of Topology},
  volume = {10},
  year = {2017},
  number = {2},
  pages = {483--584},
  issn = {1753-8416,1753-8424},
  mrclass = {14A20 (14D15 14F05 18E30 18G55 53D55)},
  mrnumber = {3653319},
  mrreviewer = {Scott\ R.\ Nollet},
  doi = {10.1112/topo.12012},
  url = {https://doi-org.scd-rproxy.u-strasbg.fr/10.1112/topo.12012},
}

@incollection{ParkYou-IntroductionShiftedSymplectic,
  author = {Park, Hyeonjun and You, Jemin},
  title = {An introduction to shifted symplectic structures},
  booktitle = {Moduli spaces, virtual invariants and shifted symplectic structures},
  series = {KIAS Springer Ser. Math.},
  volume = {4},
  pages = {37--64},
  publisher = {Springer, Singapore},
  year = {[2025] \copyright 2025},
  isbn = {978-981-97-8248-2; 978-981-97-8249-9},
  mrclass = {53D30 (14A30)},
  mrnumber = {4898355},
  doi = {10.1007/978-981-97-8249-9\_2},
  url = {https://doi-org.scd-rproxy.u-strasbg.fr/10.1007/978-981-97-8249-9_2},
}

\end{document}